\documentclass[12pt, reqno]{amsart}

\usepackage{
amsmath,
amssymb,
amsthm,
mathrsfs,
amsfonts,
ytableau,
enumitem,
comment,
upgreek,
soul,
yhmath,
mathtools,
shuffle,
kotex,
array,
caption,
tabularx,
bm
}

\usepackage[linktocpage]{hyperref}

\hypersetup{
    colorlinks=true,
    linkcolor = blue,
    citecolor=magenta,
    urlcolor=cyan,
}

\usepackage[capitalise, nameinlink]{cleveref}

\crefname{equation}{}{}
\crefname{figure}{{\sc Figure}}{{\sc Figure}}

\usepackage[euler]{textgreek}

\usepackage{tikz,tikz-cd}

\ytableausetup{mathmode, boxsize=1.2em}

\newtheorem{theorem}{Theorem}[section]
\newtheorem{proposition}[theorem]{Proposition}
\newtheorem{lemma}[theorem]{Lemma}
\newtheorem{corollary}[theorem]{Corollary}

\newtheorem*{claim*}{Claim}

\theoremstyle{definition}
\newtheorem{algorithm}[theorem]{Algorithm}
\newtheorem{example}[theorem]{Example}
\newtheorem{definition}[theorem]{Definition}

\numberwithin{equation}{section} \numberwithin{figure}{section}
\numberwithin{table}{section}

\def\Z{\mathbb Z}
\def\C{\mathbb C}

\newcommand{\nc}{\newcommand}

\nc{\ra}{\rightarrow}
\nc{\la}{\leftarrow}
\nc{\SG}{\mathfrak{S}}
\nc{\ch}{\mathrm{ch}}
\nc{\id}{\mathrm{id}}
\nc{\set}{\mathrm{set}}
\nc{\comp}{\mathrm{comp}}
\nc{\setB}{\set_B}
\nc{\compB}{\comp_B}
\nc{\QSym}{\mathrm{QSym}}
\nc{\QSymB}{\QSym^B}
\nc{\Des}{\mathrm{Des}}
\nc{\Asc}{\mathrm{Asc}}
\nc{\bfgam}{\boldsymbol{\upgamma}}
\nc{\inc}{\mathrm{inc}}
\nc{\modHn}{\text{\bf mod-}H_n(0)}
\nc{\modHBn}{\text{\bf mod-}H^B_n(0)}
\nc{\modHW}{\text{\bf mod-}H_W(0)}
\nc{\GHb}{\calG_0(H_\bullet(0))}
\nc{\GHBb}{\calG_0(H^B_\bullet(0))}
\nc{\bulB}{\cdot_B}
\nc{\bracket}[1]{\bm{\lbrack} #1 \bm{\rbrack}}
\nc{\LS}{\mathsf{LS}}
\nc{\tQ}{\widetilde{Q}}
\nc{\stB}{\mathrm{st}^B}
\nc{\ab}{\mathsf{ab}}
\nc{\conc}{\mathsf{conc}}
\nc{\upper}{\mathsf{upper}}
\nc{\sign}{\mathrm{sign}}
\nc{\opi}{\overline{\pi}}
\nc{\Conv}{\mathrm{Conv}}
\nc{\symdiff}{\, \Delta \,}
\nc{\poset}{\mathsf{poset}}

\nc{\bfA}{\mathbf{A}}
\nc{\bfF}{\mathbf{F}}
\nc{\bfx}{\mathbf{x}}

\nc{\calA}{\mathcal{A}}
\nc{\calB}{\mathcal{B}}
\nc{\calC}{\mathcal{C}}
\nc{\calF}{\mathcal{F}}
\nc{\calG}{\mathcal{G}}
\nc{\calI}{\mathcal{I}}
\nc{\calL}{\mathcal{L}}
\nc{\calR}{\mathcal{R}}

\nc{\tcalF}{\widetilde{\mathcal{F}}}

\nc{\fkL}{\mathfrak{L}}

\nc{\rmst}{\mathrm{st}}

\nc{\sfA}{\mathsf{A}}
\nc{\sfM}{\mathsf{M}}
\nc{\sfP}{\mathsf{P}}
\nc{\sfR}{\mathsf{R}}
\nc{\sfd}{\mathsf{d}}
\nc{\sft}{\mathsf{t}}

\nc{\scrA}{\mathscr{A}}
\nc{\scrB}{\mathscr{B}}
\nc{\scrP}{\mathscr{P}}
\nc{\scrS}{\mathscr{S}}

\title[Poset modules of the $0$-Hecke algebras of type $B$]{Poset modules of the $0$-Hecke algebras of type $B$}

\author[Y.-H. Kim]{Young-Hun Kim}
\address{Center for Quantum structures in Modules and Spaces, Seoul National University, Seoul 08826, Republic of Korea \& Department of Mathematics, Seoul Women’s University, Seoul 01797, Republic of Korea}
\email{ykim.math@gmail.com}

\author[D. Searles]{Dominic Searles}
\address{
Department of Mathematics and Statistics, University of Otago, 730 Cumberland St., Dunedin 9016, New Zealand
}
\email{dominic.searles@otago.ac.nz}


\date{January 30, 2026}

\begin{document}
\maketitle

\begin{abstract}
In 2001, Chow developed the theory of the $B_n$ posets $P$ and the type $B$ $P$-partition enumerators $K^B_P$.  
To provide a representation-theoretic interpretation of $K^B_P$, we define the poset modules $M^B_P$ of the 0-Hecke algebra $H_n^B(0)$ of type $B$ by endowing the set of type-$B$ linear extensions of $P$ with an $H_n^B(0)$-action.
We then show that the Grothendieck group of the category associated to type-$B$ poset modules is isomorphic to the space of type $B$ quasisymmetric functions as both a $\QSym$-module and comodule, where $\QSym$ denotes the Hopf algebra of quasisymmetric functions.
Considering an equivalence relation on $B_n$ posets, where two posets are equivalent if they share the same set of type-$B$ linear extensions, we identify a natural representative of each equivalence class, which we call a distinguished poset.
We further characterize the distinguished posets whose sets of type-$B$ linear extensions form intervals in the right weak Bruhat order on the the hyperoctahedral groups. 
Finally, we discuss the relationship among the categories associated to type-$B$ weak Bruhat interval modules, $B_n$ poset modules, and finite-dimensional $H_n^B(0)$-modules.
\end{abstract}

\section{Introduction}

The theory of \emph{$P$-partitions}, for $P$ a poset on $[n] := \{1, 2, \ldots , n\}$, has significant application in algebraic combinatorics and has been widely studied. Systematically developed by Stanley \cite{Stanley72}, this theory was later used by Gessel \cite{Gessel84} to introduce the Hopf algebra $\QSym$ of \emph{quasisymmetric functions.} The generating function $K_P$ for the $P$-partitions associated to a poset $P$ is a quasisymmetric function, known as the \emph{$P$-partition enumerator}, and the functions $K_E$ for linear posets $E$ form precisely the famed basis of \emph{fundamental quasisymmetric functions} for $\QSym$. The fundamental theorem of $P$-partitions \cite{Stanley72} states that the set of $P$-partitions for a poset $P$ is the disjoint union of the sets of $E$-partitions for the linear extensions $E$ of $P$. In particular, $K_P$ is a sum of fundamental quasisymmetric functions with nonnegative coefficients. A survey of developments in the theory of $P$-partitions can be found in \cite{Gessel15}.

In \cite{DKLT96}, Duchamp–Krob–Leclerc–Thibon established a close connection between quasisymmetric functions and the representation theory of type $A$ $0$-Hecke algebras $H_n(0)$. 
Specifically, they defined an isomorphism, known as the \emph{quasisymmetric characteristic} $\ch$, from the Grothendieck group $\GHb$ associated to type $A$ $0$-Hecke algebras to $\QSym$. Under this isomorphism, the images of the simple $H_n(0)$-modules are precisely the fundamental quasisymmetric functions.  Equipping $\GHb$ with the product arising from induction of modules, $\ch$ is in fact an isomorphism of rings \cite{DKLT96}, and it was subsequently shown in \cite{BL09} that $\ch$ is an isomorphism of Hopf algebras.

Since many notable families of quasisymmetric functions have the property that their elements expand positively in fundamental quasisymmetric functions, this connection motivated research in constructing $H_n(0)$-modules whose quasisymmetric characteristics are elements of such families, including \cite{BS22, BBSSZ15, CKNO22, KY24, NSvWVW22, Searles19, TvW15, TvW19}. In the case of $P$-partition enumerators, for each poset $P$ on $[n]$ Duchamp--Hivert--Thibon \cite{DHT02} defined an $H_n(0)$-module whose underlying space is spanned by the linear extensions of $P$, and proved the quasisymmetric characteristics of these \emph{poset modules} are precisely the functions $K_P$. Moreover, they showed the induction product of two poset modules is itself a poset module. Recently, Choi--Kim--Oh \cite{CKO24} proved that restriction of a poset module is a direct sum of tensor products of poset modules, and consequently proved that the Grothendieck group associated to poset modules is isomorphic to $\QSym$ as Hopf algebras. 
The images of poset modules, their induction and restriction under certain \emph{(anti-)automorphism twist functors} \cite{Fayers05} on the category of finite-dimensional right $H_n(0)$-modules, were also determined in \cite{CKO24}.

Theory of $P$-partitions, including the fundamental theorem, was extended to finite Coxeter type 
by Reiner \cite{Reiner92, Reiner93} in terms of root systems, and to type $B$ by Chow \cite{Chow01} in terms of $B_n$ posets. 
A \emph{$B_n$ poset} \cite{Chow01} is a poset on $[-n,n] := \{-n, \ldots , -1, 0 , 1, \ldots , n\}$ satisfying the additional condition that $i\preceq j$ if and only if $-j\preceq -i$; in the type $B$ context, $P$ will denote a $B_n$ poset. Using type $B$ $P$-partition enumerators $K^B_P$, Chow defined the ring $\QSym^B$ of \emph{type $B$ quasisymmetric functions} \cite{Chow01}. The functions $K^B_E$ for linear $B_n$ posets $E$ are the \emph{type B fundamental quasisymmetric functions}. 

Subsequently, Huang \cite{Huang17} extended the notion of quasisymmetric characteristic to classical type. In type $B$, there is an action \cite{Huang17} and a coaction \cite{Chow01} of $\QSym$ on $\QSym^B$, and an action and a coaction \cite{Huang17} of $\GHb$ on the Grothendieck group $\GHBb$ associated to type $B$ $0$-Hecke algebras. 
Under the type $B$ quasisymmetric characteristic, $\GHBb$ is isomorphic to $\QSym^B$ as both $\QSym$-modules and $\QSym$-comodules, and the images of the simple modules of type $B$ $0$-Hecke algebras are the type $B$ fundamental quasisymmetric functions \cite{Huang17}.  

In this paper, we develop the theory of the $B_n$ posets and type $B$ $P$-partitions of \cite{Chow01} in terms of the type $B$ $0$-Hecke algebras $H_n^B(0)$. For each $B_n$ poset $P$, we define a right $H_n^B(0)$-module, called a \emph{$B_n$ poset module}, whose underlying space is spanned by the \emph{type $B$ linear extensions of $P$}: the linear extensions of $P$ that are themselves $B_n$ posets. We prove that the type $B$ quasisymmetric characteristic of the $B_n$ poset module corresponding to $P$ is precisely the type $B$ $P$-partition enumerator $K^B_P$.

We obtain combinatorial formulas for an analogue of induction product for a $B_m$ poset module and a poset module for $H_n(0)$, and for the restriction of a $B_n$ poset module. 
As a corollary, we show the Grothendieck group $\bigoplus_{n\ge 0}\calG_0(\scrP^B(n))$ associated to $B_n$ poset modules is isomorphic to $\QSym^B$ as $\QSym$-modules and comodules. 
Moreover, via the type $B$ quasisymmetric characteristic we obtain formulas for the $\QSym$ action of a type $A$ $P$-partition enumerator on a type $B$ $P$-partition enumerator and the $\QSym$ coaction on a type $B$ $P$-partition enumerator. This generalises formulas given by Huang \cite{Huang17} for the $\QSym$ action and coaction on $\QSym^B$ in terms of (type $B$) fundamental quasisymmetric functions.  
We then consider (anti-)automorphism twists \cite{Fayers05} of the category of finite-dimensional right $H_n^B(0)$-modules. We derive formulas for (anti)-automorphism twists of $B_n$ poset modules, and of the associated inductions and restrictions, analogous to type $A$ results of Choi--Kim--Oh \cite{CKO24}.

Linear extensions of a poset on $[n]$ can be naturally interpreted as permutations on $[n]$; similarly, type $B$ linear extensions of $B_n$ posets are naturally associated to signed permutations. While a poset on $[n]$ is uniquely determined by its linear extensions, distinct $B_n$ posets can have the same set of type $B$ linear extensions. Defining two $B_n$ posets to be equivalent if they have the same set of type $B$ linear extensions, we identify a natural representative for each equivalence class; we call these $B_n$ posets \emph{distinguished}. 
As a result, each $B_n$ poset module is isomorphic to the $B_n$ poset module for a unique distinguished poset. 

Bj\"orner--Wachs \cite{BW91} introduced the concept of \emph{regular posets} on $[n]$. They proved that a poset is regular if and only if the set of its linear extensions is an interval in the right weak Bruhat order on the symmetric group $\SG_n$, and that every interval in right weak Bruhat interval is the set of linear extensions of a regular poset.
We define regular $B_n$ posets in an analogous manner, and prove a distinguished $B_n$ poset is regular if and only if the set of its type $B$ linear extensions is an interval in the right weak Bruhat order on the hyperoctahedral group $\SG^B_n$, and that every interval in right weak Bruhat interval in $\SG^B_n$ is the set of linear extensions of a regular $B_n$ poset.
It was noted in \cite{CKO24} that the modules for regular posets on $[n]$ are instances of the \emph{weak Bruhat interval modules} of \cite{JKLO22}; similarly, the $B_n$ poset modules for regular $B_n$ posets are isomorphic to type $B$ weak Bruhat interval modules. 
Finally we show that the Grothendieck group associated to modules for regular $B_n$ posets (or equivalently, weak Bruhat interval modules in type $B$) is isomorphic to $\QSym^B$, and thus also isomorphic to $\bigoplus_{n\ge 0}\calG_0(\scrP^B(n))$, as $\QSym$-modules and $\QSym$-comodules. This establishes a type $B$ analogue of a result of \cite{JKLO22} in type $A$. 

This paper is organized as follows. In \cref{Sec: Preliminaries}, we provide necessary background on quasisymmetric functions of types $A$ and $B$, $0$-Hecke algebras of types $A$ and $B$, poset modules for $0$-Hecke algebras of type $A$, and posets and $P$-partitions of type $B$. 
In \cref{Sec: poset modules}, we define $B_n$ poset modules and prove their type $B$ quasisymmetric characteristics are type $B$ $P$-partition enumerators. We then give formulas for induction and restriction of $B_n$ poset modules, prove the Grothendieck group associated to $B_n$ poset modules is isomorphic to $\QSym^B$, and determine the behavior of $B_n$ poset modules under (anti)-automorphism twists. 
In \cref{Sec: regular}, we define distinguished $B_n$ posets and regular $B_n$ posets, prove that intervals in right weak Bruhat order on $\SG^B_n$ are characterized by regular $B_n$ posets, and prove the Grothendieck group associated to modules for regular $B_n$ posets is isomorphic to $\QSym^B$.

\section{Preliminaries}\label{Sec: Preliminaries}

For any $m,n \in \Z$, we set $[m,n] := \{x \in \Z \mid m \le x \le n\}$ and $[n] := [1,n]$.
Unless otherwise stated, $m$ and $n$ denote positive integers.

\subsection{Coxeter groups}

Given two letters $\gamma_1, \gamma_2$ and a nonnegative integer $m$, let $[\gamma_1 \mid \gamma_2]_m$ denote the word consisting of $m$ letters alternating between $\gamma_1$ and $\gamma_2$ and finishing with $\gamma_2$.
For example, $[\gamma_1 \mid \gamma_2]_7 = \gamma_2 \gamma_1 \gamma_2 \gamma_1 \gamma_2 \gamma_1 \gamma_2$.

A \emph{Coxeter system} is a pair $(W,S)$ where $W$ is a group and $S$ is a generating set for $W$ such that each $s\in S$ is an involution, and for distinct $s,t \in S$ 
there exists $m(s,t) = m(t,s) \in \{2,3, \ldots \} \cup \{\infty\}$ such that $[s \mid t]_{m(s,t)} = [t \mid s]_{m_{(s,t)}}$ if $m_{s,t}\in \Z$, and there is no relation between $s$ and $t$ otherwise. 
The group $W$ is called a \emph{Coxeter group}. The Coxeter system $(W,S)$ is said to be finite if $W$ is finite.

Given $w \in W$, a \emph{reduced word for $w$} is a minimum-length word over the alphabet $S$ that represents $w$; the number of letters in a reduced word for $w$ is called the \emph{length} of $w$ and is denoted by $\ell(w)$.
Let
\begin{align*}
\Des_R(w) & := \{s \in S \mid \ell(w s) < \ell(w)\} \quad \text{and} \quad
\Asc_R(w) := \{s \in S \mid \ell(w s) > \ell(w)\}.
\end{align*}
Each element in $\Des_R(w)$ (respectively, $\Asc_R(w)$) is called a (right) \emph{descent} of $w$ (respectively, a (right) \emph{ascent} of $w$).
If $W$ is finite, we denote by $w_0$ the unique longest element of $W$.

The \emph{right weak Bruhat order $\preceq_R$} on $W$ is defined to be the partial order whose covering relations $\preceq^c_R$ are given as follows:
\[
w \preceq^c_R w s
\quad \text{if and only if} \quad
s \notin \Des_R(w).
\]
Given $u, w \in W$ such that $u \preceq_R w$, we define
\[
[u, w]_R := \{v \in W \mid u \preceq_R v \preceq_R w\}.
\]
A subset $U \subseteq W$ is called a \emph{right weak Bruhat interval} if $U = [u, w]_R$ for some $u, w \in U$.

The Coxeter group of type $A_{n-1}$ is the symmetric group $\SG_n$ of permutations on $n$ letters.
For $i \in [n-1]$, let $s_i$ be the simple transposition $(i, i+1)$.
For $w \in \SG_n$,
\[
\Des_R(w) = \{s_i \mid i \in [n-1] \text{ with } w(i) > w(i+1)\}.
\]

The Coxeter group of type $B_n$ is the group $\SG^B_n$ of all \emph{signed permutations of $[n]$}, that is, permutations $\sigma$ of $[-n,n]\setminus\{0\}$ such that $\sigma(-i) = -\sigma(i)$ for all $i \in [-n,n]\setminus\{0\}$. 
For convenience, throughout this paper we regard a signed permutation $\sigma$ of $[n]$ as a permutation on $[-n,n]$ by setting $\sigma(0)=0$. 
For $\sigma \in \SG^B_n$, if $\sigma(i) = w_i$ for $i = 1,2,\ldots, n$, we write $\sigma = \bracket{w_1, w_2,\ldots, w_n}$; this is called the \emph{window notation for $\sigma$}.
For $i \in [n-1]$, let $s^B_i := \bracket{1,\ldots, i-1, i+1, i, i+2, \ldots, n}$ and 
$s^B_0 := \bracket{-1, 2, 3, \ldots, n}$.
It is well-known that $\SG^B_n$ is generated by $s^B_0, s^B_1, s^B_2, \ldots, s^B_{n-1}$ subject to the relations
\begin{align*}
\begin{array}{ll}
(s^B_i)^2 = \id &\text{for $0 \le i \le n-1$}, \\[1ex]
s^B_{i} s^B_{i+1} s^B_{i} = s^B_{i+1} s^B_{i} s^B_{i+1} & \text{for $1 \le i < n-1$}, \\[1ex]
s^B_i s^B_j = s^B_j s^B_i & \text{if $|i-j| \ge 2$}, \\[1ex]
s^B_0 s^B_i = s^B_i s^B_0 & \text{if $i \ge 2$}, \\[1ex]
s^B_0 s^B_1 s^B_0 s^B_1 = s^B_1 s^B_0 s^B_1 s^B_0.
\end{array}
\end{align*}

For $w \in \SG^B_n$, by \cite[Proposition 8.1.2]{BB05} we have
\begin{equation}\label{eq: Descent for type B}
\Des_R(w) = \{s^B_i \mid i \in [0,n-1] \text{ with } w(i) > w(i+1)\}.
\end{equation}

\subsection{Quasisymmetric functions and their type $B$ analogues}
\label{subsec: QSym}

A \emph{composition} $\alpha$ of $n$, denoted by $\alpha \models n$, is a finite ordered list of positive integers $(\alpha_1, \alpha_2, \ldots, \alpha_k)$ satisfying $\sum_{i=1}^k \alpha_i = n$.
We call $k =: \ell(\alpha)$ the \emph{length} of $\alpha$ and $n =:|\alpha|$ the \emph{size} of $\alpha$. For convenience, we define the empty composition $\emptyset$ to be the unique composition of size and length $0$.
Given $\alpha = (\alpha_1, \alpha_2, \ldots,\alpha_{\ell(\alpha)}) \models n$ and $I = \{i_1 < i_2 < \cdots < i_l\} \subset [1, n-1]$, 
let 
\begin{align*}
&\set(\alpha) := \{\alpha_1,\alpha_1+\alpha_2,\ldots, \alpha_1 + \alpha_2 + \cdots + \alpha_{\ell(\alpha)-1}\}, \\
&\comp(I) := (i_1,i_2 - i_1,\ldots,n-i_l).
\end{align*}
The set of compositions of $n$ is in bijection with the set of subsets of $[n-1]$ under the correspondence $\alpha \mapsto \set(\alpha)$ (or $I \mapsto \comp(I)$).

For compositions $\alpha = (\alpha_{1}, \alpha_{2}, \ldots, \alpha_{\ell(\alpha)})$ and $\beta= (\beta_{1}, \beta_{2}, \ldots, \beta_{\ell(\beta)})$,
let $\alpha \cdot \beta$ denote the \emph{concatenation}
and $\alpha \odot \beta$ the \emph{near concatenation} of $\alpha$ and $\beta$. 
That is, 
$ \alpha \cdot \beta = (\alpha_1, \alpha_2, \ldots, \alpha_{\ell(\alpha)}, \beta_1, \beta_2, \ldots, \beta_{\ell(\beta)})$ and 
$\alpha \odot \beta = (\alpha_1,\alpha_2,\ldots, \alpha_{k-1},\alpha_{\ell(\alpha)} + 
\beta_1,\beta_2, \ldots, \beta_{\ell(\beta)})$.

Quasisymmetric functions are power series of bounded degree in variables $x_{1},x_{2},x_{3},\ldots$  with coefficients in $\Z$, such that the coefficient of the monomial $x_{1}^{\alpha _{1}}x_{2}^{\alpha _{2}}\cdots x_{k}^{\alpha _{k}}$ is equal to the coefficient of the monomial $x_{i_{1}}^{\alpha _{1}}x_{i_{2}}^{\alpha _{2}}\cdots x_{i_{k}}^{\alpha _{k}}$, for any strictly increasing sequence of positive integers $i_{1} < i_{2} < \cdots < i_{k}$ indexing the variables and any positive integer sequence $(\alpha _{1},\alpha _{2},\ldots,\alpha _{k})$ of exponents.
The ring $\QSym$ of quasisymmetric functions is a graded $\Z$-algebra, decomposing as
\[
\QSym =\bigoplus _{n\geq 0} \QSym_n,
\]
where $\QSym_n$ is the $\Z$-module consisting of all quasisymmetric functions that are homogeneous of degree $n$.

Given a composition $\alpha = (\alpha_1, \alpha_2, \ldots, \alpha_k)$ of $n$, the \emph{monomial quasisymmetric function} $M_\alpha$ is defined by 
\[
M_\alpha = \sum_{i_1 < i_2 < \cdots < i_k} x_{i_1}^{\alpha_1} x_{i_2}^{\alpha_2}
\cdots
x_{i_k}^{\alpha_k}
\]
and the \emph{fundamental quasisymmetric function} $F_\alpha$ is defined by 
\[
F_\alpha = \sum_{\substack{1 \le i_1 \le \cdots \le i_n \\ i_j < i_{j+1} \text{ if } j \in \set(\alpha)}} x_{i_1} \cdots x_{i_n}.
\]
At times, we will index fundamental quasisymmetric functions by subsets of $[n-1]$ instead of compositions. For $I\subseteq [n-1]$, we have
\[
F_I = \sum_{\substack{1 \le i_1 \le \cdots \le i_n \\ i_j < i_{j+1} \text{ if } j \in I}} x_{i_1} \cdots x_{i_n}.
\]

For every nonnegative integer $n$, it is known that both $\{F_\alpha \mid \alpha \models n\}$ and $\{M_\alpha \mid \alpha \models n\}$ are $\Z$-bases for $\QSym_n$.
Define a $\Z$-linear map 
\[
\Delta: \QSym \ra \QSym \otimes \QSym, \quad
F_\alpha \mapsto \sum_{\beta \cdot \gamma = \alpha \text{ or } \beta \odot \gamma = \alpha} F_\beta \otimes F_\gamma.
\]
It is well known that $\QSym$ is not only a ring, but also a Hopf algebra.  
For more information, see \cite[Section 5.1]{GR20}.

A \emph{type-$B$ composition} $\alpha$ of a nonnegative integer $n$, denoted by $\alpha \models_B n$, is a tuple $\alpha = (\alpha_1, \alpha_2, \ldots, \alpha_k)$ of integers such that $\alpha_1 \ge 0$, $\alpha_i >0$ for $i > 1$, and $\sum_{i = 1}^k \alpha_i = n$.
We call $k$ the \emph{length} of $\alpha$ and denote it by $\ell(\alpha)$.
Given $\alpha \models_B n$, define 
\[\setB(\alpha) = \{\alpha_1, \alpha_1+\alpha_2, \ldots, \alpha_1 + \alpha_2 + \cdots + \alpha_{\ell(\alpha)-1}\}.\]
Conversely, given $I = \{i_1 < i_2 < \cdots < i_p\} \subseteq [0,n-1]$, define
\[
\compB(I) = (i_1, i_2 - i_1, \ldots, i_k - i_{k-1}, n - i_k).
\]
One can easily see that the set of type-$B$ compositions of $n$ is in bijection with the set of subsets of $[0,n-1]$ under the correspondence $\alpha \mapsto \set_B(\alpha)$ (or $I \mapsto \compB(I)$).
For $\alpha, \beta \models_B n$, if $\set_B(\beta) \subseteq \set_B(\alpha)$, then we say that \emph{$\alpha$ refines $\beta$}, denoted $\alpha \preceq_B \beta$.

Let $X_{\ge 0 } = \{x_0, x_1, x_2, \ldots\}$ be a set of commutative variables.
Below, we present the type $B$ analogue of $\QSym$, which was first introduced by Chow \cite{Chow01}.
Given $\alpha \models_B n$, 
the \emph{type-$B$ monomial quasisymmetric function} $M^B_\alpha$ is defined by
\[
M^B_\alpha = \sum_{0 < i_2 < i_3 < \cdots < i_n} x_0^{\alpha_1} x_{i_2}^{\alpha_2} x_{i_3}^{\alpha_3} \cdots x_{i_n}^{\alpha_n}.
\]
For $n \in \Z_{\ge 0}$, let
\[
\QSymB_n := \Z \{M^B_\alpha \mid \alpha \models_B n\}
\quad \text{and} \quad 
\QSymB := \bigoplus_{n \ge 0} \QSymB_n.
\]
We call elements in $\QSymB$ \emph{type-$B$ quasisymmetric functions.}

The \emph{type-$B$ fundamental quasisymmetric function} $F^B_\alpha$ is defined by
\[
F^B_\alpha = \sum_{\substack{0 = i_0 \le i_1 \le \cdots \le i_n \\ \text{$i_j < i_{j+1}$ if $j \in \setB(\alpha)$}}} x_{i_1} x_{i_2} \cdots x_{i_n}.
\]
As in type $A$, we will at times index type-$B$ fundamental quasisymmetric functions by subsets of $[0,n-1]$ instead of type-$B$ compositions. For $I\subseteq [0,n-1]$, we have
\[
F^B_I = \sum_{\substack{0 = i_0 \le i_1 \le \cdots \le i_n \\ \text{$i_j < i_{j+1}$ if $j \in I$}}} x_{i_1} x_{i_2} \cdots x_{i_n}.
\]

In \cite[Lemma 2.2.5]{Chow01}, it was shown that for any $\alpha \models_B n$,
\[
F^B_\alpha = \sum_{\beta \preceq_B \alpha} M^B_\beta
\quad \text{and} \quad 
M^B_\alpha = \sum_{\beta \preceq_B \alpha} (-1)^{\ell(\alpha) - \ell(\beta)} F^B_\beta.
\]
In particular, $\{F^B_\alpha \mid \alpha \models_B n\}$ is also a $\Z$-basis for $\QSym^B_n$.

Let $X_{\ge 0} + Y_{>0} = \{x_0,x_1,\ldots\} \cup \{y_1,y_2, \ldots\}$ be a set of commuting variables.
Define a $\Z$-linear map 
\[
\updelta^B: \QSymB \ra \QSymB \otimes \QSym , \quad 
f(X_{\ge 0}) \mapsto f(X_{\ge 0} + Y_{>0}) \quad \text{for $f \in \QSymB$.}
\]
In \cite[Remark 3.2.4]{Chow01}, it was implicitly mentioned that $\updelta^B$ is a right $\QSym$-coaction on $\QSymB$.
In \cite[p.857]{Huang16}, Huang provided the formula
\begin{align}\label{eq: coaction on fund B}
\updelta^B(F^B_\alpha) = \sum_{0 \le i \le n} F^B_{\alpha_{\le i}} \otimes F_{\alpha_{>i}} \quad \text{for all $\alpha \models_B n$.}
\end{align}
Here,
\[
\alpha_{\le i} := (1) \ast_1 (1) \ast_2 \cdots \ast_{i-1} (1)
\quad \text{and} \quad
\alpha_{> i} := (1) \ast_{i+1} (1) \ast_{i+2} \cdots \ast_{n-1} (1),
\]
where $(\ast_1, \ast_2, \ldots, \ast_{n-1}) \in \{\cdot, \odot\}^{n-1}$ such that
$\alpha = (1) \ast_1 (1) \ast_2 \cdots \ast_{n-1} (1)$.

In addition, Huang~\cite[p.415]{Huang17} defined a right $\QSym$-action on $\QSym^B$.
We need some notation to describe this action.

Let $w = w_1, w_2, \ldots, w_n \in \Z^n$ be a word.
\begin{enumerate}[label = $\bullet$]
\item 
For $A \subseteq [n]$, let $w|_A$ be the subword of $w$ consisting of the letters whose absolute values belong to $A$. 

\item 
Let $\mathrm{Neg}(w) := \{i \in [n] \mid w_i < 0\}$. 
Enumerate the sets $\mathrm{Neg}(w)$ and $[n] \setminus \mathrm{Neg}(w)$ in increasing order as $i_1 < i_2 < \cdots i_k$ and $j_1 < j_2 < \cdots < j_{n-k}$, respectively.
Define 
$$
\widehat{w} := -w_{i_k}, -w_{i_{k-1}}, \ldots, -w_{i_1}, w_{j_1}, w_{j_2}, \ldots, w_{j_{n-k}}.
$$

\item
Let $\rmst(w)$ be the unique permutation in $\SG_n$ such that
\[
\rmst(w)(i) < \rmst(w)(j) 
\quad \text{if and only if} \quad
w_i \le w_j \quad \text{for $1 \le i < j \le n$.}
\]
\end{enumerate}
Given $u \in \SG^B_m$ and $v \in \SG_n$, define 
\[
u \shuffle^B v := \left\{
w \in \SG^B_{m+n} \; \middle| \; w|_{[m]} = u \text{ and } \rmst(\widehat{w}|_{[m+1,m+n]}) = v
\right\}.
\]
Then the right $\QSym$-action $\odot^B$ on $\QSym^B$ is defined as follows:
for $I \subseteq [0,m-1]$ and $J \subseteq [n-1]$,
\begin{align}\label{eq: odot action on QSym^B}
F^B_I \odot^B F_J :=
\sum_{w \in u \, \shuffle^B \, v} F^B_{\Des_R(w)},
\end{align}
where $u \in \SG^B_m$ and $v \in \SG_n$ such that $\Des_R(u) = I$ and $\Des_R(v) = J$, respectively.

\subsection{$0$-Hecke algebras and their representation theory}

The $0$-Hecke algebra $H_W(0)$ of a Coxeter system $(W,S)$ is the associative $\C$-algebra generated by $\{\opi_s \mid s \in S\}$ subject to the relations
\[
\opi_s^2 = -\opi_s 
\quad \text{and}\quad
[\opi_s \mid \opi_t]_{m(s,t)} = [\opi_t \mid \opi_s]_{m(s,t)}
\]
for all distinct $s,t \in S$. The algebra $H_W(0)$ is also generated by $\{\pi_s \mid s\in S\}$, where $\pi_s = \opi_s+1$. This generating set has relations $\pi_s^2 = \pi_s$ and $[\pi_s \mid \pi_t]_{m(s,t)} = [\pi_t \mid \pi_s]_{m(s,t)}$. We will typically use the generators $\opi_s$, however we will also consider the generators $\pi_s$ in the context of (anti-)automorphism twists in \cref{subsec:twists}.

In this work, we will focus on the $0$-Hecke algebra $H_n(0)$ of type $A_{n-1}$ and the $0$-Hecke algebra $H^B_n(0)$ of type $B_n$.
For $i \in [n-1]$, let $\opi_i := \opi_{s_i}$ denote a generator of $H_n(0)$, and for $i \in [0, n-1]$, let $\opi^B_i := \opi_{s^B_i}$ denote a generator of $H^B_n(0)$.

When $(W,S)$ is a finite Coxeter system of rank $n$, Norton \cite{Norton79} showed 
the irreducible $H_W(0)$-modules are all one-dimensional, and they are in bijection with the subsets of $S$.
For $I \subseteq S$, let $\bfF^W_I$ be the 
one-dimensional $\C$-vector space spanned by $v_I$ and endow it with the right $H_W(0)$-action as follows: for each $s \in S$,
\[
v_I \cdot \opi_s = \begin{cases}
-v_I & \text{if $s \in I$,} \\
0 & \text{if $s \notin I$.}
\end{cases}
\]
This module is the irreducible $H_W(0)$-module corresponding to $I$.
For convenience, for $\alpha \models n$ (respectively, $\alpha \models_B n$), let $\bfF_\alpha$ (respectively, $\bfF^B_\alpha$) be the irreducible $H_n(0)$-module ($H^B_n(0)$-module) corresponding to $\{s_i \mid i \in \set(\alpha)\}$ 
(respectively, $\{s^B_i \mid i \in \setB(\alpha)\}$).

In \cite[Section 3]{DS25}, Defant--Searles defined a notion of \emph{ascent-compatibility} for subsets $X$ of $W$ in terms of left ascents of elements of $X$, extending the type $A$ notion of \cite{Searles25}, and proved that ascent-compatible subsets give rise to a natural left $0$-Hecke module structure on $\C X$. We now state the analogue of ascent-compatibility for right ascents and right $0$-Hecke modules.

For $u, v \in W$ and $s,t \in S$, the quadruple $(u,v,s,t)$ is said to be \emph{aligned} if $s \in \Asc_R(u)$, $t \in  \Asc_R(v)$, and $u s u^{-1} = v t v^{-1}$. 
We define a subset $X$ of $W$ to be \emph{ascent-compatible} if for all $u,v \in X$ and all $s,t \in S$ such that $(u,v,s,t)$ is aligned, we have $us \in X$ if and only if $vt \in X$.

For $s \in S$, define a linear operator $\overline{\uppi}_s: \C X \ra \C X$ by 
\[
\overline{\uppi}_s(x) = \begin{cases}
-x & \text{if $s \in \Des_R(x)$}, \\
0 & \text{if $s \notin \Des_R(x)$ and $xs \notin X$}, \\
xs & \text{if $s \notin \Des_R(x)$ and $xs \in X$}
\end{cases}
\]
for all $x \in X$ and extending by linearity. The following result is the analogue of \cite[Theorem 3.1]{DS25} for right $0$-Hecke modules, and is proved in an essentially identical manner.

\begin{theorem}{{\rm (cf.} \cite[Theorem 3.1]{DS25}{\rm)}}\label{theorem:ascentcompatible}
\label{thm: ascent-comp and 0-Hecke action}
Let $(W,S)$ be a Coxeter system.
If $X$ is an ascent-compatible subset of $W$, then the linear operators $\overline{\uppi}_s$ define a right action of the $0$-Hecke algebra $H_W(0)$ on $\C X$.
\end{theorem}

Let $\modHW$ denote the category of finite-dimensional right $H_W(0)$-modules.
The \emph{Grothendieck group} of $\modHW$, denoted by $\calG_0(\modHW)$, is the $\Z$-span of the isomorphism classes of the finite dimensional right $H_W(0)$-modules, modulo the relation $[M] = [M'] + [M'']$ whenever there exists a short exact sequence $0 \ra M' \ra M \ra M'' \ra 0$.
Let
\[
\GHb := \bigoplus_{n \ge 0} \calG_0(\modHn)
\quad \text{and} \quad 
\GHBb := \bigoplus_{n \ge 0} \calG_0(\modHBn).
\]

Given $m,n \in \Z_{\ge 0}$, let us view $H_m(0) \otimes H_n(0)$ as the subalgebra of $H_{m+n}(0)$ generated by $\{\opi_i \mid i \in [m+n-1] \setminus \{m\} \}$.
For an $H_m(0)$-module $M$ and $H_n(0)$-module $N$, we define
\begin{align}\label{eq: product and coproduct type A}
[M] \boxtimes [N] = 
\left[
M \otimes N \uparrow_{H_m(0) \otimes H_n(0)}^{H_{m+n}(0)}
\right]
\quad \text{and} \quad
\Delta([M]) := \sum_{0 \le k \le m} 
\left[
M \downarrow_{H_k(0) \otimes H_{m-k}(0)}^{H_{m}(0)}
\right]. 
\end{align}
It was shown in \cite{BL09} that $\GHb$ has a Hopf algebra structure with the product $\boxtimes$ and coproduct $\Delta$.

Given $m,n \in \Z_{\ge 0}$, let us view $H^B_m(0) \otimes H_n(0)$ as the subalgebra of $H^B_{m+n}(0)$ generated by 
$\{\opi^B_i \mid i \in [0,m+n-1] \setminus \{m\}\}$.
For an $H^B_m(0)$-module $M$ and $H_n(0)$-module $N$, we define 
\begin{align*}
[M] \boxtimes^B [N] = 
\left[
M \otimes N \uparrow_{H^B_m(0) \otimes H_n(0)}^{H^B_{m+n}(0)}
\right]
\quad \text{and} \quad
\updelta^B([M]) := \sum_{0 \le k \le m} 
\left[
M \downarrow_{H^B_k(0) \otimes H_{m-k}(0)}^{H^B_{m}(0)}
\right]. 
\end{align*}
Here, we use the notation $\boxtimes^B$ due to analogy with the type $A$ definition, but it should be noted that $\boxtimes^B$ does not define a product on $\GHBb$, since $[M] \in \GHBb$ but $[N] \in \GHb$. In \cite[Theorem 5(i)]{Huang17}, Huang proved that $\boxtimes^B$ defines a right $\GHb$-action on $\GHBb$, and that $\updelta^B$ defines a right $\GHb$-coaction on $\GHBb$.

In \cite{DKLT96, KT97}, the ring isomorphism
\begin{align}\label{eq: quasi characteristic}
\ch: \GHb \ra \QSym, \quad [\bfF_{\alpha}] \mapsto F_{\alpha} \quad \text{for $\alpha \models n$ and $n \in \Z_{\ge 0}$},
\end{align}
called the \emph{quasisymmetric characteristic}, was introduced.
In fact, $\ch$ is not only a ring isomorphism, but also a Hopf algebra isomorphism \cite{BL09}. 

Huang \cite{Huang17} defined the $\Z$-linear map 
\begin{align}\label{eq: quasi characteristic of type B}
\ch^B: \GHBb \ra \QSymB, \quad [\bfF^B_\alpha] \mapsto F^B_\alpha \quad \text{for $\alpha \models_B n$ and $n \in \Z_{\ge 0}$}
\end{align}
and showed in \cite[Theorem 5(iv)]{Huang17} that $\ch^B$ is not only a $\QSym$-module isomorphism but also a $\QSym$-comodule isomorphism.

The following result is the analogue of \cite[Theorem 3.5]{DS25} for right $0$-Hecke modules, and is proved similarly. It gives the type-$B$ quasisymmetric characteristics of the modules described in \cref{theorem:ascentcompatible}.

\begin{theorem}{{\rm (cf.} \cite[Theorem 3.5]{DS25}{\rm)}}\label{theorem:characteristic}
If $X$ is an ascent-compatible subset of $\SG^B_n$, then
\[\ch^B([\C X]) = \sum_{x\in X}F^B_{\Des_R(x)}.\]
\end{theorem}

\subsection{Poset modules for $0$-Hecke algebras of type $A$}
Let $P$ be a poset on $[n]$.
A map $f: [n] \ra \Z_{\ge 0}$ is called a \emph{$P$-partition} if it satisfies the following conditions:
\begin{enumerate}[label = {\rm (\arabic*)}]
\item 
If $i \preceq_P j$, then $f(i) \le f(j)$.
\item 
If $i \preceq_P j$ and $i > j$, then $f(i) < f(j)$.
\end{enumerate}
The \emph{$P$-partition generating function} is the (type $A$) quasisymmetric function
\[
K_{P} := \sum_{f:  \text{$P$-partition}} x_{f(1)} x_{f(2)} \cdots x_{f(n)}.
\]

Modules of the type-$A$ $0$-Hecke algebra providing a representation-theoretic interpretation of $P$-partition generating functions were introduced in \cite{DHT02}.
Let
$\calL(P)$ be the set of all linear extensions of $P$ and
\[
\Sigma_{R}(P) := \{\gamma \in \SG_n \mid 
\gamma(1) \preceq_E \gamma(2) \preceq_E \cdots \preceq_E \gamma(n) \ \text{for some $E \in \calL(P)$}\}.
\]

\begin{definition}{\rm (\cite[Definition 3.18]{DHT02})}
\label{def: poset module}
Let $P$ be a poset on $[n]$. 
The \emph{poset module associated with $P$} is the right $H_n(0)$-module $M_P$ with $\C \Sigma_{R}(P)$ as the underlying space and with the $H_n(0)$-action defined by
\begin{align}\label{left action}
\gamma \cdot \opi_{i} :=
\begin{cases}
-\gamma & \text{if $i \in \Des_R(\gamma)$}, \\
0 & \text{if $i \notin \Des_R(\gamma)$ and $\gamma s_i \notin \Sigma_{R}(P)$,} \\
\gamma s_i & \text{if $i \notin \Des_R(\gamma)$ and $\gamma s_i \in \Sigma_{R}(P)$}
\end{cases} 
\end{align}
for any  $i \in [n-1]$ and $\gamma \in \Sigma_R(P)$.
\end{definition}

Given $P_1 = ([m], \preceq_{P_1})$ and $P_2 = ([n], \preceq_{P_2})$, define the poset $P_1 \sqcup P_2 := ([m+n], \preceq_{P_1 \sqcup P_2})$ by 
\[
i \prec_{P_1 \sqcup P_2} j \quad \text{if and only if} \quad i \prec_{P_1} j \ \ \text{or} \ \  i-m \prec_{P_2} j-m.
\]

\begin{proposition}{\rm (\cite[Proposition 3.21]{DHT02})}\label{prop: DHT proposition}
Let $m,n$ be positive integers. 
Then we have the following.
\begin{enumerate}[label = {\rm (\arabic*)}]
\item Let $P$ be a poset on $[n]$. Then $\ch([M_P]) = K_P$.

\item Let $P_1$ be a poset on $[m]$ and $P_2$ a poset on $[n]$. 
Then $M_{P_1 \sqcup P_2} \cong M_{P_1} \boxtimes M_{P_2}$, and thus  $\ch([M_{P_1 \sqcup P_2}]) = K_{P_1} K_{P_2} = \ch([M_{P_1}]) \ch([M_{P_2}])$.
\end{enumerate}
\end{proposition}

Let $\scrP(n)$ be the full subcategory of $\modHn$ whose objects are direct sums of finitely many isomorphic copies of poset modules.
Recently, Choi--Kim--Oh \cite{CKO24} showed that $\bigoplus_{n \ge 0} \calG_0(\scrP(n))$ inherits essential structure, such as product and coproduct, from $\GHb$.
As a result, they showed that 
\begin{align}\label{eq: G0 poset modules isom to QSym}
\bigoplus_{n \ge 0} \calG_0(\scrP(n)) \cong \QSym \quad \text{as Hopf algebras.}
\end{align}
For further details, see \cite[Section 3]{CKO24}.

\subsection{Posets and $P$-partitions of type $B$}
In this work, we will focus on the type $B$ analogues of posets on $[n]$ and $P$-partitions introduced and studied by Chow in \cite{Chow01}.

\begin{definition}{\rm (\cite[Definition 2.1.2]{Chow01})}
A poset $P = ([-n,n], \preceq_P)$ is called a \emph{$B_n$ poset} if, for any $i,j\in [-n,n]$, $i \preceq_P j$ if and only if $-j \preceq_P -i$.
\end{definition}

\begin{example}\label{eg: Bn posets}
(1) The posets 
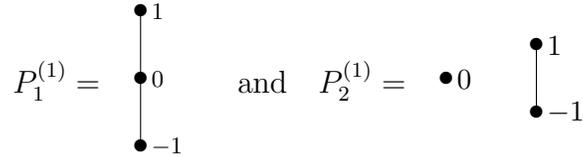
\[
\def \pp {0.6}
P^{(1)}_1 =
\begin{array}{l}
\begin{tikzpicture}
\node at (0*\pp, 1.5*\pp) {$\bullet$};
\node[right] at (0*\pp, 1.5*\pp) {\scriptsize $1$};
\node at (0*\pp, 0*\pp) {$\bullet$};
\node[right] at (0*\pp, 0*\pp) {\scriptsize $0$};
\node at (0*\pp, -1.5*\pp) {$\bullet$};
\node[right] at (0*\pp, -1.5*\pp) {\scriptsize $-1$};

\draw (0*\pp, 1.5*\pp) -- (0*\pp, 0*\pp);
\draw (0*\pp, 0*\pp) -- (0*\pp, -1.5*\pp);
\end{tikzpicture}
\end{array}
\quad \text{and} \quad
P^{(1)}_2 =
\begin{array}{l}
\begin{tikzpicture}
\node at (-2*\pp, 0*\pp) {$\bullet$};
\node[right] at (-2*\pp, 0*\pp) {\small $0$};
\node at (0*\pp, 0.75*\pp) {$\bullet$};
\node[right] at (0*\pp, 0.75*\pp) {\small $1$};
\node at (0*\pp, -0.75*\pp) {$\bullet$};
\node[right] at (0*\pp, -0.75*\pp) {\small $-1$};

\draw (0*\pp, -0.75*\pp) -- (0*\pp, 0.75*\pp);
\end{tikzpicture}
\end{array}
\]
are $B_1$ posets.
\smallskip

\noindent 
(2) The posets
\[
\def \pp {0.6}
P^{(2)}_1 =
\begin{array}{l}
\begin{tikzpicture}
\node at (0*\pp, 2*\pp) {$\bullet$};
\node[right] at (0*\pp, 2*\pp) {\scriptsize $2$};
\node at (0*\pp, 1*\pp) {$\bullet$};
\node[right] at (0*\pp, 1*\pp) {\scriptsize $-1$};
\node at (0*\pp, 0*\pp) {$\bullet$};
\node[right] at (0*\pp, 0*\pp) {\scriptsize $0$};
\node at (0*\pp, -1*\pp) {$\bullet$};
\node[right] at (0*\pp, -1*\pp) {\scriptsize $1$};
\node at (0*\pp, -2*\pp) {$\bullet$};
\node[right] at (0*\pp, -2*\pp) {\scriptsize $-2$};

\draw (0*\pp, 1*\pp) -- (0*\pp, 2*\pp);
\draw (0*\pp, 0*\pp) -- (0*\pp, 1*\pp);
\draw (0*\pp, -1*\pp) -- (0*\pp, 0*\pp);
\draw (0*\pp, -2*\pp) -- (0*\pp, -1*\pp);
\end{tikzpicture}
\end{array}, 
\quad
P^{(2)}_2 =
\begin{array}{l}
\begin{tikzpicture}
\node at (0*\pp, 1.5*\pp) {$\bullet$};
\node[right] at (0*\pp, 1.65*\pp) {\scriptsize $2$};
\node at (0*\pp, 0*\pp) {$\bullet$};
\node[right] at (0*\pp, 0*\pp) {\scriptsize $0$};
\node at (0*\pp, -1.5*\pp) {$\bullet$};
\node[right] at (0*\pp, -1.5*\pp) {\scriptsize $-2$};
\node at (1.5*\pp, 0.75*\pp) {$\bullet$};
\node[right] at (1.5*\pp, 0.75*\pp) {\scriptsize $-1$};
\node at (1.5*\pp, -0.75*\pp) {$\bullet$};
\node[right] at (1.5*\pp, -0.75*\pp) {\scriptsize $1$};

\draw (0*\pp, 0*\pp) -- (0*\pp, 1.5*\pp);
\draw (0*\pp, -1.5*\pp) -- (0*\pp, 0*\pp);

\draw (0*\pp, 1.5*\pp) -- (1.5*\pp, 0.75*\pp);
\draw (1.5*\pp, 0.75*\pp) -- (1.5*\pp, -0.75*\pp);
\draw (0*\pp, -1.5*\pp) -- (1.5*\pp, -0.75*\pp);
\end{tikzpicture}
\end{array}
\quad \text{and} \quad
P^{(2)}_3 =
\begin{array}{l}
\begin{tikzpicture}
\node at (-1*\pp, 1.5*\pp) {$\bullet$};
\node[right] at (-1*\pp, 1.5*\pp) {\scriptsize $-1$};
\node at (1*\pp, 1.5*\pp) {$\bullet$};
\node[right] at (1*\pp, 1.5*\pp) {\scriptsize $2$};
\node at (0*\pp, 0*\pp) {$\bullet$};
\node[right] at (0*\pp, 0*\pp) {\scriptsize $0$};
\node at (-1*\pp, -1.5*\pp) {$\bullet$};
\node[right] at (-1*\pp, -1.5*\pp) {\scriptsize $1$};
\node at (1*\pp, -1.5*\pp) {$\bullet$};
\node[right] at (1*\pp, -1.5*\pp) {\scriptsize $-2$};

\draw (0*\pp, 0*\pp) -- (-1*\pp, 1.5*\pp);
\draw (0*\pp, 0*\pp) -- (1*\pp, 1.5*\pp);
\draw (0*\pp, 0*\pp) -- (-1*\pp, -1.5*\pp);
\draw (0*\pp, 0*\pp) -- (1*\pp, -1.5*\pp);
\end{tikzpicture}
\end{array}
\]
are $B_2$ posets.
\end{example}

\begin{definition}{\rm (\cite[Definition 2.1.2]{Chow01})}\label{def: P-partitions type B}
Let $P = ([-n,n], \preceq_P)$ be a $B_n$ poset.
A map $f: [-n,n] \ra \Z$ is called a \emph{$P$-partition of type $B$} if it satisfies the following conditions:
\begin{enumerate}[label = {\rm (\arabic*)}]
\item If $i \preceq_P j$, then $f(i) \le f(j)$.
\item If $i \preceq_P j$ and $i > j$, then $f(i) < f(j)$.
\item $f(-i) = -f(i)$.
\end{enumerate}
\end{definition}

Note that conditions (1) and (2) in \cref{def: P-partitions type B} are precisely the conditions defining $P$-partitions of type $A$. Let $\calA^B(P)$ denote the set of all $P$-partitions of type $B$.

A \emph{type-$B$ linear extension} of a $B_n$ poset $P$ is a linear extension of $P$ that is itself a $B_n$ poset. Let $\calL^B(P)$ denote the set of type-$B$ linear extensions of $P$.

\begin{theorem}{\rm (\cite[Theorem 2.1.4]{Chow01})}\label{thm: Bn P-partition decomp}
Let $P$ be a $B_n$ poset.
Then $\calA^B(P) = \bigsqcup_{E \in \calL^B(P)} \calA^B(E)$, where $\bigsqcup$ denotes disjoint union.
\end{theorem}

Let $P$ be a $B_n$ poset.
The \emph{$P$-partition generating function of type $B$} \cite[Definition 2.1.2]{Chow01} is the type $B$ quasisymmetric function $K^B_P$ defined by 
\[
K^B_P := \sum_{f \in \calA^B(P)} x_{|f(1)|}x_{|f(2)|} \cdots x_{|f(n)|}.
\]
Let $\gamma\in \SG^B_n$, and let $P(\gamma)$ be the $B_n$ poset defined by $i\prec_{P(\gamma)} j$ if $\gamma^{-1}(i)<\gamma^{-1}(j)$. 
Note that $P(\gamma)$ is a total ordering. In \cite{Chow01}, given $I \subseteq [0,n-1]$, $F^B_I$ is defined to satisfy $F^B_I = K^B_{P(\gamma)}$ for all $\gamma \in \SG^B_n$ with $\Des_R(\gamma) = I$. 
As noted in \cite{Chow01}, it follows from \cref{thm: Bn P-partition decomp} that
\begin{align}\label{eq: Bn type P-partitions}
K^B_P 
= \sum_{E \in \calL^B(P)} K^B_E.
\end{align}

\section{Poset modules of type $B$ and the associated category}\label{Sec: poset modules}

In this section, for each $B_n$-poset $P$, we define an $H^B_n(0)$-module $M^B_P$, called the \emph{$B_n$ poset module associated to $P$}, and show that $\ch^B([M^B_P]) = K^B_P$. 
We then provide formulas for induction and restriction of $B_n$ poset modules, and show that the Grothendieck group associated to $B_n$ poset modules is isomorphic to $\QSym^B$ as $\QSym$-modules and comodules. 
Finally, we describe how $M^B_P$ behaves under (anti-)automorphism twists.

\subsection{Definition of poset modules of type $B$}

Let $P$ be a $B_n$ poset.
For $(E, \preceq_E) \in \calL^B(P)$, let $\bfgam_E$ be the signed permutation such that $P(\bfgam_E) = E$, equivalently, $\bfgam_E(1) \prec_E \bfgam_E(2) \prec_E \cdots \prec_E \bfgam_E(n)$.
We define 
\[
\Sigma^B_R(P) := \{\bfgam_E \mid E \in \calL^B(P)\}.
\]
For example, consider $P^{(2)}_3$ in \cref{eg: Bn posets}(2).
We have 
\[
\def \pp {0.6}
\calL^B\left(\hspace{-1ex}
\begin{array}{l}
\begin{tikzpicture}
\node at (-0.7*\pp, 1*\pp) {$\bullet$};
\node[right] at (-0.7*\pp, 1*\pp) {\scriptsize $-1$};
\node at (0.7*\pp, 1*\pp) {$\bullet$};
\node[right] at (0.7*\pp, 1*\pp) {\scriptsize $2$};
\node at (0*\pp, 0*\pp) {$\bullet$};
\node[right] at (0*\pp, 0*\pp) {\scriptsize $0$};
\node at (-0.7*\pp, -1*\pp) {$\bullet$};
\node[right] at (-0.7*\pp, -1*\pp) {\scriptsize $1$};
\node at (0.7*\pp, -1*\pp) {$\bullet$};
\node[right] at (0.7*\pp, -1*\pp) {\scriptsize $-2$};

\draw (0*\pp, 0*\pp) -- (-0.7*\pp, 1*\pp);
\draw (0*\pp, 0*\pp) -- (0.7*\pp, 1*\pp);
\draw (0*\pp, 0*\pp) -- (-0.7*\pp, -1*\pp);
\draw (0*\pp, 0*\pp) -- (0.7*\pp, -1*\pp);
\end{tikzpicture}
\end{array}
\hspace{-2ex}\right) 
= 
\left\{
\begin{array}{l}
\begin{tikzpicture}
\node at (0*\pp, 2*\pp) {$\bullet$};
\node[right] at (0*\pp, 2*\pp) {\scriptsize $2$};
\node at (0*\pp, 1*\pp) {$\bullet$};
\node[right] at (0*\pp, 1*\pp) {\scriptsize $-1$};
\node at (0*\pp, 0*\pp) {$\bullet$};
\node[right] at (0*\pp, 0*\pp) {\scriptsize $0$};
\node at (0*\pp, -1*\pp) {$\bullet$};
\node[right] at (0*\pp, -1*\pp) {\scriptsize $1$};
\node at (0*\pp, -2*\pp) {$\bullet$};
\node[right] at (0*\pp, -2*\pp) {\scriptsize $-2$};

\draw (0*\pp, 1*\pp) -- (0*\pp, 2*\pp);
\draw (0*\pp, 0*\pp) -- (0*\pp, 1*\pp);
\draw (0*\pp, -1*\pp) -- (0*\pp, 0*\pp);
\draw (0*\pp, -2*\pp) -- (0*\pp, -1*\pp);
\end{tikzpicture}
\end{array}, 
\begin{array}{l}
\begin{tikzpicture}
\node at (0*\pp, 2*\pp) {$\bullet$};
\node[right] at (0*\pp, 2*\pp) {\scriptsize $-1$};
\node at (0*\pp, 1*\pp) {$\bullet$};
\node[right] at (0*\pp, 1*\pp) {\scriptsize $2$};
\node at (0*\pp, 0*\pp) {$\bullet$};
\node[right] at (0*\pp, 0*\pp) {\scriptsize $0$};
\node at (0*\pp, -1*\pp) {$\bullet$};
\node[right] at (0*\pp, -1*\pp) {\scriptsize $-2$};
\node at (0*\pp, -2*\pp) {$\bullet$};
\node[right] at (0*\pp, -2*\pp) {\scriptsize $1$};

\draw (0*\pp, 1*\pp) -- (0*\pp, 2*\pp);
\draw (0*\pp, 0*\pp) -- (0*\pp, 1*\pp);
\draw (0*\pp, -1*\pp) -- (0*\pp, 0*\pp);
\draw (0*\pp, -2*\pp) -- (0*\pp, -1*\pp);
\end{tikzpicture}
\end{array}
\right\}
\quad \text{and} \quad
\Sigma^B_R(P^{(2)}_3) = \{\bracket{-1, 2}, \bracket{2, -1}\}.
\]

\begin{theorem}\label{thm: Bn poset asc compatible}
For any $B_n$ poset $P$, the set $\Sigma^B_R(P)$ is ascent-compatible.
\end{theorem}

\begin{proof}
Since a quadruple  $(\gamma,\xi,s,t) \in (\Sigma^B_R(P))^2 \times (S^B)^2$ is aligned if and only if the quadruple $(\xi,\gamma, t, s)$ is aligned, it suffices to show that if $(\gamma,\xi,s,t)$ is aligned and $\gamma s \in \Sigma^B_R(P)$, then $\xi t \in \Sigma^B_R(P)$.
Choose arbitrary $\gamma,\xi \in \Sigma^B_R(P)$ and $i,j \in [0,n-1]$ such that $(\gamma,\xi,s^B_i,s^B_j)$ is aligned and $\gamma s^B_i \in \Sigma^B_R(P)$.

Suppose that both $i$ and $j$ are nonzero.
Then
\[
\gamma(i) \prec_{P(\gamma)} \gamma(i+1)
\ \  \text{and} \ \ 
\gamma(i+1) = \gamma s^B_i(i) \prec_{P(\gamma s^B_i)} \gamma s^B_i(i+1) = \gamma(i).
\]
Since $P(\gamma), P(\gamma s^B_i) \in \calL^B(P)$, $\gamma(i)$ and $\gamma(i+1)$ are incomparable in $P$.
Assume, for the sake of contradiction, that $\xi s^B_j \notin \Sigma^B_R(P)$.
Since $\xi \in \Sigma^B_R(P)$, we have 
\[
\xi(l_1) \prec_{P} \xi(l_2) 
\quad  \text{or} \quad
\text{$\xi(l_1)$ and $\xi(l_2)$ are incomparable in $P$}
\]
for any $l_1,l_2 \in [-n,n]^2$ with $l_1 < l_2$.
Combining this with the assumption $\xi s^B_j \notin \Sigma^B_R(P)$ yields that 
\[
\xi(j) \prec_P \xi(j+1) 
\quad \text{and} \quad
-\xi(j+1) \prec_P -\xi(j).
\]
Since $i$ and $j$ are nonzero and $(\gamma,\xi,s^B_i,s^B_j)$ is aligned, we have
\[
(\gamma(i), \gamma(i+1)) (-\gamma(i), -\gamma(i+1)) 
= \gamma s^B_i \gamma^{-1}
= \xi s^B_j \xi^{-1}
= (\xi(j), \xi(j+1)) (-\xi(j), -\xi(j+1)).
\]
In addition, since $s^B_i \in \Asc_R(\gamma)$ and $s^B_j \in \Asc_R(\xi)$, by \cref{eq: Descent for type B} it follows that $\gamma(i) < \gamma(i+1)$ and $\xi(j) < \xi(j+1)$.
Thus, we have either 
\[
\gamma(i) = \xi(j) \text{ and } \gamma(i+1) = \xi(j+1), 
\quad \text{or} \quad
\gamma(i) = -\xi(j+1) \text{ and } \gamma(i+1) = -\xi(j).
\]
However, this is impossible because $\gamma(i)$ and $\gamma(i+1)$ are incomparable in $P$, whereas $\xi(j) \prec_P \xi(j+1)$ and $-\xi(j+1) \prec_P -\xi(j)$.
Therefore, $\xi s^B_j \in \Sigma^B_R(P)$.

Suppose that one of $i$ and $j$ is zero.
Note that for any $w \in \SG^B_n$ and $l \in [1,n-1]$, 
\[
w s^B_0 w^{-1} = (w(1), w(-1)) \quad \text{and} \quad w s^B_l w^{-1} = (w(l), w(l+1))(-w(l), -w(l+1)).
\]
Combining these equalities with $\gamma s^B_i \gamma^{-1} = \xi s^B_j \xi^{-1}$, we deduce that $i = 0$ if and only if $j = 0$, so we have $i = j = 0$.
Since $(\gamma(1), \gamma(-1)) = \gamma s^B_0 \gamma^{-1} = \xi s^B_0 \xi^{-1} = (\xi(1), \xi(-1))$, we have 
\[
|\gamma(1)| = |\xi(1)|.
\]
In addition, we have
\[
\gamma(-1) \prec_{P(\gamma)} \gamma(1)
\quad \text{and} \quad 
\gamma(1) = \gamma s^B_0 (-1) \prec_{P(\gamma s^B_0)} \gamma s^B_0 (1) = \gamma(-1),
\] 
which implies that $\gamma(-1)$ and $\gamma(1)$ are incomparable in $P$.
Putting these together shows that $\xi(-1)$ and $\xi(1)$ are incomparable in $P$.
Therefore, $\xi s^B_0 \in \Sigma^B_R(P)$.
\end{proof}

Now, based on \cref{thm: ascent-comp and 0-Hecke action} and \cref{thm: Bn poset asc compatible}, given a $B_n$ poset $P$, we define an $H^B_n(0)$-module with underlying space $\C \Sigma^B_R(P)$.

\begin{definition}
Let $P$ be a $B_n$ poset.
The \emph{$B_n$ poset module associated to $P$}, denoted by $M^B_P$, is the $H^B_{n}(0)$-module with underlying space $\C \Sigma^B_R(P)$ and with the right $H^B_{n}(0)$-action defined by
\[
\gamma \cdot \opi_i = \begin{cases}
-\gamma & \text{if $s^B_i \in \Des_R(\gamma)$}, \\
0 & \text{if $s^B_i \notin \Des_R(\gamma)$ and $ \gamma s^B_i \notin \Sigma^B_R(P)$}, \\
\gamma s^B_i & \text{if $s^B_i \notin \Des_R(\gamma)$ and $\gamma s^B_i \in \Sigma^B_R(P)$}
\end{cases}
\]
for all $\gamma \in \Sigma^B_R(P)$ and extending by linearity.
\end{definition}

The modules $M^B_P$ provide a representation-theoretic interpretation of the type-$B$ $P$-partition generating functions.

\begin{theorem}
Let $P$ be a $B_n$ poset. We have
\[\ch^B([M^B_P]) = K^B_P.\]
\end{theorem}
\begin{proof}
From \cref{theorem:characteristic}, it follows that
\[
\ch^B([M^B_P]) = \sum_{\gamma \in \Sigma^B_R(P)} F^B_{\Des_R(\gamma)}
= \sum_{E \in \calL^B(P)} F^B_{\Des_R(\bfgam_E)}.
\]
Recall that given $I \subseteq [0,n-1]$, $F^B_I = K^B_{P(\gamma)}$ for all $\gamma \in \SG^B_n$ with $\Des_R(\gamma) = I$.
Putting these together with \cref{eq: Bn type P-partitions} yields that
\[
\ch^B([M^B_P]) = \sum_{E \in \calL^B(P)} K^B_E = K^B_P. \qedhere
\]
\end{proof}

\subsection{The category of $B_n$ poset modules and its Grothendieck group}

For each $n \in \Z_{\ge 0}$, let $\scrP^B(n)$ be the full subcategory of $\modHBn$ whose objects are direct sums of finitely many isomorphic copies of $B_n$ poset modules.
The purpose of this subsection is to prove that $\calG_0(\scrP^B(n))$ is isomorphic to $\QSymB_n$ as a $\Z$-module.

Given a $B_n$ poset $P$, let 
\[
\inc^B(P) := \{(u,v) \in [-n,n] \mid u \prec_{E_1} v \text{ and } v \prec_{E_2} u \text{ for some $E_1, E_2 \in \calL^B(P)$} \}.
\]
Note that if $(u,v)\in \inc^B(P)$, then $u$ and $v$ are incomparable in $P$. However, due to the definition of $\calL^B(P)$, the converse is not always true. 
For example, the poset $P^{(1)}_2$ in \cref{eg: Bn posets} has incomparable elements, but $\inc^B(P^{(1)}_2) = \emptyset$ since $|\calL^B(P^{(1)}_2)| = 1$. 

For any $(u,v) \in \inc^B(P)$, let $\preceq_{P_{(u,v)}}$ be the relation on $[-n,n]$ defined as the transitive closure of the union of the covering relations of $P$ and the additional relations where $v$ covers $u$ and $-u$ covers $-v$.

\begin{lemma}
Let $P$ be a $B_n$ poset.
For any $(u,v) \in \inc^B(P)$, $P_{(u,v)} := ([-n,n], \preceq_{P_{(u,v)}})$ is a $B_n$ poset.
\end{lemma}

\begin{proof}
Let $(u,v) \in \inc^B(P)$.
We first show that $P_{(u,v)}$ is a poset. 
For an arbitrary poset $Q$ on $[-n,n]$ and incomparable elements $a,b \in Q$,  let $Q'_{(a,b)} = ([-n,n], \preceq_{Q'_{(a,b)}})$ be the poset whose order relation $\preceq_{Q'_{(a,b)}}$ is the transitive closure of the union of the covering relations of $Q$ and the relation where $b$ covers $a$.
It is straightforward to show that $Q'_{(a,b)}$ is a poset.
Since $P_{(u,v)} = (P'_{(u,v)})'_{(-v,-u)}$, to show that $P_{(u,v)}$ is a poset, it suffices to prove that $-v$ and $-u$ are incomparable in $P'_{(u,v)}$.
In the case where $u = -v$, we have $P_{(u,v)} = P'_{(u,v)}$ which is already a poset.
Hence, we assume that $u \neq -v$.

Suppose that $-u \prec_{P'_{(u,v)}} -v$.
Note that $\calL(P'_{(u,v)}) = \{E \in \calL(P) \mid u \prec_E v\}$.
This, together with the relation 
$-u \prec_{P'_{(u,v)}} -v$, implies that for all $E \in \calL(P)$ satisfying $u \prec_{E} v$, we have $-u \prec_E -v$.
However, this contradicts the assumption $(u, v) \in \inc^B(P)$, which 
guarantees the existence of an $E\in \calL(P)$ satisfying $u \prec_{E} v$ and $-v \prec_E -u$.
Suppose that $-v \prec_{P'_{(u,v)}} -u$.
Since $-v$ and $-u$ are incomparable in $P$, the only way this can happen is if $-v \prec_P u$ and $v \prec_P -u$.
But, these two relations contradict $P$ being a $B_n$ poset.
Therefore, $-v$ and $-u$ are incomparable in $P'_{(u,v)}$.

Next, let us prove that for any $i,j \in [-n,n], i \preceq_{P_{(u,v)}} j$ if and only if $-j \preceq_{P_{(u,v)}} -i$.
Let $i,j \in [-n,n]$ be with $i \preceq_{P_{(u,v)}} j$.
In the case where $i \preceq_P j$, we have $-j \preceq_P -i$ and so $-j \preceq_{P_{(u,v)}} -i$.
Suppose that $i \not\preceq_{P} j$.
Then, by the definition of $P_{(u,v)}$,
\[
i \preceq_P x_1 \prec_{P_{(u,v)}} x_2 \preceq_P j \quad \text{for some $(x_1,x_2) \in \{(u,v), (-v,-u)\}$.}
\]
Since $P$ is a $B_n$ poset, we have 
$-j \preceq_P -x_2$ and $-x_1 \preceq_P -i$.
Combining these with the relation $-x_2 \preceq_{P_{(u,v)}} -x_1$ implies $-j \preceq_{P_{(u,v)}} -i$.
\end{proof}

Let $(u,v) \in \inc^B(P)$.
By definition, $\calL^B(P)$ is the disjoint union of $\calL^B(P_{(u,v)})$ and $\calL^B(P_{(v,u)})$, and so $\Sigma^B_R(P)$ is the disjoint union of $\Sigma^B_R(P_{(u,v)})$ and $\Sigma^B_R(P_{(v,u)})$. Let $\iota: M^B_{P_{(v,u)}} \rightarrow M^B_P$ be the injection of $\mathbb{C}$-vector spaces defined by $\iota(\gamma)=\gamma$ for $\gamma\in \Sigma^B_R(P_{(v,u)})$, and ${\rm pr}: M^B_P \rightarrow M^B_{P_{(u,v)}}$ the projection of $\mathbb{C}$-vector spaces defined by 
\[{\rm pr}(\gamma) = 
\begin{cases} \gamma & \text{ if } \gamma \in \Sigma^B_R(P_{(u,v)}), \\
                0    & \text{ otherwise},
\end{cases}                
\]
for $\gamma \in \Sigma^B_R(P)$. By definition, ${\rm Im}(\iota) = {\rm ker}({\rm pr})=\mathbb{C}\Sigma^B_R(P_{(v,u)})$.

\begin{lemma}\label{lem: short exact sequence}
Let $P$ be a $B_n$ poset. If $(u,v)\in \inc^B(P)$ with $u<v$ (as integers), then
\[
\begin{tikzcd}
0 \arrow[r] & {M^B_{P_{(v,u)}} } \arrow[r] & M^B_P \arrow[r] & {M^B_{P_{(u,v)}}} \arrow[r] & 0
\end{tikzcd}
\]
is a short exact sequence of $H_n^B(0)$-modules.    
\end{lemma}
\begin{proof}
We show that $M^B_{P_{(v,u)}}$ is a $H_n^B(0)$-submodule of $M^B_P$; in other words, $M^B_{P_{(v,u)}}$ is closed under the action of $H_n^B(0)$ on $M^B_P$. 
Assume, for the sake of a contradiction, that there exist $\gamma\in \Sigma^B_R(P_{(v,u)})$ and $i\in [0, n-1]$ such that $\gamma \cdot \opi^B_i \notin M^B_{P_{(v,u)}}$. 
Then, by the definition of the $H_n^B(0)$-action, $\gamma \cdot \opi^B_i = \gamma s^B_i$. 
If $i\neq 0$, this means either $\gamma(i)=v$ and $\gamma(i+1)=u$, or $\gamma(i)=-u$ and $\gamma(i+1)=-v$. 
Since $v > u$ as integers, we have $i \in \Des_R(\gamma)$, which contradicts the equality $\gamma \cdot \opi^B_i = \gamma s^B_i$.
If $i=0$, we have three cases
\[
\textrm{(i)}~
\gamma(-1) = v  > u = \gamma(0),
\quad
\textrm{(ii)}~
\gamma(0) = v > u = \gamma(1)
\quad \text{and} \quad
\textrm{(iii)}~
\gamma(-1) = v > u = \gamma(1).
\]
In each case $\gamma(1)$ is negative, hence $0 \in \Des_R(\gamma)$, which again contradicts the equality $\gamma \cdot \opi^B_0 = \gamma s^B_0$.

It is clear that $M_{P_{(u,v)}}$ is isomorphic to the quotient module $M_P/M_{P_{(v,u)}}$.
\end{proof}

\begin{theorem}\label{thm: Grothendieck group of poset modules}
For each $n \in \Z_{\ge 0}$, $\calG_0(\scrP^B(n)) \cong \QSymB_n$ as $\Z$-modules.
\end{theorem}

\begin{proof}
For every $B_n$ poset $P$, if $\dim(M^B_P) > 1$, then by \cref{lem: short exact sequence} there exist $B_n$ posets $P'$ and $P''$ such that $\dim(M^B_{P'}), \dim(M^B_{P''}) < \dim(M^B_P)$ and 
\[
[M^B_P] = [M^B_{P'}] + [M^B_{P''}] 
\quad \text{in $\calG_0(\scrP^B(n))$.}
\]
By repeatedly applying such decompositions, we can express $[M^B_P]$ as a sum of classes of simple modules.
It follows that $\calG_0(\scrP^B(n))$ is isomorphic to $\calG_0(\modHBn)$ as a $\Z$-module, and hence isomorphic to $\QSymB_n$.
\end{proof}

\subsection{Induction of poset modules}
The purpose of this subsection is to show that, for any $B_m$ poset $P_1$ and any poset $P_2$ on $[n]$, the induced representation $M^B_{P_1} \otimes M_{P_2}\uparrow_{H^B_m(0) \otimes H_n(0)}^{H^B_{m+n}(0)}$ is isomorphic to a $B_{m+n}$ poset module. 
This immediately implies that
\[
[M^B_{P_1}] \boxtimes^B [M_{P_2}] := 
\left[
M^B_{P_1} \otimes M_{P_2} \uparrow_{H^B_m(0) \otimes H_n(0)}^{H^B_{m+n}(0)}
\right]
\]
defines a right $\bigoplus_{n \ge 0} \calG_0(\scrP(n))$-action on $\bigoplus_{n \ge 0} \calG_0(\scrP^B(n))$.

Let 
\begin{align}\label{eq: calIs}
\calI_{-1} := [-(m+n), -(m+1)], 
\quad
\calI_{0} := [-m,m]
\quad \text{and} \quad 
\calI_{1} := [m+1, m+n].
\end{align}
Given a $B_m$ poset $P_1 = ([-m,m], \preceq_{P_1})$ and a poset $P_2 = ([n], \preceq_{P_2})$, define the relation $\preceq_{P_1 \sqcup_B P_2}$ on $[-(m+n), m+n]$ by
\begin{align}\label{eq: def of sqcupB}
i \preceq_{P_1 \sqcup_B P_2} j \quad \text{if and only if} \quad
\begin{cases}
\text{$i,j \in \calI_{-1}$ and $-j-m \preceq_{P_2} -i-m$,} \\
\text{$i,j \in \calI_{0}$ and $i \preceq_{P_1} j$, or} \\
\text{$i,j \in \calI_{1}$ and $i-m \preceq_{P_2} j-m$.}
\end{cases}
\end{align}

\begin{lemma}
Let $P_1 = ([-m,m], \preceq_{P_1})$ be a $B_m$ poset and $P_2 = ([n], \preceq_{P_2})$ a poset. Then $P_1 \sqcup_B P_2 := ([-(m+n), m+n], \preceq_{P_1 \sqcup_B P_2})$ is a $B_{m+n}$ poset.
\end{lemma}

\begin{proof}
We first show that $P_1 \sqcup_B P_2$ is a poset.
It is clear that $\preceq_{P_1 \sqcup_B P_2}$ is reflexive.
To show that $\preceq_{P_1 \sqcup_B P_2}$ is antisymmetric, suppose that we have $i,j \in [-(m+n), m+n]$ such that $i \preceq_{P_1 \sqcup_B P_2} j$ and $j \preceq_{P_1 \sqcup_B P_2} i$.
By the definition of $\preceq_{P_1 \sqcup_B P_2}$, $i$ and $j$ are in the same $\calI_l$ for some $l \in \{-1,0,1\}$, otherwise they are incomparable under $\preceq_{P_1 \sqcup_B P_2}$.
If $l = -1$, then the relations $i \preceq_{P_1 \sqcup_B P_2} j$ and $j \preceq_{P_1 \sqcup_B P_2} i$ imply 
\[
-j -m \preceq_{P_2} -i - m \quad \text{and} \quad 
-i -m \preceq_{P_2} -j -m,
\]
respectively.
Combining these relations with the antisymmetry of $\preceq_{P_2}$ yields that $-j - m = -i -m$, hence $i = j$.
If $l = 0$, then the relations $i \preceq_{P_1 \sqcup_B P_2} j$ and $j \preceq_{P_1 \sqcup_B P_2} i$ imply 
\[
i \preceq_{P_1} j \quad \text{and} \quad 
j \preceq_{P_1} i,
\]
respectively.
Combining these relations with the antisymmetry of $\preceq_{P_1}$ yields that $i = j$.
If $l = 1$, then the relations $i \preceq_{P_1 \sqcup_B P_2} j$ and $j \preceq_{P_1 \sqcup_B P_2} i$ imply 
\[
i-m \preceq_{P_2} j-m \quad \text{and} \quad 
j-m \preceq_{P_2} i-m,
\]
respectively.
Combining these relations with the antisymmetry of $\preceq_{P_2}$ yields that $i-m = j-m$, hence $i = j$.

In order to prove the transitivity of $\preceq_{P_1 \sqcup_B P_2}$, suppose that we have $i, j ,k \in [-(m+n), m+n]$ such that $i \preceq_{P_1 \sqcup_B P_2} j$ and $j \preceq_{P_1 \sqcup_B P_2} k$.
Again, by the definition of $\preceq_{P_1 \sqcup_B P_2}$, $i,j,k$ are in the same $\calI_l$ for some $l \in \{-1,0,1\}$.
If $l = -1$, then the relations $i \preceq_{P_1 \sqcup_B P_2} j$ and $j \preceq_{P_1 \sqcup_B P_2} k$ imply
\[
-j-m \preceq_{P_2} -i -m \quad \text{and}\quad
-k -m \preceq_{P_2} -j -m,
\]
respectively.
Combining these relations with the transitivity of $\preceq_{P_2}$, we have $-k-m \preceq_{P_2} -i -m$, hence $i \preceq_{P_1 \sqcup_B P_2} k$.
If $l = 0$, then the relations $i \preceq_{P_1 \sqcup_B P_2} j$ and $j \preceq_{P_1 \sqcup_B P_2} k$ imply
\[
i \preceq_{P_1} j \quad \text{and}\quad
j \preceq_{P_1} k,
\]
respectively.
Combining these relations with the transitivity of $\preceq_{P_1}$, we have $i \preceq_{P_1} k$, hence $i \preceq_{P_1 \sqcup_B P_2} k$.
If $l = 1$, then the relations $i \preceq_{P_1 \sqcup_B P_2} j$ and $j \preceq_{P_1 \sqcup_B P_2} k$ imply
\[
i-m \preceq_{P_2} j -m \quad \text{and}\quad
j -m \preceq_{P_2} k -m,
\]
respectively.
Combining these relations with the transitivity of $\preceq_{P_2}$, we have $i-m \preceq_{P_2} k -m$, hence $i \preceq_{P_1 \sqcup_B P_2} k$.
Therefore, $P_1 \sqcup_B P_2$ is a poset.

In addition, from the definition of $\preceq_{P_1 \sqcup_B P_2}$ it is straightforward that 
for any $i,j \in [-(m+n), m+n]$, if $i \preceq_{P_1 \sqcup_B P_2} j$, then $-j \preceq_{P_1 \sqcup_B P_2} -i$.
\end{proof}

\begin{example}\label{eg: combining posets}
Let $\def \pp {0.6}
P_1 = 
\begin{array}{l}
\begin{tikzpicture}
\node at (-1*\pp, 2*\pp) {$\bullet$};
\node[right] at (-1*\pp, 2*\pp) {\scriptsize $-2$};
\node at (1*\pp, 2*\pp) {$\bullet$};
\node[right] at (1*\pp, 2*\pp) {\scriptsize $1$};
\node at (0*\pp, 1*\pp) {$\bullet$};
\node[right] at (0*\pp, 1*\pp) {\scriptsize $3$};
\node at (0*\pp, 0*\pp) {$\bullet$};
\node[right] at (0*\pp, 0*\pp) {\scriptsize $0$};
\node at (0*\pp, -1*\pp) {$\bullet$};
\node[right] at (0*\pp, -1*\pp) {\scriptsize $-3$};
\node at (-1*\pp, -2*\pp) {$\bullet$};
\node[right] at (-1*\pp, -2*\pp) {\scriptsize $2$};
\node at (1*\pp, -2*\pp) {$\bullet$};
\node[right] at (1*\pp, -2*\pp) {\scriptsize $-1$};

\draw (-1*\pp, 2*\pp) -- (0*\pp, 1*\pp);
\draw (1*\pp, 2*\pp) -- (0*\pp, 1*\pp);
\draw (0*\pp, 1*\pp) -- (0*\pp, 0*\pp);
\draw (0*\pp, 0*\pp) -- (0*\pp, -1*\pp);
\draw (-1*\pp, -2*\pp) -- (0*\pp, -1*\pp);
\draw (1*\pp, -2*\pp) -- (0*\pp, -1*\pp);
\end{tikzpicture}
\end{array}$ be a $B_3$ poset and 
$\def \pp {0.6}
P_2 =
\begin{array}{l}
\begin{tikzpicture}
\node at (0*\pp, 1.5*\pp) {$\bullet$};
\node[right] at (0*\pp, 1.5*\pp) {\scriptsize $1$};

\node at (0*\pp, 0*\pp) {$\bullet$};
\node[right] at (0*\pp, 0*\pp) {\scriptsize $2$};

\draw (0*\pp, 0*\pp) -- (0*\pp, 1.5*\pp);
\end{tikzpicture}
\end{array}$ be a poset on $[2]$.
Then 
\[\def \pp {0.6}
P_1 \sqcup_B P_2 = 
\begin{array}{l}
\begin{tikzpicture}
\node at (-1*\pp, 2*\pp) {$\bullet$};
\node[right] at (-1*\pp, 2*\pp) {\scriptsize $-2$};
\node at (1*\pp, 2*\pp) {$\bullet$};
\node[right] at (1*\pp, 2*\pp) {\scriptsize $1$};
\node at (0*\pp, 1*\pp) {$\bullet$};
\node[right] at (0*\pp, 1*\pp) {\scriptsize $3$};
\node at (0*\pp, 0*\pp) {$\bullet$};
\node[right] at (0*\pp, 0*\pp) {\scriptsize $0$};
\node at (0*\pp, -1*\pp) {$\bullet$};
\node[right] at (0*\pp, -1*\pp) {\scriptsize $-3$};
\node at (-1*\pp, -2*\pp) {$\bullet$};
\node[right] at (-1*\pp, -2*\pp) {\scriptsize $2$};
\node at (1*\pp, -2*\pp) {$\bullet$};
\node[right] at (1*\pp, -2*\pp) {\scriptsize $-1$};

\draw (-1*\pp, 2*\pp) -- (0*\pp, 1*\pp);
\draw (1*\pp, 2*\pp) -- (0*\pp, 1*\pp);
\draw (0*\pp, 1*\pp) -- (0*\pp, 0*\pp);
\draw (0*\pp, 0*\pp) -- (0*\pp, -1*\pp);
\draw (-1*\pp, -2*\pp) -- (0*\pp, -1*\pp);
\draw (1*\pp, -2*\pp) -- (0*\pp, -1*\pp);

\node at (3*\pp, 0.75*\pp) {$\bullet$};
\node[right] at (3*\pp, 0.75*\pp) {\scriptsize $4$};
\node at (3*\pp, -0.75*\pp) {$\bullet$};
\node[right] at (3*\pp, -0.75*\pp) {\scriptsize $5$};
\draw (3*\pp, -0.75*\pp) -- (3*\pp, 0.75*\pp);

\node at (5*\pp, 0.75*\pp) {$\bullet$};
\node[right] at (5*\pp, 0.75*\pp) {\scriptsize $-5$};
\node at (5*\pp, -0.75*\pp) {$\bullet$};
\node[right] at (5*\pp, -0.75*\pp) {\scriptsize $-4$};
\draw (5*\pp, -0.75*\pp) -- (5*\pp, 0.75*\pp);
\end{tikzpicture}
\end{array}.
\]
\end{example}

Given $\sigma \in \SG^B_m$ and $\rho \in \SG_n$, define $\sigma \bulB \rho \in \SG^B_{m+n}$ by
\[
\sigma \bulB \rho(i) := \begin{cases}
-\rho(-i -m) -m & \text{if $i \in \calI_{-1}$,} \\
\sigma(i) & \text{if $i \in \calI_0$,} \\
\rho(i-m) + m & \text{if $i \in \calI_1$.}
\end{cases}
\]
Indeed, 
\begin{align*}
\sigma \bulB \rho(-i) & = \rho(-i - m) + m = -\sigma \bulB \rho(i) 
\quad \text{for $i \in \calI_{-1}$,} \\
\sigma \bulB \rho(-i) & = \sigma(-i) = -\sigma \bulB \rho(i) 
\quad \text{for $i \in \calI_{0}$, and}\\
\sigma \bulB \rho(-i) &= - \rho(i - m) - m = -\sigma \bulB \rho(i) 
\quad \text{for $i \in \calI_1$.}
\end{align*}
For example, $\bracket{3, 1, -2} \cdot_B 12 = \bracket{3, 1, -2, 4, 5}$.
One can easily see that
\begin{align*}
(\sigma \bulB \rho)^{-1}(i) = 
\begin{cases}
-\rho^{-1}(-i -m) -m & \text{if $i \in \calI_{-1}$,} \\
\sigma^{-1}(i) & \text{if $i \in \calI_0$,} \\
\rho^{-1}(i-m) + m & \text{if $i \in \calI_1$.}
\end{cases}
\end{align*}
Given a subset $X$ of $[0,m+n-1]$, let $(\SG^B_{m+n})_X$ denote the parabolic subgroup of $\SG^B_{m+n}$ generated by $\{s_i^B : i\in X\}$. Let $I^{(m)} := [0,m+n-1] \setminus \{m\}$, and let ${}^{I^{(m)}}(\SG^B_{m+n})$ denote the set of minimal length right coset representatives of $(\SG^B_{m+n})_{I^{(m)}}$. 
Given $\sigma \in \SG^B_m$ and $\rho \in \SG_n$, define $\sigma \shuffle_B \rho \subseteq \SG^B_{m+n}$ by
\begin{align*}
\sigma \shuffle_B \rho := \left\{(\sigma \bulB \rho) \delta \; \middle| \; \delta \in {}^{I^{(m)}}(\SG^B_{m+n})\right\}.
\end{align*}
By \cite[Proposition 8.1.4]{BB05},
\begin{align}\label{eq: description for minimal coset representatives}
{}^{I^{(m)}}(\SG^B_{m+n}) = 
\left\{
\sigma \in \SG^B_{m+n} 
\; \middle| \; 
\begin{array}{l}
0 < \sigma^{-1}(1) < \cdots < \sigma^{-1}(m), \\
\sigma^{-1}(m+1) < \cdots < \sigma^{-1}(m+n)
\end{array}
\right\}.
\end{align}
For $X \subseteq \SG^B_m$ and $Y \subseteq \SG_n$, let
\[
X \shuffle_B Y := \bigcup_{\sigma \in X, \;  \rho \in Y} \sigma \shuffle_B \rho.
\]

\begin{lemma}\label{lem: B shuffle of Sigma}
Let $P_1 = ([-m,m], \preceq_{P_1})$ be a $B_m$ poset and $P_2 = ([n], \preceq_{P_2})$ a poset.
Then 
\[
\Sigma^B_R(P_1) \shuffle_B \Sigma_R(P_2) = \Sigma^B_R(P_1 \sqcup_B P_2).
\]
\end{lemma}

\begin{proof}
To prove the inclusion $\Sigma^B_R(P_1) \shuffle_B \Sigma_R(P_2) \subseteq \Sigma^B_R(P_1 \sqcup_B P_2)$, take any $(\gamma \bulB \gamma') \delta \in \Sigma^B_R(P_1) \shuffle_B \Sigma_R(P_2)$.
Since $(\gamma \bulB \gamma') \delta \in \SG^B_{m+n}$, it suffices to show that $(\gamma \bulB \gamma') \delta \in \Sigma_R(P_1 \sqcup_B P_2)$, equivalently,
\[
\left( (\gamma \bulB \gamma') \delta \right)^{-1}(i) \le \left( (\gamma \bulB \gamma') \delta \right)^{-1}(j)
\quad \text{for $i,j \in [-(m+n), m+n]$ with $i \preceq_{P_1 \sqcup_B P_2} j$.}
\]
By the definition of $\preceq_{P_1 \sqcup_B P_2}$, we only need to consider the following cases:
\begin{enumerate}[label = {\rm (\roman*)}]
\item $i,j \in \calI_{-1}$ and $-j-m \preceq_{P_2} -i-m$,
\item $i,j \in \calI_{0}$ and $i \preceq_{P_1} j$,
\item $i,j \in \calI_{1}$ and $i-m \preceq_{P_2} j-m$.
\end{enumerate}

Suppose that $i,j \in \calI_{-1}$ and $-j-m \preceq_{P_2} -i-m$.
From the assumption $i,j \in \calI_{-1}$, it follows that
\[
(\gamma \bulB \gamma')^{-1}(i) 
= -\gamma'^{-1}(-i-m)-m
\quad \text{and} \quad
(\gamma \bulB \gamma')^{-1}(j) 
= -\gamma'^{-1}(-j-m)-m.
\]
Since $\gamma' \in \Sigma_R(P_2)$, the assumption $-j-m \preceq_{P_2} -i-m$ implies that $\gamma'^{-1}(-j-m) \le \gamma'^{-1}(-i-m)$, so
\[
-\gamma'^{-1}(-i-m)-m \le -\gamma'^{-1}(-j-m)-m.
\]
In addition, it is clear that both $-\gamma'^{-1}(-i-m)-m$ and $-\gamma'^{-1}(-j-m)-m$ are in $\calI_{-1}$.
Thus, by \cref{eq: description for minimal coset representatives}, 
\[
\left( (\gamma \bulB \gamma') \delta \right)^{-1}(i) =
\delta^{-1} (-\gamma'^{-1}(-i-m)-m) 
\le 
\delta^{-1} (-\gamma'^{-1}(-j-m)-m)
= \left( (\gamma \bulB \gamma') \delta \right)^{-1}(j).
\]

Suppose that $i,j \in \calI_{0}$ and $i \preceq_{P_1} j$.
From the assumption $i,j \in \calI_{0}$, it follows that
\[
(\gamma \bulB \gamma')^{-1}(i) 
= \gamma^{-1}(i)
\quad \text{and} \quad
(\gamma \bulB \gamma')^{-1}(j) 
= \gamma^{-1}(j).
\]
Since $\gamma \in \Sigma^B_R(P_1)$, the assumption $i \preceq_{P_1} j$ implies that $\gamma^{-1}(i) \le \gamma^{-1}(j)$.
It is clear that $\gamma^{-1}(i), \gamma^{-1}(j) \in \calI_0$.
Thus, by \cref{eq: description for minimal coset representatives}, 
\[
\left( (\gamma \bulB \gamma') \delta \right)^{-1}(i) =
\delta^{-1} (\gamma^{-1}(i)) 
\le 
\delta^{-1} (\gamma^{-1}(j))
= \left( (\gamma \bulB \gamma') \delta \right)^{-1}(j).
\]

Suppose that $i,j \in \calI_{1}$ and $i-m \preceq_{P_2} j-m$.
From the assumption $i,j \in \calI_{1}$, it follows that
\[
(\gamma \bulB \gamma')^{-1}(i) 
= \gamma'^{-1}(i-m)+m
\quad \text{and} \quad
(\gamma \bulB \gamma')^{-1}(j) 
= \gamma'^{-1}(j-m)+m.
\]
Since $\gamma' \in \Sigma_R(P_2)$, the assumption $i-m \preceq_{P_2} j-m$ implies that $\gamma'^{-1}(i-m) \le \gamma'^{-1}(j-m)$, so
\[
\gamma'^{-1}(i-m)+m \le \gamma'^{-1}(j-m)+m.
\]
It is clear that both $\gamma'^{-1}(i-m)+m$ and $\gamma'^{-1}(j-m)+m$ are in $\calI_{1}$.
Thus, by \cref{eq: description for minimal coset representatives}, 
\[
\left( (\gamma \bulB \gamma') \delta \right)^{-1}(i) =
\delta^{-1} (\gamma'^{-1}(i-m)+m) 
\le 
\delta^{-1} (\gamma'^{-1}(j-m)+m)
= \left( (\gamma \bulB \gamma') \delta \right)^{-1}(j).
\]

To prove the inclusion $\Sigma^B_R(P_1 \sqcup_B P_2) \subseteq \Sigma^B_R(P_1) \shuffle_B \Sigma_R(P_2)$, take any $\gamma \in \Sigma^B_R(P_1 \sqcup_B P_2)$.
Let $x_1, x_2, \ldots, x_m \in [m+n]$ such that 
\[
0 < x_1 < x_2 < \cdots <x_m
\quad \text{and} \quad
\{|\gamma(x_1)|, |\gamma(x_2)|, \ldots, |\gamma(x_m)|\} = [m]
\]
and let $y_1, y_2, \ldots, y_n \in [-(m+n), m+n]$ such that 
\[
y_1 < y_2 < \cdots < y_n
\quad \text{and} \quad 
\{\gamma(y_1), \gamma(y_2), \ldots, \gamma(y_n)\} = [m+1, m+n].
\]

Let $\gamma_1 := [\gamma(x_1), \gamma(x_2), \ldots, \gamma(x_m)] \in \SG^B_m$.
Since $\gamma(x_1), \gamma(x_2), \ldots, \gamma(x_m) \in \calI_0$, for any $i,j\in [m]$ we have
$\gamma_1(i) \preceq_{P_1} \gamma_1(j)$ if and only if $\gamma(x_i) \preceq_{P_1 \sqcup_B P_2} \gamma(x_j)$.
In addition, if $\gamma(x_i) \preceq_{P_1 \sqcup_B P_2} \gamma(x_j)$, then $x_i \le x_j$, which implies $i \le j$.
Therefore, if $\gamma_1(i) \preceq_{P_1} \gamma_1(j)$ then $i \le j$, and thus $\gamma_1 \in \Sigma_R^B(P_1)$.

Define $\gamma_2 \in \SG_n$ by $\gamma_2(i) := \gamma(y_i)-m$ for $1 \le i \le n$.
Since $\gamma(y_1), \gamma(y_2), \ldots, \gamma(y_n) \in [m+1,m+n]$, for any $i,j\in [n]$ we have $\gamma_2(i) \preceq_{P_2} \gamma_2(j)$ if and only if $\gamma(y_i) \preceq_{P_1 \sqcup_B P_2} \gamma(y_j)$.
In addition, if $\gamma(y_i) \preceq_{P_1 \sqcup_B P_2} \gamma(y_j)$, then $y_i \le y_j$, which implies $i \le j$.
Therefore, if $\gamma_2(i) \preceq_{P_2} \gamma_2(j)$ then $i \le j$, and thus $\gamma_2 \in \Sigma_R(P_2)$.

Let $\delta := (\gamma_1 \bulB \gamma_2)^{-1} \gamma$. Then $\gamma =(\gamma_1 \bulB \gamma_2) \delta$; we will show $\delta  \in {}^{I^{(m)}}(\SG^B_{m+n})$.
For $k \in [m]$, 
\[
\delta^{-1}(k) 
= \gamma^{-1} \circ (\gamma_1 \bulB \gamma_2)(k) 
= \gamma^{-1}(\gamma_1(k))
= \gamma^{-1}(\gamma(x_k))
= x_k.
\]
For $k \in [m+1,m+n]$,
\[
\delta^{-1}(k) 
= \gamma^{-1} \circ (\gamma_1 \bulB \gamma_2)(k) 
= \gamma^{-1}(\gamma_2(k-m)+m)
= \gamma^{-1}(\gamma(y_{k-m}))
= y_{k-m}.
\]
Since $(x_i)_{1 \le i \le m}$ and $(y_j)_{1 \le j \le n}$ are increasing sequences, we have
\[
0 < 
\delta^{-1}(1) < \delta^{-1}(2) < \cdots < \delta^{-1}(m)
\quad \text{and} \quad
\delta^{-1}(m+1) < \delta^{-1}(m+2) < \cdots <\delta^{-1}(m+n),
\]
which implies $\delta \in {}^{I^{(m)}}(\SG^B_{m+n})$.
Therefore, $\gamma =(\gamma_1 \bulB \gamma_2) \delta \in \Sigma^B_R(P_1) \shuffle_B \Sigma_R(P_2)$.
\end{proof}

\begin{example}
Revisit \cref{eg: combining posets}.
Let $\gamma = [3,1,4,-5,-2]$.
One sees that $\gamma \in \Sigma^B_R(P_1 \sqcup_B P_2)$.
The integers $x_1,x_2,x_3$ and $y_1,y_2$, chosen in the proof of \cref{lem: B shuffle of Sigma}, are given by
\[
x_1 = 1, \  x_2 = 2, \  x_3 = 5,
\quad \text{and} \quad
y_1 = -4, \  y_2 = 3.
\]
So, $\gamma_1 = \bracket{3,1,-2}$, $\gamma_2 = 21$, and
\[
\delta = (\gamma_1 \bulB \gamma_2)^{-1} \gamma 
= \bracket{3,1,-2,5,4}^{-1} \bracket{3,1,4,-5,-2} 
= \bracket{1,2,5,-4,3}.
\]
Indeed, one can see that $\gamma_1 \in \Sigma^B_R(P_1)$, $\gamma_2 \in \Sigma_R(P_2)$, and $\delta \in {}^{I^{(3)}}(\SG^B_{5})$, that is,
$\gamma = (\gamma_1 \cdot_B \gamma_2) \delta \in \Sigma^B_R(P_1) \shuffle_B \Sigma_R(P_2)$.
\end{example}

Now, we are ready to prove the main theorem of this subsection.

\begin{theorem}\label{thm: induction product}
Let $P_1 = ([-m,m], \preceq_{P_1})$ be a $B_m$ poset and $P_2 = ([n], \preceq_{P_2})$ a poset. Then $(M^B_{P_1} \otimes M_{P_2})\uparrow_{H^B_m(0) \otimes H_n(0)}^{H^B_{m+n}(0)} \cong M^B_{P_1 \sqcup_B P_2}$.
\end{theorem}

\begin{proof}
Let $\sfR := \{\opi^B_{\delta} \mid \delta \in {}^{I^{(m)}}(\SG^B_{m+n})\}$ and
\[
(\Sigma^B_R(P_1) \otimes \Sigma_R(P_2)) \otimes \sfR := \left\{ (\gamma \otimes \gamma') \otimes \opi^B_\delta \; \middle| \; \gamma \in \Sigma^B_R(P_1), \gamma' \in \Sigma_R(P_2), \opi^B_\delta \in \sfR \right\}.
\]
It is clear that $(\Sigma^B_R(P_1) \otimes \Sigma_R(P_2)) \otimes \sfR$ is a basis for $(M^B_{P_1} \otimes M_{P_2}) \otimes_{H^B_m(0) \otimes H_n(0)} H^B_{m+n}(0)$.
Define a linear map
\[
\fkL: (M^B_{P_1} \otimes M_{P_2}) \otimes_{H^B_m(0) \otimes H_n(0)} H^B_{m+n}(0) \ra M^B_{P_1 \sqcup_B P_2}, \quad
(\gamma \otimes \gamma') \otimes \opi^B_\delta \mapsto (\gamma \bulB \gamma') \cdot \opi^B_\delta
\]
for $(\gamma \otimes \gamma') \otimes \opi^B_\delta \in (\Sigma^B_R(P_1) \otimes \Sigma_R(P_2)) \otimes \sfR$.
By \cref{lem: B shuffle of Sigma}, $\fkL$ is a $\C$-linear isomorphism.

Let us show that $\fkL$ preserves the $H^B_{m+n}(0)$-action.
Take any $(\gamma \otimes \gamma') \otimes \opi^B_\delta \in (\Sigma^B_R(P_1) \otimes \Sigma_R(P_2)) \otimes \sfR$ and $i \in [0,m+n-1]$.
If $i \in \Des_R(\delta)$, then 
\[
\fkL((\gamma \otimes \gamma') \otimes \opi^B_\delta \cdot \opi^B_i)
= - (\gamma \bulB \gamma') \cdot \opi^B_\delta
= \fkL((\gamma \otimes \gamma') \otimes \opi^B_\delta) \cdot \opi^B_i.
\] 
Suppose that $i \notin \Des_R(\delta)$.
By \cref{lem: B shuffle of Sigma}, if $i \in [0,m-1]$, then $\gamma s^B_i \in \Sigma^B_R(P_1)$ if and only if $(\gamma \bulB \gamma') s^B_i \in \Sigma^B_R(P_1 \sqcup_B P_2)$, and if $i \in [m+1,m+n-1]$, then $\gamma' s_{i-m} \in \Sigma_R(P_2)$ if and only if $(\gamma \bulB \gamma')s^B_i \in \Sigma^B_R(P_1 \sqcup_B P_2)$.
Thus,
\begin{align}\label{eq: action on gam bulB gam'}
(\gamma \bulB \gamma') \cdot \opi^B_i = 
\begin{cases}
(\gamma \cdot \opi^B_i) \bulB \gamma' & \text{for $i \in [0,m-1]$,} \\
\gamma \bulB (\gamma' \cdot \opi_{i-m}) & \text{for $i \in [m+1, m+n-1]$},
\end{cases}
\end{align}
where $\gamma \bulB 0$ and $0 \bulB \gamma'$ are defined to be $0$. 
By \cite[Equation (2.12)]{BB05}, there exists a unique triple 
\[
(\xi, \xi', \delta') \in (\SG^B_{m+n})_{[0,m-1]} \times (\SG^B_{m+n})_{[m+1, m+n-1]} \times {}^{I^{(m)}}(\SG^B_{m+n})
\]
such that $\delta s^B_i = \xi \xi' \delta'$ and $\ell(\delta s^B_i) = \ell(\xi \xi') + \ell(\delta')$.
It follows that $\opi^B_\delta \opi^B_i = \opi^B_{\xi\xi'} \opi^B_{\delta'}$.
By \cref{eq: action on gam bulB gam'}, we have
\[
\fkL((\gamma \otimes \gamma') \otimes \opi^B_\delta \cdot \opi^B_i)
= (\gamma \cdot \opi^B_{\xi} \bulB \gamma' \cdot \opi_{\xi'_{-m}}) \cdot \opi^B_{\delta'}
= \fkL((\gamma \otimes \gamma') \otimes \opi^B_\delta) \cdot \opi^B_i,
\]
where $\xi'_{-m} \in \SG_n$ is defined by $\xi'_{-m}(i) = \xi'(i+m)$ for $1 \le i \le n$.
\end{proof}

\begin{corollary}\label{cor: module isom}
The Grothendieck group $\bigoplus_{n \ge 0} \calG_0(\scrP^B(n))$ is isomorphic to $\QSym^B$ as $\QSym$-modules.
\end{corollary}
\begin{proof}
Let $M$ be a $B_m$ poset module and $N$ a module for a poset on $[n]$. By \cref{thm: induction product}, we have $M\boxtimes^B N$ is a $B_{m+n}$ poset module. Combining this with \cref{thm: Grothendieck group of poset modules}, the fact that $\GHBb$ is isomorphic as $\QSym$-modules to $\QSym^B$ (\cite[Theorem 5(i)]{Huang17}), and the fact that $\bigoplus_{n \ge 0} \calG_0(\scrP(n))$ is isomorphic to $\QSym$ as Hopf algebras (\cite[Theorem 3.5]{CKO24}), the result follows.
\end{proof}

As a further corollary, we obtain a generalization of \cref{eq: odot action on QSym^B}, which describes the right $\QSym$-action $\odot^B$ on $\QSym^B$.

\begin{corollary}\label{cor: KBP product}
Let $P_1 = ([-m,m], \preceq_{P_1})$ be a  $B_m$ poset and $P_2 = ([n], \preceq_{P_2})$ a poset. Then 
$ K^B_{P_1} \odot^B K_{P_2} = K^B_{P_1 \sqcup_B P_2}$.
\end{corollary}

\begin{proof}
By combining \cref{thm: induction product} with the fact that $\ch^B$ is a right $\QSym$-module isomorphism, the assertion follows immediately.
\end{proof}

\begin{example}
Let 
$\def \pp {0.6}
P_1 = \begin{array}{l}
\begin{tikzpicture}
\node at (-0.7*\pp, 1*\pp) {$\bullet$};
\node[right] at (-0.7*\pp, 1*\pp) {\scriptsize $1$};
\node at (0.7*\pp, 1*\pp) {$\bullet$};
\node[right] at (0.7*\pp, 1*\pp) {\scriptsize $-2$};
\node at (0*\pp, 0*\pp) {$\bullet$};
\node[right] at (0*\pp, 0*\pp) {\scriptsize $0$};
\node at (-0.7*\pp, -1*\pp) {$\bullet$};
\node[right] at (-0.7*\pp, -1*\pp) {\scriptsize $-1$};
\node at (0.7*\pp, -1*\pp) {$\bullet$};
\node[right] at (0.7*\pp, -1*\pp) {\scriptsize $2$};

\draw (0*\pp, 0*\pp) -- (-0.7*\pp, 1*\pp);
\draw (0*\pp, 0*\pp) -- (0.7*\pp, 1*\pp);
\draw (0*\pp, 0*\pp) -- (-0.7*\pp, -1*\pp);
\draw (0*\pp, 0*\pp) -- (0.7*\pp, -1*\pp);
\end{tikzpicture}
\end{array}$
be a $B_2$ poset and
$\def \hp {1}
\def \vp {0.4}
P_2 = \begin{array}{l}
\begin{tikzpicture}
\node at (-1*\hp, 0*\vp) {$\bullet$};
\node[right] at (-1*\hp, 0*\vp) {\scriptsize $1$};
\end{tikzpicture}
\end{array}$
be a poset on $[1]$. 

Then 
\[\def \pp {0.6}
P_1 \sqcup_B P_2 = 
\begin{array}{l}
\begin{tikzpicture}
\node at (-0.7*\pp, 1*\pp) {$\bullet$};
\node[right] at (-0.7*\pp, 1*\pp) {\scriptsize $1$};
\node at (0.7*\pp, 1*\pp) {$\bullet$};
\node[right] at (0.7*\pp, 1*\pp) {\scriptsize $-2$};
\node at (0*\pp, 0*\pp) {$\bullet$};
\node[right] at (0*\pp, 0*\pp) {\scriptsize $0$};
\node at (-0.7*\pp, -1*\pp) {$\bullet$};
\node[right] at (-0.7*\pp, -1*\pp) {\scriptsize $-1$};
\node at (0.7*\pp, -1*\pp) {$\bullet$};
\node[right] at (0.7*\pp, -1*\pp) {\scriptsize $2$};

\draw (0*\pp, 0*\pp) -- (-0.7*\pp, 1*\pp);
\draw (0*\pp, 0*\pp) -- (0.7*\pp, 1*\pp);
\draw (0*\pp, 0*\pp) -- (-0.7*\pp, -1*\pp);
\draw (0*\pp, 0*\pp) -- (0.7*\pp, -1*\pp);

\node at (3*\pp, 0*\pp) {$\bullet$};
\node[right] at (3*\pp, 0*\pp) {\scriptsize $3$};

\node at (5*\pp, 0*\pp) {$\bullet$};
\node[right] at (5*\pp, 0*\pp) {\scriptsize $-3$};
\end{tikzpicture}
\end{array}.
\]

The $H_3^B(0)$-action on $\C\Sigma_R^B(P_1 \sqcup_B P_2)$ is illustrated in the following figure:

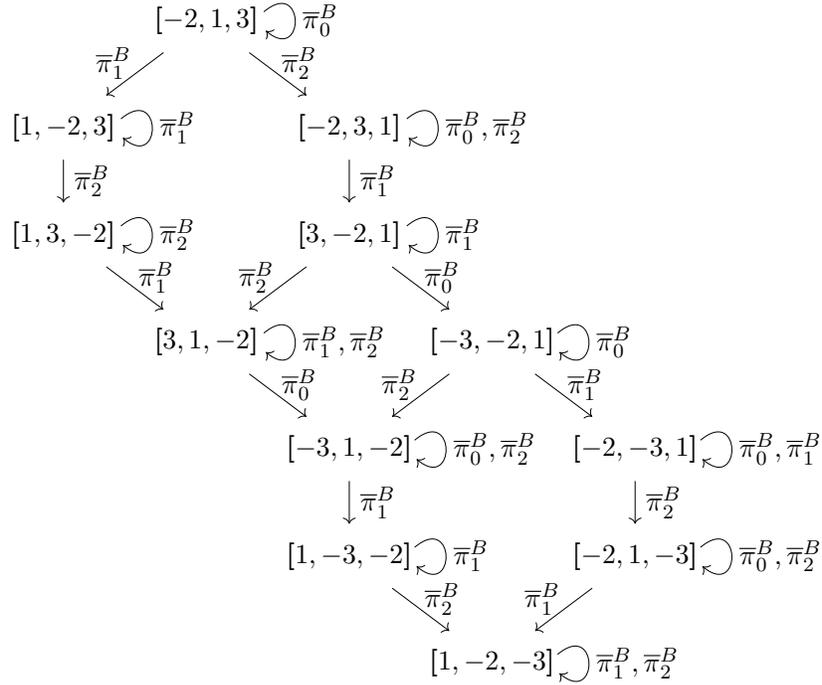
\begin{displaymath}
\def \hp {2}
\def \vp {1.5}
\def \lp {0.55}
\scalebox{0.95}{$
\begin{tikzpicture}

\node at (-1*\hp, 6*\vp) {\small $\bracket{-2,1,3}$};
\node at (-1*\hp + 1.2*\lp , 6*\vp) {} edge [out=40,in=320, loop] ();
\node at (-1*\hp + 2.9*\lp, 6*\vp) {\small $\opi^B_0$};

\draw [->] (-1.3*\hp, 5.7*\vp) -- (-1.7*\hp, 5.3*\vp);
\node at (-1.65*\hp, 5.6*\vp) {\small $\opi^B_1$};

\draw [->] (-0.7*\hp, 5.7*\vp) -- (-0.3*\hp, 5.3*\vp);
\node at (-0.35*\hp, 5.6*\vp) {\small $\opi^B_2$};

\node at (-2*\hp, 5*\vp) {\small $\bracket{1,-2,3}$};
\node at (-2*\hp + 1.2*\lp , 5*\vp) {} edge [out=40,in=320, loop] ();
\node at (-2*\hp + 2.9*\lp, 5*\vp) {\small $\opi^B_1$};

\node at (0*\hp, 5*\vp) {\small $\bracket{-2,3,1}$};
\node at (0*\hp + 1.2*\lp , 5*\vp) {} edge [out=40,in=320, loop] ();
\node at (0*\hp + 3.5*\lp, 5*\vp) {\small $\opi^B_0, \opi^B_2$};

\draw [->] (-2*\hp, 4.7*\vp) -- (-2*\hp, 4.3*\vp);
\node at (-1.8*\hp, 4.5*\vp) {\small $\opi^B_2$};

\draw [->] (0*\hp, 4.7*\vp) -- (0*\hp, 4.3*\vp);
\node at (0.2*\hp, 4.5*\vp) {\small $\opi^B_1$};

\node at (-2*\hp, 4*\vp) {\small $\bracket{1,3,-2}$};
\node at (-2*\hp + 1.2*\lp , 4*\vp) {} edge [out=40,in=320, loop] ();
\node at (-2*\hp + 2.9*\lp, 4*\vp) {\small $\opi^B_2$};

\node at (0*\hp, 4*\vp) {\small $\bracket{3,-2,1}$};
\node at (0*\hp + 1.2*\lp , 4*\vp) {} edge [out=40,in=320, loop] ();
\node at (0*\hp + 2.9*\lp, 4*\vp) {\small $\opi^B_1$};

\draw [->] (-1.7*\hp, 3.7*\vp) -- (-1.3*\hp, 3.3*\vp);
\node at (-1.35*\hp, 3.6*\vp) {\small $\opi^B_1$};

\draw [->] (-0.3*\hp, 3.7*\vp) -- (-0.7*\hp, 3.3*\vp);
\node at (-0.65*\hp, 3.6*\vp) {\small $\opi^B_2$};

\draw [->] (0.3*\hp, 3.7*\vp) -- (0.7*\hp, 3.3*\vp);
\node at (0.65*\hp, 3.6*\vp) {\small $\opi^B_0$};

\node at (-1*\hp, 3*\vp) {\small $\bracket{3,1,-2}$};
\node at (-1*\hp + 1.2*\lp , 3*\vp) {} edge [out=40,in=320, loop] ();
\node at (-1*\hp + 3.5*\lp, 3*\vp) {\small $\opi^B_1, \opi^B_2$};

\node at (1*\hp, 3*\vp) {\small $\bracket{-3,-2,1}$};
\node at (1*\hp + 1.4*\lp , 3*\vp) {} edge [out=40,in=320, loop] ();
\node at (1*\hp + 3.1*\lp, 3*\vp) {\small $\opi^B_0$};

\draw [->] (-0.7*\hp, 2.7*\vp) -- (-0.3*\hp, 2.3*\vp);
\node at (-0.35*\hp, 2.6*\vp) {\small $\opi^B_0$};

\draw [->] (0.7*\hp, 2.7*\vp) -- (0.3*\hp, 2.3*\vp);
\node at (0.35*\hp, 2.6*\vp) {\small $\opi^B_2$};

\draw [->] (1.3*\hp, 2.7*\vp) -- (1.7*\hp, 2.3*\vp);
\node at (1.65*\hp, 2.6*\vp) {\small $\opi^B_1$};

\node at (0*\hp, 2*\vp) {\small $\bracket{-3,1,-2}$};
\node at (0*\hp + 1.4*\lp , 2*\vp) {} edge [out=40,in=320, loop] ();
\node at (0*\hp + 3.7*\lp, 2*\vp) {\small $\opi^B_0, \opi^B_2$};

\node at (2*\hp, 2*\vp) {\small $\bracket{-2,-3,1}$};
\node at (2*\hp + 1.4*\lp , 2*\vp) {} edge [out=40,in=320, loop] ();
\node at (2*\hp + 3.7*\lp, 2*\vp) {\small $\opi^B_0, \opi^B_1$};

\draw [->] (0*\hp, 1.7*\vp) -- (0*\hp, 1.3*\vp);
\node at (0.2*\hp, 1.5*\vp) {\small $\opi^B_1$};

\draw [->] (2*\hp, 1.7*\vp) -- (2*\hp, 1.3*\vp);
\node at (2.2*\hp, 1.5*\vp) {\small $\opi^B_2$};

\node at (0*\hp, 1*\vp) {\small $\bracket{1,-3,-2}$};
\node at (0*\hp + 1.4*\lp , 1*\vp) {} edge [out=40,in=320, loop] ();
\node at (0*\hp + 3.1*\lp, 1*\vp) {\small $\opi^B_1$};

\node at (2*\hp, 1*\vp) {\small $\bracket{-2,1,-3}$};
\node at (2*\hp + 1.4*\lp , 1*\vp) {} edge [out=40,in=320, loop] ();
\node at (2*\hp + 3.7*\lp, 1*\vp) {\small $\opi^B_0, \opi^B_2$};

\draw [->] (0.3*\hp, 0.7*\vp) -- (0.7*\hp, 0.3*\vp);
\node at (0.65*\hp, 0.6*\vp) {\small $\opi^B_2$};

\draw [->] (1.7*\hp, 0.7*\vp) -- (1.3*\hp, 0.3*\vp);
\node at (1.35*\hp, 0.6*\vp) {\small $\opi^B_1$};

\node at (1*\hp, 0*\vp) {\small $\bracket{1,-2,-3}$};
\node at (1*\hp + 1.4*\lp , 0*\vp) {} edge [out=40,in=320, loop] ();
\node at (1*\hp + 3.7*\lp, 0*\vp) {\small $\opi^B_1, \opi^B_2$};

\end{tikzpicture}
$}
\end{displaymath}
In this picture, $\opi^B_i$ acts by $0$ if it is not represented.
\end{example}

\subsection{Restriction of poset modules}

In this subsection, we prove that for any $B_n$ poset $P$, the restriction $M^B_P \downarrow^{H^B_n(0)}_{H^B_m(0) \otimes H_{n-m}(0)}$ is isomorphic to a direct sum of tensor products of a $B_m$ poset module and a poset module for $H_{n-m}(0)$.
Thus, we prove that
\begin{align*}
\updelta^B([M]) := \sum_{0 \le m \le n} 
\left[
M \downarrow_{H^B_m(0) \otimes H_{n-m}(0)}^{H^B_{n}(0)}
\right]
\end{align*}
defines a right $\bigoplus_{n \ge 0} \calG_0(\scrP(n))$-coaction on $\bigoplus_{n \ge 0} \calG_0(\scrP^B(n))$.

Let $P$ be a $B_n$ poset.
A subset $Q$ of $P$ is called an \emph{induced subposet} (simply, \emph{subposet}) of $P$ if $Q$ is equipped with the partial order $\preceq_Q$ defined by $x \preceq_Q y$ if and only if $x \preceq_P y$ for all $x,y\in Q$.
Let $Q$ be a subposet of $P$, and let $P \setminus Q$ denote the subposet of $P$ with underlying set $P \setminus Q$.
Define another subposet $\tQ$ of $P$ by 
\[
\tQ := \{y \in P \mid \text{$y \preceq_P x$ for some $x \in Q$}\}.
\]
Letting $Q = \{x_1 < x_2 < \cdots < x_k\}$, we define $\rmst(Q)$ to be the poset $([k], \preceq_{\rmst(Q)})$ such that for all $i,j \in [k]$, $i \preceq_{\rmst(Q)} j$ if and only if $x_i \preceq_Q x_j$. 
We call a subposet $Q$ of $P$ a \emph{type-$B$ subposet} if 
\begin{enumerate}
    \item $0\in Q$, and
    \item $x \in Q$ implies $-x \in Q$.
\end{enumerate}
For a type-$B$ subposet $Q$ of $P$ with $2m + 1$ elements, we will denote the elements of $Q$ by $x^Q_i$, with $i\in [-m,m]$, such that 
\begin{align}\label{eq: elements of Q}
x^Q_{-m} < \cdots < x^Q_{-1} < x^Q_0 = 0 < x^Q_1 < \cdots < x^Q_m
\end{align}
(as integers), unless otherwise stated.
We define $\stB(Q)$ to be the $B_m$ poset with underlying set $([-m,m], \preceq_{\stB(Q)})$, such that $i \preceq_{\stB(Q)} j$ if and only if $x^Q_i \preceq_Q x^Q_j$ for $i,j \in [-m, m]$.

\begin{definition}
Let $P$ be a $B_n$ poset. 
\begin{enumerate}[label={\rm (\arabic*)}, itemsep = 1ex]
\item 
A type-$B$ subposet $Q$ of $P$ is called a \emph{type-$B$ lower subposet of $P$} if for all $x \in Q$ and $y \in P$ satisfying $-x \preceq_P y \preceq_P x$, we have $y \in Q$.

\item
Let $Q$ be a type-$B$ lower subposet of $P$.  
A subposet $U$ of $P \setminus \tQ$ is called an \emph{upper subposet of $P$ over $Q$} if the following conditions hold:
\begin{enumerate}[label={\rm (\roman*)}]
\item For each $x \in P \setminus \tQ$, there exists a unique $y \in U$ such that $|x| = |y|$. 
\item If $x \in P \setminus \tQ$, $y\in U$, and $y\preceq_P x$, then $x\in U$.
\end{enumerate}
\end{enumerate}
\end{definition}

Let $P$ be a $B_n$ poset.
For $m \in [0,n]$, we denote by $\LS^B(P;m)$ the set of all type-$B$ lower subposets of $P$ with $2m+1$ elements.
For $Q \in \LS^B(P;m)$, let
$\upper(Q)$ be the set of all upper subposets of $P$ over $Q$.
For $U \in \upper(Q)$, we will denote the elements of $U$ by $y^U_j$, with $j\in [n-m]$, such that 
\begin{align}\label{eq: elements of U}
y^U_1 < y^U_2 < \cdots < y^U_{n-m}
\end{align}
(as integers).

Given $\gamma \in \SG^B_n$ and $m \in [0,n]$, let $\rmst^+_{m}(\gamma)$ be the signed permutation in $\SG^B_{m}$ defined by
\begin{align}\label{eq: st plus}
\rmst^+_{m}(\gamma)(i) 
& := 
\begin{cases}
\sign(\gamma(i))
\left| \big\{
j \in [m] \; \middle| \;
|\gamma(j)| \le |\gamma(i)| 
\big\} \right|
& \text{if $i \in [-m,m] \setminus \{0\}$,} \\
0 & \text{if $i = 0$,}
\end{cases}
\end{align}
where $\sign(\gamma(i)) := \gamma(i) /|\gamma(i)|$, and let $\rmst^-_{n-m}(\gamma)$ be the permutation in $\SG_{n-m}$ defined by
\begin{align}\label{eq: st minus}
\rmst^-_{n-m}(\gamma)(i) 
& := |\{j \in [m+1, n] \mid \gamma(j) \le \gamma(i+m) \}|
\quad \text{for all $1 \le i \le n-m$.}
\end{align}
For example, if $\gamma = [-4, 7, -1, 3, -6, 2, -5]$ then $\rmst_4^+(\gamma) = [-3, 4, -1, 2]$ and $\rmst^-_3(\gamma) = 132$. 
Given a $B_n$ poset $P$ and $\gamma \in \Sigma^B_R(P)$, we define subposets $Q_\gamma$ and $U_\gamma$ of $P$ by
\[
Q_\gamma := \{\gamma(i) \mid i \in [-m,m]\}
\quad \text{and} \quad
U_\gamma := \{\gamma(i) \mid i \in [m+1, n]\}.
\]
By definition, $Q_\gamma$ is a type-$B$ subposet of $P$.

\begin{example}\label{eg: decomposition of Bn poset}
Let $
\def \hp {1}
\def \vp {0.4}
P = 
\begin{array}{l}
\begin{tikzpicture}
\node at (-1*\hp, 1.5*\vp) {$\bullet$};
\node[right] at (-1*\hp, 1.5*\vp) {\scriptsize $2$};
\node at (-1*\hp, 0*\vp) {$\bullet$};
\node[right] at (-1*\hp, 0*\vp) {\scriptsize $0$};
\node at (-1*\hp, -1.5*\vp) {$\bullet$};
\node[right] at (-1*\hp, -1.5*\vp) {\scriptsize $-2$};
\draw[] (-1*\hp, 1.5*\vp) -- (-1*\hp, -1.5*\vp);

\node at (0*\hp, 0.75*\vp) {$\bullet$};
\node[right] at (0*\hp, 0.75*\vp) {\scriptsize $1$};
\node at (0*\hp, -0.75*\vp) {$\bullet$};
\node[right] at (0*\hp, -0.75*\vp) {\scriptsize $3$};
\draw[] (0*\hp, 0.75*\vp) -- (0*\hp, -0.75*\vp);

\node at (1*\hp, 0.75*\vp) {$\bullet$};
\node[right] at (1*\hp, 0.75*\vp) {\scriptsize $-3$};
\node at (1*\hp, -0.75*\vp) {$\bullet$};
\node[right] at (1*\hp, -0.75*\vp) {\scriptsize $-1$};
\draw[] (1*\hp, 0.75*\vp) -- (1*\hp, -0.75*\vp);
\end{tikzpicture}
\end{array}$ and let $m=1$.

We have
\begin{displaymath}
\def \hp {1}
\def \vp {0.4}
\LS^B(P;1) = \left\{
Q_1 := \begin{array}{l}
\begin{tikzpicture}
\node at (-1*\hp, 0*\vp) {$\bullet$};
\node[right] at (-1*\hp, 0*\vp) {\scriptsize $0$};

\node at (0*\hp, 0.75*\vp) {$\bullet$};
\node[right] at (0*\hp, 0.75*\vp) {\scriptsize $1$};

\node at (1*\hp, -0.75*\vp) {$\bullet$};
\node[right] at (1*\hp, -0.75*\vp) {\scriptsize $-1$};
\end{tikzpicture}
\end{array},
\quad
Q_2 := 
\begin{array}{l}
\begin{tikzpicture}
\node at (-1*\hp, 1.5*\vp) {$\bullet$};
\node[right] at (-1*\hp, 1.5*\vp) {\scriptsize $2$};
\node at (-1*\hp, 0*\vp) {$\bullet$};
\node[right] at (-1*\hp, 0*\vp) {\scriptsize $0$};
\node at (-1*\hp, -1.5*\vp) {$\bullet$};
\node[right] at (-1*\hp, -1.5*\vp) {\scriptsize $-2$};
\draw[] (-1*\hp, 1.5*\vp) -- (-1*\hp, -1.5*\vp);
\end{tikzpicture}
\end{array},
\quad 
Q_3 := 
\begin{array}{l}
\begin{tikzpicture}
\node at (-1*\hp, 0*\vp) {$\bullet$};
\node[right] at (-1*\hp, 0*\vp) {\scriptsize $0$};

\node at (0*\hp, -0.75*\vp) {$\bullet$};
\node[right] at (0*\hp, -0.75*\vp) {\scriptsize $3$};

\node at (1*\hp, 0.75*\vp) {$\bullet$};
\node[right] at (1*\hp, 0.75*\vp) {\scriptsize $-3$};
\end{tikzpicture}
\end{array}
\right\}.
\end{displaymath}
Moreover,
\begin{displaymath}
\def \hp {1}
\def \vp {0.4}
\tQ_1 = \begin{array}{l}
\begin{tikzpicture}
\node at (-1*\hp, 0*\vp) {$\bullet$};
\node[right] at (-1*\hp, 0*\vp) {\scriptsize $0$};
\node at (-1*\hp, -1.5*\vp) {$\bullet$};
\node[right] at (-1*\hp, -1.5*\vp) {\scriptsize $-2$};
\draw[] (-1*\hp, 0*\vp) -- (-1*\hp, -1.5*\vp);

\node at (0*\hp, 0.75*\vp) {$\bullet$};
\node[right] at (0*\hp, 0.75*\vp) {\scriptsize $1$};
\node at (0*\hp, -0.75*\vp) {$\bullet$};
\node[right] at (0*\hp, -0.75*\vp) {\scriptsize $3$};
\draw[] (0*\hp, 0.75*\vp) -- (0*\hp, -0.75*\vp);

\node at (1*\hp, -0.75*\vp) {$\bullet$};
\node[right] at (1*\hp, -0.75*\vp) {\scriptsize $-1$};
\end{tikzpicture}
\end{array},
\quad
\tQ_2 = 
\begin{array}{l}
\begin{tikzpicture}
\node at (-1*\hp, 1.5*\vp) {$\bullet$};
\node[right] at (-1*\hp, 1.5*\vp) {\scriptsize $2$};
\node at (-1*\hp, 0*\vp) {$\bullet$};
\node[right] at (-1*\hp, 0*\vp) {\scriptsize $0$};
\node at (-1*\hp, -1.5*\vp) {$\bullet$};
\node[right] at (-1*\hp, -1.5*\vp) {\scriptsize $-2$};
\draw[] (-1*\hp, 1.5*\vp) -- (-1*\hp, -1.5*\vp);
\end{tikzpicture}
\end{array},
\quad 
\tQ_3 = 
\begin{array}{l}
\begin{tikzpicture}
\node at (-1*\hp, 0*\vp) {$\bullet$};
\node[right] at (-1*\hp, 0*\vp) {\scriptsize $0$};
\node at (-1*\hp, -1.5*\vp) {$\bullet$};
\node[right] at (-1*\hp, -1.5*\vp) {\scriptsize $-2$};
\draw[] (-1*\hp, 0*\vp) -- (-1*\hp, -1.5*\vp);

\node at (0*\hp, -0.75*\vp) {$\bullet$};
\node[right] at (0*\hp, -0.75*\vp) {\scriptsize $3$};

\node at (1*\hp, 0.75*\vp) {$\bullet$};
\node[right] at (1*\hp, 0.75*\vp) {\scriptsize $-3$};
\node at (1*\hp, -0.75*\vp) {$\bullet$};
\node[right] at (1*\hp, -0.75*\vp) {\scriptsize $-1$};
\draw[] (1*\hp, 0.75*\vp) -- (1*\hp, -0.75*\vp);
\end{tikzpicture}
\end{array},
\end{displaymath}
and
\begin{displaymath}
\def \hp {1}
\def \vp {0.4}
\begin{array}{rl}
\upper(Q_1) 
& = \left\{
U_{1,1} := 
\begin{array}{l}
\begin{tikzpicture}
\node at (0*\hp, 0*\vp) {$\bullet$};
\node[right] at (0*\hp, 0*\vp) {\scriptsize $2$};

\node at (1*\hp, 0*\vp) {$\bullet$};
\node[right] at (1*\hp, 0*\vp) {\scriptsize $-3$};
\end{tikzpicture}
\end{array}
\right\}, 
\\[1ex]

\upper(Q_2) 
& = \left\{
U_{2,1} := 
\begin{array}{l}
\begin{tikzpicture}
\node at (0*\hp, 0.75*\vp) {$\bullet$};
\node[right] at (0*\hp, 0.75*\vp) {\scriptsize $1$};
\node at (0*\hp, -0.75*\vp) {$\bullet$};
\node[right] at (0*\hp, -0.75*\vp) {\scriptsize $3$};
\draw[] (0*\hp, 0.75*\vp) -- (0*\hp, -0.75*\vp);
\end{tikzpicture}
\end{array}
, \quad
U_{2,2} := 
\begin{array}{l}
\begin{tikzpicture}
\node at (0*\hp, 0.75*\vp) {$\bullet$};
\node[right] at (0*\hp, 0.75*\vp) {\scriptsize $-3$};
\node at (0*\hp, -0.75*\vp) {$\bullet$};
\node[right] at (0*\hp, -0.75*\vp) {\scriptsize $-1$};
\draw[] (0*\hp, 0.75*\vp) -- (0*\hp, -0.75*\vp);
\end{tikzpicture}
\end{array}
, \quad
U_{2,3} := 
\begin{array}{l}
\begin{tikzpicture}
\node at (0*\hp, 0*\vp) {$\bullet$};
\node[right] at (0*\hp, 0*\vp) {\scriptsize $1$};
\node at (1*\hp, 0*\vp) {$\bullet$};
\node[right] at (1*\hp, 0*\vp) {\scriptsize $-3$};
\end{tikzpicture}
\end{array}
\right\}, 
\\[3ex]

\upper(Q_3) 
& = \left\{
U_{3,1} := 
\begin{array}{l}
\begin{tikzpicture}
\node at (0*\hp, 0*\vp) {$\bullet$};
\node[right] at (0*\hp, 0*\vp) {\scriptsize $2$};
\node at (1*\hp, 0*\vp) {$\bullet$};
\node[right] at (1*\hp, 0*\vp) {\scriptsize $1$};
\end{tikzpicture}
\end{array}
\right\}.
\end{array}
\end{displaymath}

We also have

\begin{displaymath}
\def \hp {1}
\def \vp {0.4}
\stB(Q_2) = 
\begin{array}{l}
\begin{tikzpicture}
\node at (-1*\hp, 1.5*\vp) {$\bullet$};
\node[right] at (-1*\hp, 1.5*\vp) {\scriptsize $1$};
\node at (-1*\hp, 0*\vp) {$\bullet$};
\node[right] at (-1*\hp, 0*\vp) {\scriptsize $0$};
\node at (-1*\hp, -1.5*\vp) {$\bullet$};
\node[right] at (-1*\hp, -1.5*\vp) {\scriptsize $-1$};
\draw[] (-1*\hp, 1.5*\vp) -- (-1*\hp, -1.5*\vp);
\end{tikzpicture}
\end{array} \,\,\,
\mbox{ and } \,\,\, 
\rmst(U_{2,2}) = \begin{array}{l}\begin{tikzpicture}
\node at (0*\hp, 0.75*\vp) {$\bullet$};
\node[right] at (0*\hp, 0.75*\vp) {\scriptsize $1$};
\node at (0*\hp, -0.75*\vp) {$\bullet$};
\node[right] at (0*\hp, -0.75*\vp) {\scriptsize $2$};
\draw[] (0*\hp, 0.75*\vp) -- (0*\hp, -0.75*\vp);
\end{tikzpicture}
\end{array}.
\end{displaymath}

For $\gamma = [2,-1,-3]\in \Sigma_R^B(P)$, we have $Q_{\gamma} = Q_2$ and $U_{\gamma} = U_{2,2}$. Moreover, $(x^{Q_2}_{-1}, x^{Q_2}_{0}, x^{Q_2}_{1}) = (-2,0,2)$ and $(y^{U_{2,2}}_1, y^{U_{2,2}}_2) = (-3,-1)$.
\end{example}

\begin{lemma}\label{lem: bijection for restriction}
Let $P$ be a $B_n$ poset and let $m\in [0,n]$. 
The map 
\[
\calF_{P;m}:
\Sigma^B_R(P) \ra  
\bigsqcup_{Q \in \LS^B(P;m)}
\bigsqcup_{U \in \upper(Q)}
\Sigma^B_R(\stB(Q)) \times \Sigma_R(\rmst(U)),
\]
sending $\gamma \mapsto (\rmst^+_{m}(\gamma), \rmst^-_{n-m}(\gamma)) \in \Sigma^B_R(\stB(Q_\gamma)) \times \Sigma_R(\rmst(U_\gamma))$, is a bijection.
\end{lemma}

\begin{proof}
We first show that $\calF_{P;m}$ is well-defined. Specifically, we will show that for every $\gamma \in \Sigma^B_R(P)$, we have $Q_\gamma \in \LS^B(P;m)$, $U_\gamma \in \upper(Q)$, $\rmst^+_{m}(\gamma) \in \Sigma^B_R(\stB(Q_\gamma))$, and $\rmst^-_{n-m}(\gamma) \in \Sigma_R(\rmst(U_\gamma))$.

Take any $\gamma \in \Sigma^B_R(P)$.
For any $x \in Q_\gamma$ and $y \in P$ satisfying $-x \preceq_P y \preceq_P x$, we have $\gamma^{-1}(-x) \le \gamma^{-1}(y) \le \gamma^{-1}(x)$, so $y \in \{\gamma(i) \mid i \in [-m,m]\} = Q_\gamma$.
Hence $Q_\gamma \in \LS^B(P;m)$.

From the definitions of $Q_\gamma$ and $U_\gamma$, we see that $\gamma^{-1}(x) \in [-m,m]$ and $\gamma^{-1}(y) \in [m+1, m+n]$ for all $x \in Q_\gamma$ and $y \in U_\gamma$.
Since $\gamma \in \Sigma^B_R(P)$, we have $y \not\preceq_P x$ for all $x \in Q_\gamma$ and $y \in U_\gamma$, thus $U_\gamma \subseteq P \setminus \widetilde{Q_\gamma}$.
Again from the definitions of $Q_\gamma$ and $U_\gamma$, we see that for each $x \in P \setminus \widetilde{Q_\gamma}$, there exists a unique $y \in U_\gamma$ such that $|x| = |y|$.
In addition, for $x,y \in P \setminus \widetilde{Q_\gamma}$ with $x \preceq_P y$, if $x \in U_\gamma$, then $m+1 \le \gamma^{-1}(x) \le \gamma^{-1}(y)$ and so $y \in U_\gamma$.
Putting these all together, we have $U_\gamma \in \upper(Q)$.

Let us show that $\rmst^+_{m}(\gamma) \in \Sigma^B_R(\stB(Q_\gamma))$ and $\rmst^-_{n-m}(\gamma) \in \Sigma_R(\rmst(U_\gamma))$.
From \cref{eq: elements of Q} and \cref{eq: elements of U}, recall that we use the notation $x^{Q_\gamma}_{-m}, \ldots, x^{Q_\gamma}_{m}$ and $y^{U_\gamma}_1, \ldots, y^{U_\gamma}_{n-m}$ to denote the elements of $Q_\gamma$ and $U_\gamma$, respectively.
Note that
\begin{align*}
\rmst^+_{m}(\gamma)(\gamma^{-1}(x^{Q_\gamma}_i)) = 
\sign(i) \left| \left\{
p \in [m] \; \middle| \;
|\gamma(p)| \le |x^{Q_\gamma}_i| 
\right\} \right| = i \quad \text{for $i \in [-m,m]$}
\end{align*}
and
\begin{align*}
\rmst^-_{n-m}(\gamma)(\gamma^{-1}(y^{U_\gamma}_i)) = 
\left| \left\{
p \in [m+1,n] \; \middle| \;
\gamma(p) \le y^{U_\gamma}_i 
\right\} \right| = i \quad \text{for $i \in [n-m]$}.
\end{align*}
For any $i,j \in [-m,m]$ with $i \preceq_{\stB(Q_\gamma)} j$, we have $\gamma^{-1}(x^{Q_\gamma}_i) \le \gamma^{-1}(x^{Q_\gamma}_j)$.
It follows that
\[
\rmst^+_{m}(\gamma)^{-1}(i) 
= \gamma^{-1}(x^{Q_\gamma}_i) 
\le \gamma^{-1}(x^{Q_\gamma}_j)
= \rmst^+_{m}(\gamma)^{-1}(j).
\]
Similarly, for any $i,j \in [n-m]$ with $i \preceq_{\rmst(U_\gamma)} j$, we have $\gamma^{-1}(y^{U_\gamma}_i) \le \gamma^{-1}(y^{U_\gamma}_j)$, which implies
\[
\rmst^-_{n-m}(\gamma)^{-1}(i) 
= \gamma^{-1}(y^{U_\gamma}_i) - m  
\le \gamma^{-1}(y^{U_\gamma}_j) - m
= \rmst^-_{n-m}(\gamma)^{-1}(j).
\]
Consequently, $\rmst^+_{m}(\gamma) \in \Sigma^B_R(\stB(Q_\gamma))$ and $\rmst^-_{n-m}(\gamma) \in \Sigma_R(\rmst(U_\gamma))$, therefore $\calF_{P;m}$ is well-defined.

We next find the inverse of $\calF_{P;m}$.
Given $Q \in \LS^B(P;m)$, $U \in \upper(Q)$, $\gamma_1 \in \Sigma^B_R(\stB(Q))$, and $\gamma_2 \in \Sigma_R(\rmst(U))$, let
$\conc(\gamma_1, \gamma_2)$ be the signed permutation in $\SG^B_n$ defined by
\[
\conc(\gamma_1, \gamma_2)(i) = \begin{cases}
-y^{U}_{\gamma_2(-i-m)} & \text{for $i \in [-n, -m-1]$,}\\
x^{Q}_{\gamma_1(i)} & \text{for $i \in [-m, m]$,} \\
y^{U}_{\gamma_2(i-m)} & \text{for $i \in [m+1,n]$.}
\end{cases}
\]
See \cref{eg: decomposition of Bn poset 2} for an example. 
Then we have that $\conc(\gamma_1, \gamma_2)([-n, -m-1]) = -U$, $\conc(\gamma_1, \gamma_2)([-m, m]) = Q$,  and 
$\conc(\gamma_1, \gamma_2)([m+1, n]) = U$, where $-U := \{-x \mid x \in U\}$.
Define
\[
\calF'_{P;m}:
\bigsqcup_{Q \in \LS^B(P;m)}
\bigsqcup_{U \in \upper(Q)}
\Sigma^B_R(\stB(Q)) \times \Sigma_R(\rmst(U))
\ra
\Sigma^B_R(P),
\quad
(\gamma_1, \gamma_2) \mapsto \conc(\gamma_1,\gamma_2)
\]
for $Q \in \LS^B(P;m)$, $U \in \upper(Q)$, $\gamma_1 \in \Sigma^B_R(\stB(Q))$, and  $\gamma_2 \in \Sigma_R(\rmst(U))$.

We claim that $\calF'_{P;m}$ is well-defined, i.e., that $\conc(\gamma_1,\gamma_2)\in \Sigma^B_R(P)$.
Take any $Q \in \LS^B(P;m)$, $U \in \upper(Q)$, $\gamma_1 \in \Sigma^B_R(\stB(Q))$, and $\gamma_2 \in \Sigma_R(\rmst(U))$.
It suffices to show that 
\[
\conc(\gamma_1, \gamma_2)^{-1}(a) \le \conc(\gamma_1, \gamma_2)^{-1}(b)
\quad
\text{for any $a,b \in [-n, n]$ with $a \preceq_P b$.}
\]
Choose any $a,b \in [-n, n]$ with $a \preceq_P b$.
We will deal with the following three cases
\[
\text{(i) } a \in -U, \quad 
\text{(ii) }a \in Q, \quad \text{and} \quad 
\text{(iii) } a \in U.
\]

Suppose that $a \in -U$.
If $b \in Q \cup U$, then $\conc(\gamma_1, \gamma_2)^{-1}(a) <-m \le \conc(\gamma_1, \gamma_2)^{-1}(b)$.
Assume that $b \in -U$.
Since $a \preceq_P b$ and $-a,-b \in U$, we have $-b \preceq_U -a$.
Let $i_a, i_b \in [n-m]$ such that $y^U_{i_a} = -a$ and $y^U_{i_b} = -b$.
The relation $y^U_{i_b} \preceq_U y^U_{i_a}$ implies $i_b \preceq_{\rmst(U)} i_a$.
Since $\gamma_2 \in \Sigma_R(\rmst(U))$, we have $\gamma_2^{-1}(i_b) \le \gamma_2^{-1}(i_a)$.
One can see that
\[
\conc(\gamma_1, \gamma_2)^{-1}(-a)
= \gamma^{-1}_2(i_a) + m
\quad \text{and} \quad
\conc(\gamma_1, \gamma_2)^{-1}(-b) 
= \gamma^{-1}_2(i_b) + m.
\]
Putting these observations together, we have $\conc(\gamma_1, \gamma_2)^{-1}(-b) \le \conc(\gamma_1, \gamma_2)^{-1}(-a)$, which is equivalent to $\conc(\gamma_1, \gamma_2)^{-1}(a) \le \conc(\gamma_1, \gamma_2)^{-1}(b)$.

Suppose that $a \in Q$.
If $b \in U$, then $
\conc(\gamma_1, \gamma_2)^{-1}(a) \le m < \conc(\gamma_1, \gamma_2)^{-1}(b)$.
If $b \in -U$, then $-b \in \tQ$ because $-b \preceq_P -a$ and $-a \in Q$.
This leads to a contradiction since $-b \in U \subseteq P \setminus \tQ$.
Assume that $b \in Q$.
Let $i_a, i_b \in [-m,m]$ such that $x^Q_{i_a} = a$ and $x^Q_{i_b} = b$.
The relation $x^Q_{i_a} \preceq_Q x^Q_{i_b}$ implies that $i_a \preceq_{\stB(Q)} i_b$.
Since $\gamma_1 \in \Sigma^B_R(\stB(Q))$, we have $\gamma_1^{-1}(i_a) \le \gamma_1^{-1}(i_b)$.
One can see that
\[
\conc(\gamma_1, \gamma_2)^{-1}(a)
= \gamma^{-1}_1(i_a)
\quad \text{and} \quad
\conc(\gamma_1, \gamma_2)^{-1}(b)
= \gamma^{-1}_1(i_b).
\]
Thus, $\conc(\gamma_1, \gamma_2)^{-1}(a) \le \conc(\gamma_1, \gamma_2)^{-1}(b)$.

Suppose that $a \in U$.
Since $a \preceq_P b$ and $a \in P \setminus \tQ$, we have $b \in P \setminus \tQ$.
So, by the definition of upper subposets, we have $b \in U$.
Let $i_a, i_b \in [n-m]$ such that $y^U_{i_a} = a$ and $y^U_{i_b} = b$.
The relation $y^U_{i_a} \preceq_U y^U_{i_b}$ implies $i_a \preceq_{\rmst(U)} i_b$.
Since $\gamma_2 \in \Sigma_R(\rmst(U))$, we have $\gamma_2^{-1}(i_a) \le \gamma_2^{-1}(i_b)$.
One can see that
\[
\conc(\gamma_1, \gamma_2)^{-1}(a)
= \gamma^{-1}_2(i_a) + m
\quad \text{and} \quad
\conc(\gamma_1, \gamma_2)^{-1}(b) 
= \gamma^{-1}_2(i_b) + m.
\]
Thus, $\conc(\gamma_1, \gamma_2)^{-1}(a) \le \conc(\gamma_1, \gamma_2)^{-1}(b)$.

Finally, we show that $\calF'_{P;m}$ is the inverse of $\calF_{P;m}$.
To prove that $\calF'_{P;m}$ is a left inverse of $\calF_{P;m}$, let $\gamma \in \Sigma^B_R(P)$ and $i \in [n]$.
Then
\begin{align*}
(\calF'_{P;m} \circ \calF_{P;m}) (\gamma) (i)
& = \begin{cases}
x^{Q_\gamma}_{\rmst^+_{m}(\gamma)(i)} & \text{if $i \in [m]$,} \\[1ex]
y^{U_\gamma}_{\rmst^-_{n-m}(\gamma)(i-m)} & \text{if $i \in [m+1, n]$.}
\end{cases}
\end{align*}
In the case where $i \in [m]$, let $k_i \in [-m,m]$ such that $\gamma(i) = x^{Q_\gamma}_{k_i}$.
Then we have 
\[
\rmst^+_{m}(\gamma)(i) = \sign(x^{Q_\gamma}_{k_i}) \left| \left\{
j \in [m] \; \middle| \; |\gamma(j)| \le |x^{Q_\gamma}_{k_i}|
\right\} \right| 
= k_i.
\]
In the case where $i \in [m+1, n]$, let $k_i \in [n-m]$ such that $\gamma(i) = y^{U_\gamma}_{k_i}$.
Then we have
\[
\rmst^-_{n-m}(\gamma)(i-m) 
= \left| \left\{
j \in [m+1,n] \; \middle| \; \gamma(j) \le y^{U_\gamma}_{k_i}
\right\} \right|
= k_i.
\]
Thus, $(\calF'_{P;m} \circ \calF_{P;m})  (\gamma) (i) = \gamma(i)$.

To prove that $\calF'_{P;m}$ is a right inverse of $\calF_{P;m}$, let us choose $Q \in \LS^B(P;m)$, $U \in \upper(Q)$, $\gamma_1 \in \Sigma^B_R(\stB(Q))$, and $\gamma_2 \in \Sigma_R(\rmst(U))$.
Note that
\[
(\calF_{P;m} \circ \calF'_{P;m}) ((\gamma_1, \gamma_2))
= \left(\rmst^+_{m}(\conc(\gamma_1, \gamma_2)), \, \rmst^-_{n-m}(\conc(\gamma_1, \gamma_2))\right).
\]
By definition, for $i \in [m]$ we have $\conc(\gamma_1, \gamma_2)(i)
= x^Q_{\gamma_1(i)}$, so
\begin{align*}
\rmst^+_{m}(\conc(\gamma_1, \gamma_2)) (i)
& = \sign(x^Q_{\gamma_1(i)})
\left| \left\{
j \in [m] \; \middle| \; |x^Q_{\gamma_1(j)}| \le |x^Q_{\gamma_1(i)}| 
\right\} \right| = \gamma_1(i).
\end{align*}
Similarly, for $i \in [n-m]$ we have $\conc(\gamma_1, \gamma_2)(i+m) = y^U_{\gamma_2(i)}$, so
\begin{align*}
\rmst^-_{n-m}(\conc(\gamma_1, \gamma_2)) (i) 
=  |\{j \in [n-m] \mid y^U_{\gamma_2(j)} \le y^U_{\gamma_2(i)} \}| 
= \gamma_2(i).
\end{align*}
Putting these all together shows that $(\calF_{P;m} \circ \calF'_{P;m}) ((\gamma_1, \gamma_2)) = (\gamma_1, \gamma_2)$. Hence $\calF'_{P;m}$ is the inverse of $\calF_{P;m}$, and $\calF_{P;m}$ is a bijection.
\end{proof}

\begin{example}\label{eg: decomposition of Bn poset 2}
As in \cref{eg: decomposition of Bn poset}, let $
\def \hp {1}
\def \vp {0.4}
P = 
\begin{array}{l}
\begin{tikzpicture}
\node at (-1*\hp, 1.5*\vp) {$\bullet$};
\node[right] at (-1*\hp, 1.5*\vp) {\scriptsize $2$};
\node at (-1*\hp, 0*\vp) {$\bullet$};
\node[right] at (-1*\hp, 0*\vp) {\scriptsize $0$};
\node at (-1*\hp, -1.5*\vp) {$\bullet$};
\node[right] at (-1*\hp, -1.5*\vp) {\scriptsize $-2$};
\draw[] (-1*\hp, 1.5*\vp) -- (-1*\hp, -1.5*\vp);

\node at (0*\hp, 0.75*\vp) {$\bullet$};
\node[right] at (0*\hp, 0.75*\vp) {\scriptsize $1$};
\node at (0*\hp, -0.75*\vp) {$\bullet$};
\node[right] at (0*\hp, -0.75*\vp) {\scriptsize $3$};
\draw[] (0*\hp, 0.75*\vp) -- (0*\hp, -0.75*\vp);

\node at (1*\hp, 0.75*\vp) {$\bullet$};
\node[right] at (1*\hp, 0.75*\vp) {\scriptsize $-3$};
\node at (1*\hp, -0.75*\vp) {$\bullet$};
\node[right] at (1*\hp, -0.75*\vp) {\scriptsize $-1$};
\draw[] (1*\hp, 0.75*\vp) -- (1*\hp, -0.75*\vp);
\end{tikzpicture}
\end{array}$ and let $m=1$.

Recall that
\begin{displaymath}
\def \hp {1}
\def \vp {0.4}
Q_1 = 
\begin{array}{l}
\begin{tikzpicture}
\node at (-1*\hp, 0*\vp) {$\bullet$};
\node[right] at (-1*\hp, 0*\vp) {\scriptsize $0$};

\node at (0*\hp, 0.75*\vp) {$\bullet$};
\node[right] at (0*\hp, 0.75*\vp) {\scriptsize $1$};

\node at (1*\hp, -0.75*\vp) {$\bullet$};
\node[right] at (1*\hp, -0.75*\vp) {\scriptsize $-1$};
\end{tikzpicture}
\end{array} \in \LS^B(P;1) \,\,\,
\mbox{ and } \,\,\, 
U_{1,1} = 
\begin{array}{l}
\begin{tikzpicture}
\node at (0*\hp, 0*\vp) {$\bullet$};
\node[right] at (0*\hp, 0*\vp) {\scriptsize $2$};

\node at (1*\hp, 0*\vp) {$\bullet$};
\node[right] at (1*\hp, 0*\vp) {\scriptsize $-3$};
\end{tikzpicture}
\end{array} \in \upper(Q_1).
\end{displaymath}
We have
\[
\Sigma^B_R(\stB(Q_1)) = \{[1],\ [-1]\}
\quad \text{and} \quad
\Sigma_R(\rmst(U_{1,1})) = \{12, 21\}. 
\]
Let $\gamma = \bracket{-1,2,-3}\in \Sigma_R^B(P)$. Then $\rmst^+_{1}(\gamma) = [-1]$, $\rmst^-_{2}(\gamma) = 21$, $Q_{\gamma} = Q_1$, and $U_{\gamma} = U_{1,1}$.
So,
\[
\calF_{P;1}(\bracket{-1,2,-3}) = ([-1],21) \in \Sigma^B_R(\stB(Q_1)) \times \Sigma_R(\rmst(U_{1,1})).
\]
Note that $(x^{Q_1}_{-1}, x^{Q_1}_{0}, x^{Q_1}_{1}) = (-1,0,1)$ and $(y^{U_{1,1}}_1, y^{U_{1,1}}_2) = (-3,2)$. Therefore, we have $\conc([-1], 21) = [-1,2,-3]$, and so
\[
\calF'_{P;1}([-1], 21) = [-1,2,-3] \in \Sigma^B_R(P).
\]
\end{example}

We are now ready to prove the main result of this subsection.

\begin{theorem}\label{thm: restriction}
Let $P$ be a $B_n$ poset and let $m\in[0,n]$.
Then
\[
M^B_P \downarrow^{H^B_n(0)}_{H^B_m(0) \otimes H_{n-m}(0)} 
\cong
\bigoplus_{Q \in \LS^B(P;m)}
\bigoplus_{U \in \upper(Q)}
M^B_{\stB(Q)} \otimes M_{\rmst(U)}.
\]
\end{theorem}

\begin{proof}
Let 
$\tcalF_{P;m} : M^B_{P} \ra 
\bigoplus_{Q \in \LS^B(P;m)}
\bigoplus_{U \in \upper(Q)} 
M^B_{\stB(Q)} \otimes M_{\rmst(U)}$ be the $\C$-linear map sending $\gamma\mapsto (\rmst^+_m(\gamma)\otimes \rmst^-_{n-m}(\gamma))$ for $\gamma \in \Sigma^B_R(P)$. 
By \cref{lem: bijection for restriction}, $\tcalF_{P;m}$ is a $\C$-linear isomorphism.
To prove that $\tcalF_{P;m}$ is an $H^B_m(0) \otimes H_{n-m}(0)$-module homomorphism, take any $\gamma \in \Sigma^B_R(P)$ and $i \in [0,n]$.
By considering \cref{eq: st plus} and \cref{eq: st minus}, one can see that
\begin{align*}
\tcalF_{P;m}(\gamma \cdot \opi^B_i) 
& = \rmst^+_{m}(\gamma \cdot \opi^B_i) \otimes \rmst^-_{n-m}(\gamma \cdot \opi^B_i)
\\
& = \begin{cases}
\rmst^+_{m}(\gamma) \cdot \opi^B_i \otimes \rmst^-_{n-m}(\gamma)
& \text{if $1 \leq i \leq m-1$,} \\
\rmst^+_{m}(\gamma) \otimes \rmst^-_{n-m}(\gamma) \cdot \opi_{i-m}
& \text{if $m+1 \leq i \leq n-1$}
\end{cases}
\\
&= \begin{cases}
\tcalF_{P;m}(\gamma) \cdot (\opi^B_i \otimes 1)
& \text{for $1 \leq i \leq m-1$,} \\
\tcalF_{P;m}(\gamma) \cdot (1 \otimes \opi_{i-m})
& \text{for $m+1 \leq i \leq n-1$.}
\end{cases}
\end{align*}
Thus, $\tcalF_{P;m}$ is an $H^B_m(0) \otimes H_{n-m}(0)$-module isomorphism.
\end{proof}

\begin{corollary}\label{cor: comodule isom}
The Grothendieck group $\bigoplus_{n \ge 0} \calG_0(\scrP^B(n))$ is isomorphic to $\QSym^B$ as $\QSym$-comodules.
\end{corollary}
\begin{proof}
By \cref{thm: restriction}, the restriction to $H^B_m(0)\otimes H_{n-m}(0)$ of a $B_n$ poset module is a direct sum of tensor products of $B_m$ poset modules and modules for posets on $[n-m]$. Combining this with \cref{thm: Grothendieck group of poset modules}, the fact that $\GHBb$ is isomorphic as $\QSym$-comodules to $\QSym^B$ (\cite[Theorem 5(i)]{Huang17}), and the fact that $\bigoplus_{n \ge 0} \calG_0(\scrP(n))$ is isomorphic to $\QSym$ as Hopf algebras (\cite[Theorem 3.5]{CKO24}), the result follows.
\end{proof}

As a further corollary, we obtain a generalization of \cref{eq: coaction on fund B}, which describes the right $\QSym$-coaction $\updelta^B$ on $\QSym^B$.

\begin{corollary}\label{cor: KBP coproduct}
Let $P$ be a $B_n$ poset. Then 
\[
\updelta^B(K^B_{P}) = 
\sum_{0 \le m \le n} \left(
\sum_{Q \in \LS^B(P;m)}
\sum_{U \in \upper(Q)}
K^B_{\stB(Q)} \otimes K_{\rmst(U)}
\right).
\]
\end{corollary}

\begin{proof}
By combining \cref{thm: restriction} with the fact that $\ch^B$ is a right $\QSym$-comodule isomorphism, the assertion follows immediately.
\end{proof}

\begin{example}
Revisit \cref{eg: decomposition of Bn poset}.
The $H^B_1(0)\otimes H_2(0)$-action on $\Sigma^B_R(P)$ is illustrated in the following figure:
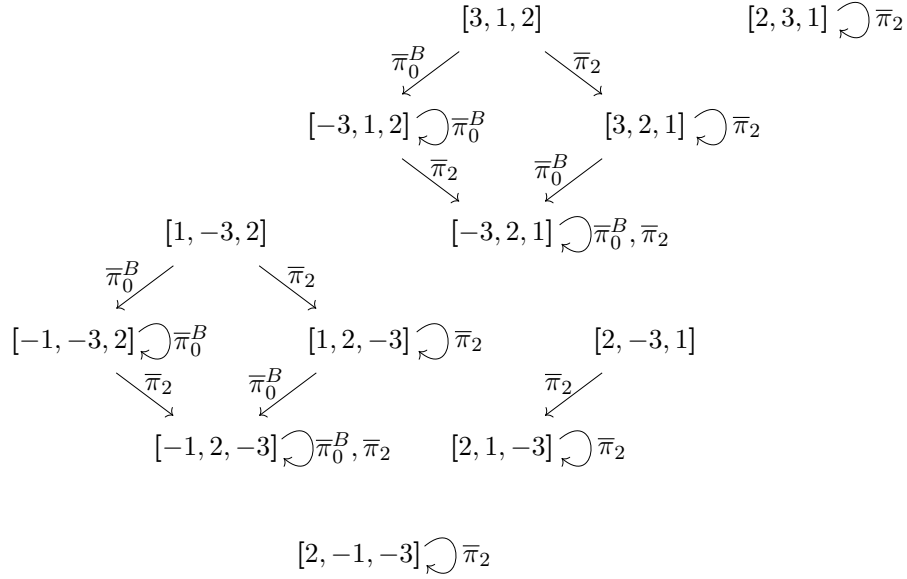
\begin{displaymath}
\def \hp {2}
\def \vp {1.5}
\def \lp {0.55}
\scalebox{0.95}{$
\begin{tikzpicture}

\node at (0*\hp, 5*\vp) {\small $\bracket{3,1,2}$};

\node at (2*\hp, 5*\vp) {\small $\bracket{2,3,1}$};
\node at (2*\hp + \lp , 5*\vp) {} edge [out=40,in=320, loop] ();
\node at (2*\hp + 2.6*\lp, 5*\vp) {\small $\opi_2$};

\draw [->] (-0.3*\hp, 4.7*\vp) -- (-0.7*\hp, 4.3*\vp);
\node at (-0.65*\hp, 4.6*\vp) {\small $\opi^B_0$};

\draw [->] (0.3*\hp, 4.7*\vp) -- (0.7*\hp, 4.3*\vp);
\node at (0.6*\hp, 4.6*\vp) {\small $\opi_2$};

\node at (-1*\hp, 4*\vp) {\small $\bracket{-3,1,2}$};
\node at (-1*\hp + 1.2*\lp , 4*\vp) {} edge [out=40,in=320, loop] ();
\node at (-1*\hp + 2.8*\lp, 4*\vp) {\small $\opi^B_0$};

\node at (1*\hp, 4*\vp) {\small $\bracket{3,2,1}$};
\node at (1*\hp + \lp , 4*\vp) {} edge [out=40,in=320, loop] ();
\node at (1*\hp + 2.6*\lp, 4*\vp) {\small $\opi_2$};

\draw [->] (-0.7*\hp, 3.7*\vp) -- (-0.3*\hp, 3.3*\vp);
\node at (-0.4*\hp, 3.6*\vp) {\small $\opi_2$};

\draw [->] (0.7*\hp, 3.7*\vp) -- (0.3*\hp, 3.3*\vp);
\node at (0.35*\hp, 3.6*\vp) {\small $\opi^B_0$};

\node at (-2*\hp, 3*\vp) {\small $\bracket{1,-3,2}$};

\node at (0*\hp, 3*\vp) {\small $\bracket{-3,2,1}$};
\node at (0*\hp + 1.2*\lp , 3*\vp) {} edge [out=40,in=320, loop] ();
\node at (0*\hp + 3.3*\lp, 3*\vp) {\small $\opi^B_0, \opi_2$};

\draw [->] (-2.3*\hp, 2.7*\vp) -- (-2.7*\hp, 2.3*\vp);
\node at (-2.65*\hp, 2.6*\vp) {\small $\opi^B_0$};

\draw [->] (-1.7*\hp, 2.7*\vp) -- (-1.3*\hp, 2.3*\vp);
\node at (-1.4*\hp, 2.6*\vp) {\small $\opi_2$};

\node at (-3*\hp, 2*\vp) {\small $\bracket{-1,-3,2}$};
\node at (-3*\hp + 1.4*\lp , 2*\vp) {} edge [out=40,in=320, loop] ();
\node at (-3*\hp + 3*\lp, 2*\vp) {\small $\opi^B_0$};

\node at (-1*\hp, 2*\vp) {\small $\bracket{1,2,-3}$};
\node at (-1*\hp + 1.2*\lp , 2*\vp) {} edge [out=40,in=320, loop] ();
\node at (-1*\hp + 2.8*\lp, 2*\vp) {\small $\opi_2$};

\node at (1*\hp, 2*\vp) {\small $\bracket{2,-3,1}$};

\draw [->] (-2.7*\hp, 1.7*\vp) -- (-2.3*\hp, 1.3*\vp);
\node at (-2.4*\hp, 1.6*\vp) {\small $\opi_2$};

\draw [->] (-1.3*\hp, 1.7*\vp) -- (-1.7*\hp, 1.3*\vp);
\node at (-1.65*\hp, 1.6*\vp) {\small $\opi^B_0$};

\draw [->] (0.7*\hp, 1.7*\vp) -- (0.3*\hp, 1.3*\vp);
\node at (0.4*\hp, 1.6*\vp) {\small $\opi_2$};

\node at (-2*\hp, 1*\vp) {\small $\bracket{-1,2,-3}$};
\node at (-2*\hp + 1.4*\lp , 1*\vp) {} edge [out=40,in=320, loop] ();
\node at (-2*\hp + 3.5*\lp, 1*\vp) {\small $\opi^B_0, \opi_2$};

\node at (0*\hp, 1*\vp) {\small $\bracket{2,1,-3}$};
\node at (0*\hp + 1.2*\lp , 1*\vp) {} edge [out=40,in=320, loop] ();
\node at (0*\hp + 2.8*\lp, 1*\vp) {\small $\opi_2$};

\node at (-1*\hp, 0*\vp) {\small $\bracket{2,-1,-3}$};
\node at (-1*\hp + 1.4*\lp , 0*\vp) {} edge [out=40,in=320, loop] ();
\node at (-1*\hp + 3*\lp, 0*\vp) {\small $\opi_2$};

\end{tikzpicture}
$}
\end{displaymath}
And, the $H^B_1(0)\otimes H_2(0)$-action on 
$\bigsqcup_{Q \in \LS^B(P;1)}
\bigsqcup_{U \in \upper(Q)}
\Sigma^B_R(\stB(Q)) \times \Sigma_R(\rmst(U))$ is illustrated in the following figure:
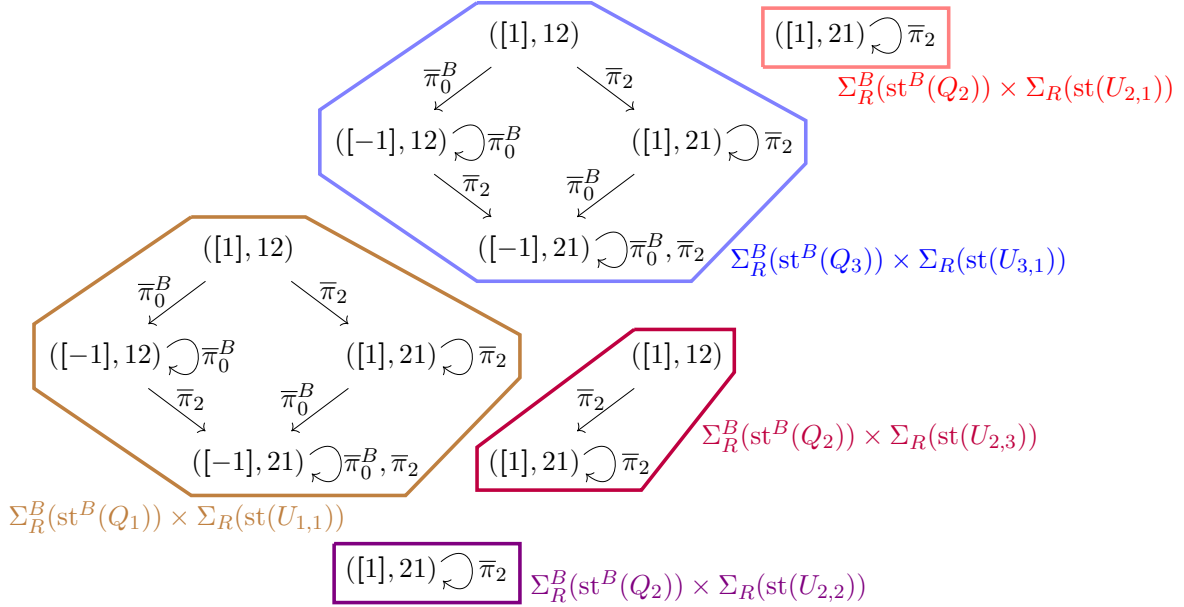
\begin{displaymath}
\def \hp {2}
\def \vp {1.5}
\def \lp {0.55}
\scalebox{0.95}{$
\begin{tikzpicture}

\node at (0*\hp, 5*\vp) {\small $(\bracket{1}, 12)$};

\node at (2*\hp, 5*\vp) {\small $(\bracket{1}, 21)$};
\node at (2*\hp + \lp , 5*\vp) {} edge [out=40,in=320, loop] ();
\node at (2*\hp + 2.6*\lp, 5*\vp) {\small $\opi_2$};

\draw [->] (-0.3*\hp, 4.7*\vp) -- (-0.7*\hp, 4.3*\vp);
\node at (-0.65*\hp, 4.6*\vp) {\small $\opi^B_0$};

\draw [->] (0.3*\hp, 4.7*\vp) -- (0.7*\hp, 4.3*\vp);
\node at (0.6*\hp, 4.6*\vp) {\small $\opi_2$};

\node at (-1*\hp, 4*\vp) {\small $(\bracket{-1}, 12)$};
\node at (-1*\hp + 1.3*\lp , 4*\vp) {} edge [out=40,in=320, loop] ();
\node at (-1*\hp + 2.9*\lp, 4*\vp) {\small $\opi^B_0$};

\node at (1*\hp, 4*\vp) {\small $(\bracket{1}, 21)$};
\node at (1*\hp + \lp , 4*\vp) {} edge [out=40,in=320, loop] ();
\node at (1*\hp + 2.6*\lp, 4*\vp) {\small $\opi_2$};

\draw [->] (-0.7*\hp, 3.7*\vp) -- (-0.3*\hp, 3.3*\vp);
\node at (-0.4*\hp, 3.6*\vp) {\small $\opi_2$};

\draw [->] (0.7*\hp, 3.7*\vp) -- (0.3*\hp, 3.3*\vp);
\node at (0.35*\hp, 3.6*\vp) {\small $\opi^B_0$};

\node at (-2*\hp, 3*\vp) {\small $(\bracket{1}, 12)$};

\node at (0*\hp, 3*\vp) {\small $(\bracket{-1}, 21)$};
\node at (0*\hp + 1.3*\lp , 3*\vp) {} edge [out=40,in=320, loop] ();
\node at (0*\hp + 3.4*\lp, 3*\vp) {\small $\opi^B_0, \opi_2$};

\draw [->] (-2.3*\hp, 2.7*\vp) -- (-2.7*\hp, 2.3*\vp);
\node at (-2.65*\hp, 2.6*\vp) {\small $\opi^B_0$};

\draw [->] (-1.7*\hp, 2.7*\vp) -- (-1.3*\hp, 2.3*\vp);
\node at (-1.4*\hp, 2.6*\vp) {\small $\opi_2$};

\node at (-3*\hp, 2*\vp) {\small $(\bracket{-1}, 12)$};
\node at (-3*\hp + 1.3*\lp , 2*\vp) {} edge [out=40,in=320, loop] ();
\node at (-3*\hp + 2.9*\lp, 2*\vp) {\small $\opi^B_0$};

\node at (-1*\hp, 2*\vp) {\small $(\bracket{1}, 21)$};
\node at (-1*\hp + 1*\lp , 2*\vp) {} edge [out=40,in=320, loop] ();
\node at (-1*\hp + 2.6*\lp, 2*\vp) {\small $\opi_2$};

\node at (1*\hp, 2*\vp) {\small $(\bracket{1}, 12)$};

\draw [->] (-2.7*\hp, 1.7*\vp) -- (-2.3*\hp, 1.3*\vp);
\node at (-2.4*\hp, 1.6*\vp) {\small $\opi_2$};

\draw [->] (-1.3*\hp, 1.7*\vp) -- (-1.7*\hp, 1.3*\vp);
\node at (-1.65*\hp, 1.6*\vp) {\small $\opi^B_0$};

\draw [->] (0.7*\hp, 1.7*\vp) -- (0.3*\hp, 1.3*\vp);
\node at (0.4*\hp, 1.6*\vp) {\small $\opi_2$};

\node at (-2*\hp, 1*\vp) {\small $(\bracket{-1}, 21)$};
\node at (-2*\hp + 1.3*\lp , 1*\vp) {} edge [out=40,in=320, loop] ();
\node at (-2*\hp + 3.4*\lp, 1*\vp) {\small $\opi^B_0, \opi_2$};

\node at (0*\hp, 1*\vp) {\small $(\bracket{1}, 21)$};
\node at (0*\hp + 1*\lp , 1*\vp) {} edge [out=40,in=320, loop] ();
\node at (0*\hp + 2.6*\lp, 1*\vp) {\small $\opi_2$};

\node at (-1*\hp, 0*\vp) {\small $(\bracket{1}, 21)$};
\node at (-1*\hp + 1*\lp , 0*\vp) {} edge [out=40,in=320, loop] ();
\node at (-1*\hp + 2.6*\lp, 0*\vp) {\small $\opi_2$};

\draw[red!50, line width=0.5mm] (1.6*\hp, 4.7*\vp) -- (2.9*\hp, 4.7*\vp) -- (2.9*\hp, 5.25*\vp) -- (1.6*\hp, 5.25*\vp) -- (1.6*\hp, 4.7*\vp); 
\node at (3.3*\hp, 4.5*\vp) {\small \color{red} $\Sigma^B_R(\stB(Q_2)) \times \Sigma_R(\rmst(U_{2,1}))$};

\draw[blue!50, line width=0.5mm] (-0.4*\hp, 5.3*\vp) -- (0.4*\hp, 5.3*\vp) -- (1.9*\hp, 4.2*\vp) -- (1.9*\hp, 3.8*\vp) -- (1.1*\hp, 2.7*\vp) -- (-0.4*\hp, 2.7*\vp) -- (-1.5*\hp, 3.7*\vp) -- (-1.5*\hp, 4.3*\vp) -- (-0.4*\hp, 5.3*\vp); 
\node at (2.55*\hp, 2.9*\vp) {\small \color{blue} $\Sigma^B_R(\stB(Q_3)) \times \Sigma_R(\rmst(U_{3,1}))$};

\draw[brown, line width=0.5mm]  (-2.4*\hp, 3.3*\vp) -- (-1.6*\hp, 3.3*\vp) -- (-0.1*\hp, 2.2*\vp) -- (-0.1*\hp, 1.8*\vp) -- (-0.9*\hp, 0.7*\vp) -- (-2.4*\hp, 0.7*\vp) -- (-3.5*\hp, 1.7*\vp) -- (-3.5*\hp, 2.3*\vp) -- (-2.4*\hp, 3.3*\vp); 
\node at (-2.5*\hp, 0.5*\vp) {\small \color{brown} $\Sigma^B_R(\stB(Q_1)) \times \Sigma_R(\rmst(U_{1,1}))$};

\draw[purple, line width=0.5mm]  (0.7*\hp, 2.25*\vp) -- (1.4*\hp, 2.25*\vp) -- (1.4*\hp, 1.85*\vp) -- (0.75*\hp, 0.75*\vp) -- (-0.4*\hp, 0.75*\vp) -- (-0.4*\hp, 1.15*\vp) -- (0.7*\hp, 2.25*\vp); 
\node at (2.35*\hp, 1.25*\vp) {\small \color{purple} $\Sigma^B_R(\stB(Q_2)) \times \Sigma_R(\rmst(U_{2,3}))$};

\draw[violet, line width=0.5mm]  (-1.4*\hp, 0.25*\vp) -- (-0.1*\hp, 0.25*\vp) -- (-0.1*\hp, -0.3*\vp) -- (-1.4*\hp, -0.3*\vp) -- (-1.4*\hp, 0.25*\vp);
\node at (1.15*\hp, -0.15*\vp) {\small \color{violet} $\Sigma^B_R(\stB(Q_2)) \times \Sigma_R(\rmst(U_{2,2}))$};

\end{tikzpicture}
$}
\end{displaymath}
In the above figures, we denote $\opi^B_0\otimes {\rm id}$ by $\opi^B_0$, and ${\rm id} \otimes \opi_1$ by $\opi_2$.
\end{example}

\subsection{(Anti-)automorphism twists}\label{subsec:twists}
In this subsection, we describe the images of $B_n$ poset modules, and their induction and restriction, under automorphism and anti-automorphism twists.
We begin by introducing some necessary terminology.

For $A$ an associative $\C$-algebra, let $\text{\bf mod-}A$ denote the category of finite-dimensional right $A$-modules. 
Given an isomorphism $f: B \ra A$ of associative $\C$-algebras and a right $A$-module $M$, define an action $\cdot_f$ of $B$ on the underlying space of $M$ by  
\[
m \cdot_f b :=  m \cdot f(b)  \quad \text{for $b \in B$ and $m \in M$,}
\]
and let $f[M]$ denote the resulting right $B$-module. 
The \emph{$f$-twist} is the covariant functor $\mathbf{T}^+_f : \text{\bf mod-}A \ra \text{\bf mod-}B$ that sends a right $A$-module $M$ to $f[M]$ and sends an $A$-module homomorphism $\phi: M \ra N$ to $\mathbf{T}^+_f(\phi): f[M] \ra f[N]$, where $\mathbf{T}^+_f(\phi)(m) := \phi(m)$ for $m \in M$.

Given an anti-isomorphism $g: B \ra A$ of associative $\C$-algebras and a right $A$-module $M$, define an action $\cdot^g$ of $B$ on $M^*$, the dual space of $M$, by
\[
(\delta \cdot^g b)(m) := \delta(m \cdot g(b))
\quad \text{for $b \in B$, $\delta \in M^*$, and $m \in M$,}
\]and let $g[M]$ denote the resulting right $B$-module. 
The \emph{$g$-twist} is the contravariant functor $\mathbf{T}^-_g: \text{\bf mod-}A \ra \text{\bf mod-}B$ that sends a right $A$-module $M$ to $g[M]$ and sends an $A$-module homomorphism $\phi:M \ra N$ to $\mathbf{T}^-_g(\phi): g[N] \ra g[M]$, where $\mathbf{T}^-_g(\phi)(\delta) := \delta \circ \phi$ for $\delta \in M^*$.

Let $(W,S)$ be a finite Coxeter system.
In \cite{Fayers05}, Fayers studies the twists arising from two involutive automorphisms $\upphi$ and $\uptheta$, and an involutive anti-automorphism $\upchi$ on $H_W(0)$, which are defined by
\[
\upphi(\opi_s) = \opi_{w_0 s w_0}, 
\quad
\uptheta (\opi_s) = - \pi_s, 
\quad \text{and} \quad 
\upchi(\opi_s) = \opi_s \quad (s \in S).
\]
Note that $\upphi, \uptheta$ and $\upchi$ commute with each other. Abusing notation, for any $H^B_n(0)$-module or $H_n(0)$-module $M$, we will also use $\upphi[M]$, $\uptheta[M]$ and $\upchi[M]$ to denote the modules $T_\upphi^+(M), T_\uptheta^+(M)$ and $T_\upchi^-(M)$, respectively. 

Let $P$ be a poset on a finite subset $I = \{i_1 < i_2 < \cdots < i_n\}$ of $\Z$. We define two posets on the same underlying set $I$, denoted by $P^*$ and $\overline{P}$, whose orders are given by
\[
i_j \preceq_{P^*} i_k 
\ \  \text{if and only if} \ \  
i_k \preceq_{P} i_j
\quad \text{and} \quad
i_j \preceq_{\overline{P}} i_k 
\ \  \text{if and only if} \ \  
i_{n+1-j} \preceq_{P} i_{n+1-k}.
\]

Note that the maps $P\mapsto P^*$ and $P\mapsto \overline{P}$ are both involutions.
Note also that $P^* = \overline{P}$ for any $B_n$ poset $P$, however, this distinction will be needed when we consider twists of poset modules of type $A$ in \cref{thm: auto twists and ind res}.

Given a $B_n$ poset $P$, we define another right $H^B_n(0)$-action on $\C\Sigma^B_R(P)$ by
\begin{align*}
\gamma \cdot \pi^B_i:= 
\begin{cases}
\gamma & \text{if $i \in \Des_R(\gamma)$},\\
0 & \text{if $i \notin \Des_R(\gamma)$ and $\gamma s^B_i \notin \Sigma^B_R(P)$},\\
\gamma s^B_i & \text{if $i \notin \Des_R(\gamma)$  and $\gamma s^B_i \in \Sigma^B_R(P)$}
\end{cases}
\end{align*}
for $i \in [0,n-1]$ and $\gamma \in \Sigma^B_R(P)$. 
We denote the resulting right $H^B_n(0)$-module by $\sfM^B_P$. Note this action is similar (but not identical) to the action defining $M^B_P$; one can prove it is an $H^B_n(0)$-action in a similar manner to the proof of \cref{thm: Bn poset asc compatible}. Moreover one can show, using similar arguments, that \cref{thm: induction product} and \cref{thm: restriction} remain valid when stated for $\sfM^B_P$. We will also make use of the modules $\sfM_P$ associated to posets $P$ on $[n]$, which are defined in \cite[Section 3]{CKO24} in an analogous manner to $\sfM^B_P$. 

With these definitions, we are now ready to describe the effect of the (anti)-automorphism twists on $B_n$ poset modules. Note that, in $\SG^B_n$, $w_0 s^B_i w_0 = s^B_i$ for all $i \in [0,n-1]$, so $\upphi$ is equal to the identity map on $H^B_n(0)$. 
Therefore, we only consider $\uptheta$ and $\upchi$ in the following theorem.

\begin{theorem}\label{thm: auto-twists}
Let $P$ be a $B_n$ poset.
Then 
\[
\uptheta[M^B_P] \cong \sfM^B_P,
\quad
\upchi[M^B_P] \cong \sfM^B_{P^*},
\quad \text{and} \quad
\uptheta \circ \upchi[M^B_P] \cong M^B_{P^*}
\quad \text{as $H^B_n(0)$-modules.}
\]
\end{theorem}

\begin{proof}
It is straightforward that the $\C$-linear map \begin{align*}
f_1 &: \uptheta[M^B_P] \ra \sfM^B_P, \quad \gamma \mapsto (-1)^{\ell(\gamma)} \gamma  \quad \text{for $\gamma \in \Sigma^B_R(P)$}
\end{align*}
is an $H^B_n(0)$-module isomorphism.
In addition, if we prove $\upchi[M^B_P] \cong \sfM^B_{P^*}$, then it follows that
\[
\uptheta \circ \upchi[M^B_P] \cong \uptheta [\sfM^B_{P^*}] \cong M^B_{P^*}.
\]

Let us prove that $\upchi[M^B_P] \cong \sfM^B_{P^*}$.
Define a $\C$-linear map
\begin{align*}
f_2 &: \upchi[M^B_{P}] \ra \sfM^B_{P^*}, 
\quad \gamma^* \mapsto w_0 \gamma 
\quad \text{for $\gamma \in \Sigma^B_R(P)$,}
\end{align*}
where $\gamma^*$ denotes the dual of $\gamma$ with respect to the basis $\Sigma^B_R(P)$ for $M^B_P$.
In a similar manner to the proof of \cite[Theorem 4]{JKLO22}, one can prove that the $H^B_n(0)$-action is compatible with the linear map $f_2$.
Thus, it remains to prove that $f_2$ is a well-defined $\C$-linear isomorphism.

Let $\gamma$ be an element of $\Sigma^B_R(P)$, and choose arbitrary $u, v \in [-n, n]$ with $u \preceq_{P^*} v$.
By the definition of $P^*$, we have $v \preceq_{P} u$.
This implies $\gamma^{-1}(v) \le \gamma^{-1}(u)$, or equivalently, $\gamma^{-1}(-u) \le \gamma^{-1}(-v)$.
Since $w_0(i) = -i$ for all $i \in [-n, n]$, we have $(w_0 \gamma)^{-1}(u) \le (w_0 \gamma)^{-1}(v)$.
This shows that $w_0 \gamma$ belongs to $\Sigma^B_R(P^*)$, thus $f_2$ is well-defined.
Moreover, since the map 
$P \mapsto P^*$ 
is involutive, the inverse of $f_2$ is given by $\gamma \mapsto (w_0\gamma)^*$. Therefore, $f_2$ is a $\C$-linear isomorphism.
\end{proof}

The following theorem shows how the $\uptheta$-twist and the $\upchi$-twist interact with induction and restriction.
To state and prove this theorem, we need the $\upphi$-, $\uptheta$- and $\upchi$-twists on poset modules of type $A$, which were determined by Choi--Kim--Oh in \cite[Theorem 3.6]{CKO24}. Note that, unlike for $H_n^B(0)$, the automorphism $\upphi$ is not the identity on $H_n(0)$.

\begin{theorem}\label{thm: auto twists and ind res}
We have the following.
\begin{enumerate}[label = {\rm (\arabic*)}]
\item 
Let $P_1$ be a $B_m$ poset and $P_2$ a poset on $[n]$. Then
\[
\uptheta\left[M^B_{P_1} \otimes M_{P_2}\uparrow_{H^B_m(0) \otimes H_n(0)}^{H^B_{m+n}(0)} \right] 
\cong
\uptheta\left[M^B_{P_1}] \otimes \uptheta[M_{P_2}\right] \uparrow_{H^B_m(0) \otimes H_n(0)}^{H^B_{m+n}(0)}
\]
and
\[
\upchi\left[M^B_{P_1} \otimes M_{P_2}\uparrow_{H^B_m(0) \otimes H_n(0)}^{H^B_{m+n}(0)} \right] 
\cong
\upchi\left[M^B_{P_1}] \otimes 
(\upchi \circ \upphi)[M_{P_2} \right]\uparrow_{H^B_m(0) \otimes H_n(0)}^{H^B_{m+n}(0)}.
\]

\item 
Let $P$ be a $B_n$ poset and $0 \le m \le n$.
Then
\[
\uptheta\left[M^B_{P}\right] \downarrow^{H^B_{n}(0)}_{H^B_m(0) \otimes H_{n-m}(0)} 
\cong
(\uptheta \otimes \uptheta) \left[M^B_{P}\downarrow^{H^B_{n}(0)}_{H^B_m(0) \otimes H_{n-m}(0)} \right] 
\]
and
\[
\upchi\left[M^B_{P}\right] \downarrow^{H^B_{n}(0)}_{H^B_m(0) \otimes H_{n-m}(0)}
\cong
(\upchi \otimes \upchi)\left[M^B_{P} \downarrow^{H^B_{n}(0)}_{H^B_m(0) \otimes H_{n-m}(0)} \right]. 
\]
\end{enumerate}
\end{theorem}
\begin{proof}
(1)
From \cref{thm: induction product}, \cref{thm: auto-twists}, and \cite[Theorem 3.6]{CKO24}, the first isomorphism is immediate.
For the second isomorphism, note that
$(P_1 \sqcup_B P_2)^* = P_1^* \sqcup_B P_2^*$.
This, together with \cref{thm: induction product} and \cref{thm: auto-twists}, implies that
\[
\upchi\left[M^B_{P_1} \otimes M_{P_2}\uparrow_{H^B_m(0) \otimes H_n(0)}^{H^B_{m+n}(0)} \right]
\cong 
\sfM^B_{P^*_1 \sqcup_B P^*_2}.
\]
On the other hand, combining \cref{thm: induction product}, \cref{thm: auto-twists}, and \cite[Theorem 3.6]{CKO24}, we have
\[
\upchi\left[M^B_{P_1}] \otimes 
(\upchi \circ \upphi)[M_{P_2} \right]\uparrow_{H^B_m(0) \otimes H_n(0)}^{H^B_{m+n}(0)}
\cong
\sfM^B_{P^*_1 \sqcup_B P^*_2}.
\]

(2) 
From \cref{thm: restriction}, \cref{thm: auto-twists}, and \cite[Theorem 3.6]{CKO24}, the first isomorphism is immediate.
By \cref{thm: restriction} and \cref{thm: auto-twists}, we have
\[
\upchi\left[M^B_{P}\right] \downarrow^{H^B_{n}(0)}_{H^B_m(0) \otimes H_{n-m}(0)}
\cong
\sfM^B_{P^*} \downarrow^{H^B_{n}(0)}_{H^B_m(0) \otimes H_{n-m}(0)}
\cong
\bigoplus_{Q \in \LS^B(P^*;m)}
\bigoplus_{U \in \upper(Q)}
\sfM^B_{\stB(Q)} \otimes \sfM_{\rmst(U)}.
\]
On the other hand, combining \cref{thm: restriction}, \cref{thm: auto-twists}, and \cite[Theorem 3.6]{CKO24}, we have
\[
(\upchi \otimes \upchi)\left[M^B_{P} \downarrow^{H^B_{n}(0)}_{H^B_m(0) \otimes H_{n-m}(0)} \right] \cong
\bigoplus_{Q' \in \LS^B(P;m)}
\bigoplus_{U' \in \upper(Q')}
\sfM^B_{\stB(Q')^*} \otimes \sfM_{\overline{\rmst(U')}}.
\]

We claim that the right-hand sides of the two expressions above are equal. To see this, we observe the following:
\begin{enumerate}[label = {\rm (\roman*)}]
\item 
$Q \in \LS^B(P^*;m)$ if and only if $Q^* \in \LS^B(P;m)$.
\item 
For any $Q \in \LS^B(P;m)$, $\rmst^B(Q^*) = \rmst^B(Q)^*$.
\item $\upper(Q^*) = \{-U : U\in \upper(Q)\}$ (where $-U$ is the poset on $\{-i:i\in U\}$ such that $-x\preceq_{-U} -y$ if and only if $x\preceq_{U} y$), and $\rmst(-U) = \rmst(\overline{U})$ for all $U\in \upper(Q)$. 
\item 
For any $Q \in \LS^B(P;m)$ and $U \in \upper(Q)$, $\rmst(\overline{U}) = \overline{\rmst(U)}$.
\end{enumerate}
These observations imply that
\[
\bigoplus_{Q \in \LS^B(P^*;m)}
\bigoplus_{U \in \upper(Q)}
\sfM^B_{\stB(Q)} \otimes \sfM_{\rmst(U)}
=
\bigoplus_{Q' \in \LS^B(P;m)}
\bigoplus_{U' \in \upper(Q')}
\sfM^B_{\stB(Q')^*} \otimes \sfM_{\overline{\rmst(U')}}
\]
as desired.
\end{proof}

\section{Connection to weak Bruhat interval modules}\label{Sec: regular}
In this section, we connect the structure of $B_n$ posets to the right weak Bruhat order on $\SG^B_n$. We begin by defining a family of $B_n$ posets that we call \emph{distinguished} $B_n$ posets, and show there is precisely one distinguished $B_n$ poset having a given set of type-$B$ linear extensions. Consequently, it suffices to consider only distinguished $B_n$ posets when studying $B_n$ poset modules. 

Intervals in weak Bruhat order form an important family of convex subsets of $\SG^B_n$, or more generally of a Coxeter group $W$, and the $0$-Hecke modules associated to these intervals have been studied in types $A$ and $B$ and in arbitrary finite type (e.g., \cite{JKLO22, YY24, BS23}). 
We identify a subfamily of distinguished $B_n$ posets that we call \emph{regular} $B_n$ posets, analogous to the regular posets on $[n]$ defined in \cite{BW91}, and prove that a subset of $\SG^B_n$ is an interval in right weak Bruhat order if and only if it is the set of type-$B$ linear extensions of a regular $B_n$ poset. 
Finally, we consider the full subcategory $\scrB^B(n)$ of $\scrP^B(n)$ whose objects are direct sums of finitely many isomorphic copies of $B_n$ poset modules for regular $B_n$ posets (or equivalently, right weak Bruhat interval modules). 
We prove that the corresponding Grothendieck group $\bigoplus_{n\ge0} \scrB^B(n)$ is isomorphic to the Grothendieck group $\bigoplus_{n \ge 0} \calG_0(\modHBn)$ (and therefore, also isomorphic to $\QSym^B$) as right $\QSym$-modules and right $\QSym$-comodules. 

\subsection{Distinguished posets}
A poset on $[n]$ is uniquely determined by its set of linear extensions. However, the same is not true for $B_n$ posets: if $P$ and $Q$ are $B_n$ posets, it is possible that $\calL^B(P) = \calL^B(Q)$ even if $P\neq Q$. For example, the posets $P^{(1)}_1$ and $P^{(1)}_2$ in \cref{eg: Bn posets} are different but have the same (singleton) set of type-$B$ linear extensions. 
However, if $\calL^B(P) = \calL^B(Q)$, then it follows immediately that $M_P \cong M_Q$.
This motivates us to define an equivalence relation on the set of $B_n$ posets by
\[
P\sim Q \quad \text{if $\calL^B(P)=\calL^B(Q)$.}
\]
In this subsection, we provide a nice representative for each equivalence class under $\sim$.

\begin{definition}
We call a $B_n$ poset \emph{distinguished} if whenever $x$ and $-x$ are comparable, $x$ (and $-x$) is comparable to $0$.
\end{definition}

We will show that there is exactly one distinguished $B_n$ poset in each equivalence class. In fact, the distinguished $B_n$ posets are precisely the $B_n$ posets obtained as the intersection of the linear orders associated to a set of signed permutations.
To clarify this statement, we collect the following terminology:
\begin{enumerate}[label = {\rm (\roman*)}]
\item 
For any $\sigma \in \SG^B_n$, let $L_\sigma = ([-n,n], \preceq_{L_\sigma})$ be the linear $B_n$ poset defined by
\[
\sigma(-n) \preceq_{L_\sigma} \sigma(- n +1) \preceq_{L_\sigma} \cdots \preceq_{L_\sigma} \sigma(n).
\]

\item 
For any poset $P = ([-n,n], \preceq_P)$, we identify $\preceq_P$ with the set 
\[\{(i,j) \in [-n,n]^2 \mid i \preceq_{P} j\}.\] 

\item 
For any $U \subseteq \SG^B_n$, let $\poset(U)$ be the poset defined by
\[
\poset(U) = \left([-n,n], \bigcap_{\sigma \in U} \preceq_{L_{\sigma}} \right).
\]
\end{enumerate}

\begin{example}\label{eg: intersection of cal L}
Let $U = \{[2,1], [1,-2]\} \subseteq \SG^B_2$.
Note that
\[
\def \hp {1}
\def \vp {0.4}
L_{[2,1]} = 
\begin{array}{l}
\begin{tikzpicture}
\node at (0, 3*\vp) {$\bullet$};
\node[right] at (0, 3*\vp) {\scriptsize $1$};
\node at (0, 1.5*\vp) {$\bullet$};
\node[right] at (0, 1.5*\vp) {\scriptsize $2$};
\node at (0, 0*\vp) {$\bullet$};
\node[right] at (0, 0*\vp) {\scriptsize $0$};
\node at (0, -1.5*\vp) {$\bullet$};
\node[right] at (0, -1.5*\vp) {\scriptsize $-2$};
\node at (0, -3*\vp) {$\bullet$};
\node[right] at (0, -3*\vp) {\scriptsize $-1$};
\draw[] (0, 3*\vp) -- (0, -3*\vp);
\end{tikzpicture}
\end{array}
\quad \text{and} \quad
L_{[1,-2]} = 
\begin{array}{l}
\begin{tikzpicture}
\node at (0, 3*\vp) {$\bullet$};
\node[right] at (0, 3*\vp) {\scriptsize $-2$};
\node at (0, 1.5*\vp) {$\bullet$};
\node[right] at (0, 1.5*\vp) {\scriptsize $1$};
\node at (0, 0*\vp) {$\bullet$};
\node[right] at (0, 0*\vp) {\scriptsize $0$};
\node at (0, -1.5*\vp) {$\bullet$};
\node[right] at (0, -1.5*\vp) {\scriptsize $-1$};
\node at (0, -3*\vp) {$\bullet$};
\node[right] at (0, -3*\vp) {\scriptsize $2$};
\draw[] (0, 3*\vp) -- (0, -3*\vp);
\end{tikzpicture}
\end{array}.
\]
We have
\[
\preceq_{\poset(U)} \, = \, 
\preceq_{L_{[2,1]}} \cap \preceq_{L_{[1,-2]}} = \left\{ 
(-1,-2), (-1,0), (-1,1), (0,1), (2,1)
\right\}.
\]
The Hasse diagram of $\poset(U)$ is shown below. 
\[
\def \hp {0.7}
\def \vp {0.4}
\poset(U) = 
\begin{array}{l}
\begin{tikzpicture}
\node at (0.75*\hp, 1.5*\vp) {$\bullet$};
\node[right] at (0.75*\hp, 1.5*\vp) {\scriptsize $1$};

\draw[] (0.75*\hp, 1.5*\vp) -- (1.5*\hp, 0*\vp);
\draw[] (0.75*\hp, 1.5*\vp) -- (-0.75*\hp, -1.5*\vp);

\node at (-1.5*\hp, 0*\vp) {$\bullet$};
\node[right] at (-1.5*\hp, 0*\vp) {\scriptsize $-2$};

\node at (0, 0*\vp) {$\bullet$};
\node[right] at (0, 0*\vp) {\scriptsize $0$};

\node at (1.5*\hp, 0) {$\bullet$};
\node[right] at (1.5*\hp, 0) {\scriptsize $2$};

\draw[] (-1.5*\hp, 0) -- (-0.75*\hp, -1.5*\vp);

\node at (-0.75*\hp, -1.5*\vp) {$\bullet$};
\node[right] at (-0.75*\hp, -1.5*\vp) {\scriptsize $-1$};
\end{tikzpicture}
\end{array}
\]
Note that $\poset(U)$ is a distinguished $B_n$ poset.
\end{example}

Let $U\subseteq \SG^B_n$. We will prove that $\poset(U)$ is distinguished. Subsequently, for any $B_n$ poset $P$, we will prove that $\poset(\Sigma^B_R(P))$ is the unique distinguished $B_n$ poset in the equivalence class of $P$ under $\sim$. 

\begin{lemma}\label{lem: x y incomp}
Let $P$ be a $B_n$ poset, $x \in [-n,n] \setminus \{0\}$, and $y \in [-n,n]$.
If $x \not\prec_P -x$ and $x$ is incomparable to $y$ in $P$, then there exists $E \in \calL^B(P)$ such that $y \prec_{E} x$.
\end{lemma}

\begin{proof}
We will explicitly construct $E \in \calL^B(P)$ such that $y \prec_{E} x$. To this end, 
we choose a sequence of posets $P_1, P_2, \ldots, P_{n}$ and elements $a_1, a_2, \ldots, a_{n}$ inductively as follows.
\begin{enumerate}[itemsep = 1ex, label = \textbf{P\arabic*.}]
\item 
Set $P_1 := P$, $A_1 := \{a \in [-n,n] \mid x \preceq_P a\}$, and let $a_1$ be a maximal element in $A_1$ with respect to $\preceq_P$.
\item
For $i = 2, 3, \ldots, |A_1|$, let $P_i$ be the poset obtained by deleting $a_{i-1}$ and $-a_{i-1}$ from $P_{i-1}$.
Let $A_i := A_{i-1} \setminus \{a_{i-1}\}$, and let $a_i$ be a maximal element of $A_i$ with respect to $\preceq_P$.
\item 
For $i = |A_1|+1, \ldots, n$, let $P_i$ be the poset obtained by deleting $a_{i-1}$ and $-a_{i-1}$ from $P_{i-1}$.
Let $a_i$ be a nonzero maximal element in $P_i$.
\end{enumerate}

Note that \textbf{P3} is well-defined: for each $i = |A_1|+1, \ldots, n$ there exists a nonzero maximal element $a_i$ in $P_i$, since $P$ is a $B_n$ poset and $|P_{i}| > 1$.
We claim that none of $a_1, a_2, \ldots , a_n$ are equal to $0$. It suffices to confirm that $0$ cannot be chosen as an $a_i$ in \textbf{P1} or \textbf{P2}. The assumption $x \not\prec_P -x$ implies that $x \not\prec_P 0$; otherwise (since $P$ is a $B_n$ poset) we have $x \prec_P 0 \prec_P -x$, contradicting the assumption. Hence, $0\notin A_1$, which implies that $0$ cannot be chosen as an $a_i$ in \textbf{P1} or \textbf{P2}.

Now, let $E = ([-n,n], \preceq_E)$ be the linear order defined by 
\[
-a_1 \prec_E -a_2 \prec_E \cdots \prec_E -a_n \prec_E 0 \prec_E a_n \prec_E \cdots \prec_E a_2 \prec_E a_1.
\]
By the choice of $a_1, a_2, \ldots, a_n$, we have $E \in \calL^B(P)$ and $x = a_{|A_1|}$.
Moreover, since $y$ is incomparable to $x$ in $P$, we have $y \notin A_1$, and thus $y \prec_E x$.
\end{proof}

\begin{lemma}\label{lem: intersection of Bn linear order}
The following hold.
\begin{enumerate}[label = {\rm (\arabic*)}]
\item 
For any $U \subseteq \SG^B_n$, $\poset(U)$ is a distinguished $B_n$ poset.
\item
For any distinguished $B_n$ poset $P$, we have $P = \poset(\Sigma^B_R(P))$.
\end{enumerate}
\end{lemma}
\begin{proof}
(1) 
Let $U \subseteq \SG^B_n$. 
It is straightforward to verify that for any $x,y \in [-n,n]$, $x \preceq_{\poset(U)} y$ if and only if $-y \preceq_{\poset(U)} -x$, thus $\poset(U)$ is a $B_n$ poset.

To prove that $\poset(U)$ is distinguished, assume that there exists $x \in [-n,n]\setminus \{0\}$ such that $-x \prec_{\poset(U)} x$.
Then for all $\sigma \in U$, we have $-x \prec_{L_\sigma} x$.
Since $L_\sigma$ is a linear $B_n$ poset for all $\sigma\in U$, it follows that $-x \prec_{L_\sigma} 0 \prec_{L_\sigma} x$ for all $\sigma\in U$.
This implies that $-x \prec_{\poset(U)} 0 \prec_{\poset(U)} x$.
Therefore, $\poset(U)$ is distinguished. 

(2)
Let $P$ be a distinguished $B_n$ poset.
It suffices to show that $\preceq_P\, = \bigcap_{E \in \calL^B(P)} \preceq_E$. 
One can easily see that $\preceq_P\, = \bigcap_{E \in \calL(P)} \preceq_E$, and therefore $\preceq_P \, \subseteq  \bigcap_{E \in \calL^B(P)} \preceq_E$.
To prove the opposite inclusion $\preceq_P \, \supseteq \, \bigcap_{E \in \calL^B(P)} \preceq_E$, suppose that there exists $(x,y) \in \bigcap_{E \in \calL^B(P)} \preceq_E$ such that $(x,y) \notin \, \preceq_P$.

We may assume that $x\neq 0$, since for any $B_n$ poset $Q$, $(0,y)\in \, \preceq_Q$ if and only if $(-y,0)\in \, \preceq_Q$, and both $P$ and (by (1)) $\poset(\Sigma^B_R(P))$ are $B_n$ posets. 
If $y \prec_P x$, then $y \prec_E x$ for all $E \in \calL^B(P)$.
This implies that $(x,y) \notin \bigcap_{E \in \calL^B(P)} \preceq_E$, which contradicts the choice of $(x,y)$.  
So, we assume that $x$ and $y$ are incomparable in $P$.
If $x \not\prec_P -x$, then \cref{lem: x y incomp} implies that there exists $F \in \calL^B(P)$ such that $y \prec_F x$. 
It follows that $(x,y) \notin \bigcap_{E \in \calL^B(P)} \preceq_E$, which contradicts the choice of $(x,y)$. 
So, we assume that $x \prec_P -x$. 
If $y=0$, then since $P$ is distinguished and $x \prec_P -x$, we have $x\prec_P 0 = y$,  which contradicts $(x,y) \notin \, \preceq_P$. Therefore, we assume that $y\neq 0$.
Since $x$ is incomparable to $y$ in $P$, $-x$ is incomparable to $-y$ in $P$.
If $-y \not\prec_P y$, then by \cref{lem: x y incomp} again, there exists $F \in \calL^B(P)$ such that $-x \prec_F -y$, equivalently, $y \prec_F x$. 
This implies that $(x,y) \notin \bigcap_{E \in \calL^B(P)} \preceq_E$, which contradicts the choice of $(x,y)$. 
So, we assume that $-y \prec_P y$.
Now, since $P$ is distinguished, we have $x, -y \prec_P 0 \prec_P -x, y$.
In particular, $x \prec_P y$ which contradicts $(x,y) \notin \, \preceq_P$.
\end{proof}

In fact, for any $B_n$ poset $P$, $\poset(\Sigma^B_R(P))$ is a distinguished $B_n$ poset that has the same set of type-$B$ linear extensions as $P$.

\begin{lemma}\label{lem: poset have same calL}
For any $B_n$ poset $P$, we have $P \sim \poset(\Sigma^B_R(P))$.
\end{lemma}
\begin{proof}
For all $E \in \calL^B(P)$, $\bigcap_{\sigma \in \Sigma^B_R(P)} \preceq_{L_{\sigma}} \, \subseteq \, \preceq_E$, which implies $E \in \calL^B( \poset(\Sigma^B_R(P)))$.
Hence, $\calL^B(P) \subseteq \calL^B( \poset(\Sigma^B_R(P)))$.

To prove the opposite inclusion, take any $E \in \calL^B(\poset(\Sigma^B_R(P)))$.
Let $i,j \in [-n,n]$ with $i \preceq_P j$.
Then $i \preceq_{L_\sigma} j$ for all $\sigma \in \Sigma^B_R(P)$, and so $i \preceq_{\poset(\Sigma^B_R(P))} j$.
Since $E$ is a type-$B$ linear extension of $\poset(\Sigma^B_R(P))$, it follows that $i \preceq_{E} j$.
Since $i, j \in [-n,n]$ with $i \preceq_P j$ were arbitrary, this shows that $E \in \calL^B(P)$.
Therefore, $\calL^B( \poset(\Sigma^B_R(P))) \subseteq \calL^B(P)$.
\end{proof}

\begin{theorem}\label{thm: dist are representatives}
For any $B_n$ poset $P$, $\poset(\Sigma^B_R(P))$ is the unique distinguished $B_n$ poset in the equivalence class of $P$ under $\sim$.
\end{theorem}
\begin{proof}
Due to \cref{lem: intersection of Bn linear order} and \cref{lem: poset have same calL}, it suffices to show that for any two distinct distinguished $B_n$ posets $P$ and $Q$, we have $\calL^B(P)\neq \calL^B(Q)$.

Let $P$ and $Q$ be distinguished $B_n$ posets such that $P \neq Q$.
Without loss of generality, we may assume that there exist $x,y \in [-n,n]$ such that $x \neq 0$, $x \prec_P y$, and $x \not\prec_Q y$. 
Since every $E \in \calL^B(P)$ satisfies $x \prec_E y$, it suffices to show that there exists $F \in \calL^B(Q)$ such that $y \prec_F x$.

If $y \prec_Q x$, then we immediately have $\calL^B(P)\neq \calL^B(Q)$.
Assume that $x$ and $y$ are incomparable in $Q$.

Suppose that $y = 0$ or $y = -x$.
Since $Q$ is a distinguished $B_n$ poset and $x$ is incomparable to $y$ in $Q$, $x$ is incomparable to $0$ and $-x$ in $Q$.
By \cref{lem: x y incomp}, there exists $F \in \calL^B(Q)$ such that $y \prec_F x$.

Suppose that $y$ is neither $0$ nor $-x$. 
There are three cases to consider. If $x \not\prec_Q -x$, then \cref{lem: x y incomp} implies that there exists $F \in \calL^B(Q)$ such that $y \prec_F x$.
Similarly, if $-y \not\prec_Q y$, then there exists $F \in \calL^B(Q)$ such that $-x \prec_F -y$ by \cref{lem: x y incomp}, and thus $y \prec_F x$. 
The remaining possibility is $x \prec_Q -x$ and $-y \prec_Q y$. In this case, since $Q$ is distinguished, $x \prec_Q 0 \prec_Q -x$ and $-y \prec_Q 0 \prec_Q y$.
This contradicts the assumption that $x$ and $y$ are incomparable in $Q$.
\end{proof}

\subsection{Regular posets}
Bj\"orner--Wachs \cite{BW91} showed that the sets of permutations corresponding to linear extensions of posets on $[n]$ are precisely the \emph{convex} subsets of $\SG_n$. Analogously, the sets of signed permutations corresponding to linear extensions of $B_n$ posets are precisely the convex subsets of $\SG^B_n$. This was proved by Gaetz--Gao \cite{GG22}. Their result is stated in terms of type $B$ posets that are defined in a slightly different way than $B_n$ posets; in particular, the posets used in \cite{GG22} do not contain $0$. Nevertheless, their type $B$ posets are equivalent to the distinguished $B_n$ posets in the following sense: there is a natural bijection from the set of distinguished $B_n$ posets to the type $B$ posets of \cite{GG22} on $[-n,n] \setminus \{0\}$, sending each distinguished $B_n$ poset $P$ to the subposet $P \setminus \{0\}$.  

Moreover, Bj\"orner--Wachs \cite{BW91} defined a subfamily of the posets on $[n]$, called \emph{regular posets}, and showed that the sets of linear extensions of regular posets are precisely the intervals in the right weak Bruhat order on $\SG_n$. 
In this subsection, we define regular $B_n$ posets in an analogous manner to the regular posets of \cite{BW91}, and prove that regular $B_n$ posets characterize the intervals in the right weak Bruhat order on $\SG^B_n$.

The \emph{permutohedron of type $B_n$} is the undirected graph whose vertices are the signed permutations in $\SG^B_n$, and whose edges are $(\sigma_1, \sigma_2)$ for signed permutations $\sigma_1$ and $\sigma_2$ such that $\sigma_1$ covers $\sigma_2$ or $\sigma_2$ covers $\sigma_1$ in $\preceq_R$.

\begin{definition}
Let $U \subseteq \SG^B_n$.
\begin{enumerate}[label = {\rm (\arabic*)}]
\item 
We say $U$ is \emph{convex} if, for all $\sigma_1, \sigma_2 \in U$, every shortest path in the permutohedron of type $B_n$ connecting $\sigma_1$ and $\sigma_2$ lies in $U$.

\item
The \emph{convex hull} of $U$, denoted $\Conv(U)$, is the intersection of all convex subsets of $\SG^B_n$ that contain $U$.
\end{enumerate}
\end{definition}

\begin{proposition}{\rm (\cite[Proposition 4.14]{GG22})}
\label{prop: convex and poset}
For $U \subseteq \SG^B_n$, we have $\Conv(U) = \Sigma^B_R(\poset(U))$.
In particular, $U$ is convex if and only if $U = \Sigma^B_R(\poset(U))$.
\end{proposition}

Combining \cref{lem: poset have same calL} and \cref{prop: convex and poset} yields the following.

\begin{corollary}\label{cor: P convex}
For any $B_n$ poset $P$, $\Sigma^B_R(P)$ is a convex subset of $\SG^B_n$.
\end{corollary}

We now consider intervals in right weak Bruhat order. We begin by collecting some notation and facts that we will need in the proof of \cref{thm: regular and interval}, the main result of this subsection.  
The following is well-known, e.g., \cite[Section 10]{BW88}. 

\begin{lemma}\label{lem: Conv and interval}
For every $\gamma_1, \gamma_2 \in \SG^B_n$ with $\gamma_1 \preceq_R \gamma_2$, $\Conv(\{\gamma_1, \gamma_2\}) = [\gamma_1, \gamma_2]_R$.
\end{lemma}

Let $K^B_n$ denote the set of reflections in $\SG^B_n$. By \cite[Proposition 8.1.5]{BB05},
\begin{align}\label{eq: Prop 8.1.5 of BB05}
K^B_n = \{ \sft_{(i,j)} \mid 1 \le i < j \le n \} \cup \{ \sft_{(i,j)} \mid 1 \le i \le -j \le n \},
\end{align}
where 
\[
\sft_{(i,j)} := (i,j)(-i, -j) \quad \text{for $1 \le i < |j| \le n$}
\quad \text{and} \quad 
\sft_{(i,-i)} := (i,-i) \quad \text{for $i \in [n]$}.
\]

Given $\sigma \in \SG^B_n$, define  
\[
I^B(\sigma) := \{\sft \in K^B_n \mid \ell(\sft \sigma) < \ell(\sigma) \},
\]
where $\sft$ (without a subscript) denotes an arbitrary element of $K^B_n$. 

Combining \cref{eq: Prop 8.1.5 of BB05} with \cite[Proposition 8.1.6]{BB05} yields
\begin{align}\label{eq: inversion description}
I^B(\sigma) = 
\left\{
\sft_{(i,j)} \in K^B_n \; \middle| \;
\begin{array}{l}
\text{$1 \le i < j \le n$ and $\sigma^{-1}(i) > \sigma^{-1}(j)$, or} \\[1ex]
\text{$1 \le i \le -j \le n$ and $\sigma^{-1}(j) > \sigma^{-1}(i)$}
\end{array}
\right\}.
\end{align}
Let $\sigma, \rho \in \SG^B_n$. By \cite[Proposition 3.1.3]{BB05}, 
\begin{align}\label{eq: weak order and inversion}
\sigma \preceq_R \rho
\quad \text{if and only if} \quad
I^B(\sigma) \subseteq I^B(\rho).
\end{align}

We now introduce a subfamily of distinguished $B_n$ posets. We call these posets \emph{regular}, following the analogous definition of regular posets on $[n]$ \cite[p. 110]{BW91}. We will show that regular $B_n$ posets are precisely the distinguished $B_n$ posets $P$ for which $\Sigma^B_R(P)$ is a right weak Bruhat interval.

\begin{definition}
Given a $B_n$ poset $P$, we say that $P$ is \emph{regular} if $P$ is a distinguished $B_n$ poset and for all $x,y,z \in [-n,n]$ with $x < y < z$, 
\begin{align}\label{eq: regularity}
&\text{$x \prec_P y$ or $y \prec_P z$ if $x \prec_P z$}
\quad \text{and} \quad
\text{$z \prec_P y$ or $y \prec_P x$ if $z \prec_P x$}. 
\end{align}
\end{definition}

Let $P$ be a regular $B_n$ poset.
Towards proving that $\Sigma^B_R(P)$ is a right weak Bruhat interval, we define elements $\upsigma_P$ and $\uprho_P$ of $\SG^B_n$ via the following algorithm.

\begin{algorithm}\label{alg: upsig and up rho}
Let $Q_1^{(0)}$ and $Q_2^{(0)}$ be the subposets of $P$ with underlying sets
\begin{align*}
Q_1^{(0)}
& = \{x \in [-n,n] \mid 0 \prec_P x\} \cup \{x \in [n] \mid \text{$x$ is incomparable to $0$}\}, \\
Q_2^{(0)}
& = \{x \in [-n,n] \mid 0 \prec_P x\} \cup \{x \in [-n,-1] \mid \text{$x$ is incomparable to $0$}\},
\end{align*}
and set $k := 1$.

\begin{enumerate}[label = {\it Step \arabic*.}]
\item 
Let $a_{k}$ be the smallest integer among the minimal elements of $Q_1^{(k-1)}$ and $b_{k}$ be the largest integer among the minimal elements of $Q_2^{(k-1)}$.
Set $Q_1^{(k)} := Q_1^{(k-1)} \setminus \{a_{k}\}$ and $Q_2^{(k)} := Q_2^{(k-1)} \setminus \{b_{k}\}$.

\item
If $k < n$, then set $k := k+1$ and go to {\it Step 1}; otherwise, go to {\it Step 3}.

\item 
Define $\upsigma_P := [a_1,a_2,\ldots, a_n] \in \SG^B_n$ and $\uprho_P := [b_1,b_2,\ldots, b_n] \in \SG^B_n$.
Return $\upsigma_P$ and $\uprho_P$, and terminate the algorithm.
\end{enumerate}
\end{algorithm}
If $P$ is clear from context, we will drop the subscript $P$ from $\upsigma_P$ and $\uprho_P$.

\begin{example}
Let 
\[
\def \hp {0.6}
\def \vp {0.8}
P = 
\begin{array}{l}
\begin{tikzpicture}

\node at (0*\hp, 4*\vp) {$\bullet$};
\node[right] at (0*\hp, 4*\vp) {\scriptsize $4$};

\draw[] (0*\hp, 4*\vp) -- (-1*\hp, 3*\vp);
\draw[] (0*\hp, 4*\vp) -- (1*\hp, 3*\vp);

\node at (-1*\hp, 3*\vp) {$\bullet$};
\node[right] at (-1*\hp, 3*\vp) {\scriptsize $-3$};

\node at (1*\hp, 3*\vp) {$\bullet$};
\node[right] at (1*\hp, 3*\vp) {\scriptsize $1$};

\draw[] (-1*\hp, 3*\vp) -- (-2*\hp, 2*\vp);
\draw[] (-1*\hp, 3*\vp) -- (0*\hp, 2*\vp);
\draw[] (1*\hp, 3*\vp) -- (0*\hp, 2*\vp);
\draw[] (1*\hp, 3*\vp) --(2*\hp, 2*\vp);

\node at (-2*\hp, 2*\vp) {$\bullet$};
\node[right] at (-2*\hp, 2*\vp) {\scriptsize $-2$};

\node at (0*\hp, 2*\vp) {$\bullet$};
\node[right] at (0*\hp, 2*\vp) {\scriptsize $0$};

\node at (2*\hp, 2*\vp) {$\bullet$};
\node[right] at (2*\hp, 2*\vp) {\scriptsize $2$};

\draw[] (-2*\hp, 2*\vp) -- (-1*\hp, 1*\vp);
\draw[] (0*\hp, 2*\vp) -- (-1*\hp, 1*\vp);
\draw[] (0*\hp, 2*\vp) -- (1*\hp, 1*\vp);
\draw[] (2*\hp, 2*\vp) --(1*\hp, 1*\vp);

\node at (-1*\hp, 1*\vp) {$\bullet$};
\node[right] at (-1*\hp, 1*\vp) {\scriptsize $-1$};

\node at (1*\hp, 1*\vp) {$\bullet$};
\node[right] at (1*\hp, 1*\vp) {\scriptsize $3$};

\draw[] (-1*\hp, 1*\vp) -- (0*\hp, 0*\vp);
\draw[] (1*\hp, 1*\vp) -- (0*\hp, 0*\vp);

\node at (0*\hp, 0*\vp) {$\bullet$};
\node[right] at (0*\hp, 0*\vp) {\scriptsize $-4$};

\node at (4*\hp, 2.5*\vp) {$\bullet$};
\node[right] at (4*\hp, 2.5*\vp) {\scriptsize $5$};

\node at (4*\hp, 1.5*\vp) {$\bullet$};
\node[right] at (4*\hp, 1.5*\vp) {\scriptsize $6$};

\draw[] (4*\hp, 2.5*\vp) -- (4*\hp, 1.5*\vp);

\node at (5.5*\hp, 2.5*\vp) {$\bullet$};
\node[right] at (5.5*\hp, 2.5*\vp) {\scriptsize $-6$};

\node at (5.5*\hp, 1.5*\vp) {$\bullet$};
\node[right] at (5.5*\hp, 1.5*\vp) {\scriptsize $-5$};

\draw[] (5.5*\hp, 2.5*\vp) -- (5.5*\hp, 1.5*\vp);

\end{tikzpicture}
\end{array}.
\]
Note that $P$ is a regular $B_6$ poset. We have 
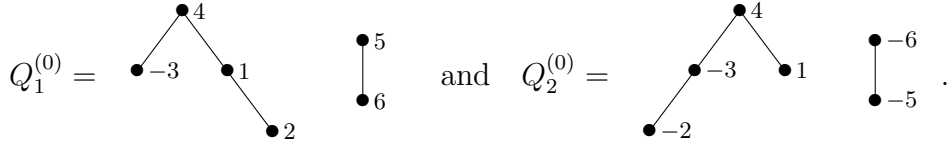
\[
\def \hp {0.6}
\def \vp {0.8}
Q^{(0)}_1 = 
\begin{array}{l}
\begin{tikzpicture}

\node at (0*\hp, 4*\vp) {$\bullet$};
\node[right] at (0*\hp, 4*\vp) {\scriptsize $4$};

\draw[] (0*\hp, 4*\vp) -- (-1*\hp, 3*\vp);
\draw[] (0*\hp, 4*\vp) -- (1*\hp, 3*\vp);

\node at (-1*\hp, 3*\vp) {$\bullet$};
\node[right] at (-1*\hp, 3*\vp) {\scriptsize $-3$};

\node at (1*\hp, 3*\vp) {$\bullet$};
\node[right] at (1*\hp, 3*\vp) {\scriptsize $1$};

\draw[] (1*\hp, 3*\vp) --(2*\hp, 2*\vp);

\node at (2*\hp, 2*\vp) {$\bullet$};
\node[right] at (2*\hp, 2*\vp) {\scriptsize $2$};

\node at (4*\hp, 3.5*\vp) {$\bullet$};
\node[right] at (4*\hp, 3.5*\vp) {\scriptsize $5$};

\node at (4*\hp, 2.5*\vp) {$\bullet$};
\node[right] at (4*\hp, 2.5*\vp) {\scriptsize $6$};

\draw[] (4*\hp, 3.5*\vp) -- (4*\hp, 2.5*\vp);

\end{tikzpicture}
\end{array}
\quad \text{and} \quad
Q^{(0)}_2 = 
\begin{array}{l}
\begin{tikzpicture}

\node at (0*\hp, 4*\vp) {$\bullet$};
\node[right] at (0*\hp, 4*\vp) {\scriptsize $4$};

\draw[] (0*\hp, 4*\vp) -- (-1*\hp, 3*\vp);
\draw[] (0*\hp, 4*\vp) -- (1*\hp, 3*\vp);

\node at (-1*\hp, 3*\vp) {$\bullet$};
\node[right] at (-1*\hp, 3*\vp) {\scriptsize $-3$};

\node at (1*\hp, 3*\vp) {$\bullet$};
\node[right] at (1*\hp, 3*\vp) {\scriptsize $1$};

\draw[] (-1*\hp, 3*\vp) --(-2*\hp, 2*\vp);

\node at (-2*\hp, 2*\vp) {$\bullet$};
\node[right] at (-2*\hp, 2*\vp) {\scriptsize $-2$};

\node at (3*\hp, 3.5*\vp) {$\bullet$};
\node[right] at (3*\hp, 3.5*\vp) {\scriptsize $-6$};

\node at (3*\hp, 2.5*\vp) {$\bullet$};
\node[right] at (3*\hp, 2.5*\vp) {\scriptsize $-5$};

\draw[] (3*\hp, 3.5*\vp) -- (3*\hp, 2.5*\vp);

\end{tikzpicture}
\end{array}.
\]
Applying {\it Step 1} and {\it Step 2} of \cref{alg: upsig and up rho}, we compute the $a_k$'s and $b_k$'s as follows:
\[
\begin{array}{>{\centering\arraybackslash$} p{0.8cm} <{$}||
>{\centering\arraybackslash$} p{0.8cm} <{$}|
>{\centering\arraybackslash$} p{0.8cm} <{$}|
>{\centering\arraybackslash$} p{0.8cm} <{$}|
>{\centering\arraybackslash$} p{0.8cm} <{$}|
>{\centering\arraybackslash$} p{0.8cm} <{$}|
>{\centering\arraybackslash$} p{0.8cm} <{$} }
k & 1 & 2 & 3 & 4 & 5 & 6 \\ \hline
a_k & -3 & 2 & 1 & 4 & 6 & 5 \\ \hline
b_k & 1 & -2 & -3 & 4 & -5 & -6   
\end{array}
\]
Thus, applying {\it Step 3}, $\upsigma_P = [-3,2,1,4,6,5]$ and $\uprho_P = [1,-2,-3,4,-5,-6]$.
\end{example}

We now collect some facts that will be helpful in the proof of \cref{thm: regular and interval}. We adopt the notation from \cref{alg: upsig and up rho} and for convenience, we omit the superscript $(0)$ from $Q_1^{(0)}$ and $Q_2^{(0)}$.

\begin{lemma}\label{lem: x incomparable to y}
Let $P$ be a regular $B_n$ poset. For any $z\in [-n,n]$ we have 
\begin{enumerate}[label = {\rm (\alph*)}]
\item 
$z \in Q_1 \cap Q_2$ if and only if $0 \prec_P z$,
\item 
$z \in Q_1$ and $z \notin Q_2$ if and only if $z$ is incomparable to $0$ in $P$ and $z > 0$,
\item 
$z \notin Q_1$ and $z \in Q_2$ if and only if $z$ is incomparable to $0$ in $P$ and $z < 0$,
\item
$z \notin Q_1 \cup Q_2$ if and only if $z \preceq_P 0$.
\end{enumerate}
Moreover, for any $x,y\in [-n,n]\setminus\{0\}$ such that $x<y$ and $x$ is incomparable to $y$ in $P$, we have
\begin{enumerate}[label = {\rm (\alph*)}]
\setcounter{enumi}{4}
\item 
$x \prec_{L_{\upsigma}} y$ if and only if $x \notin Q_1$ or $y \in Q_1$,
\item
 $y \prec_{L_{\upsigma}} x$ if and only if $x \in Q_1$ and $y \notin Q_1$,
\item
$x \prec_{L_{\uprho}} y$ if and only if $x \notin Q_2$ and $y \in Q_2$,
\item
 $y \prec_{L_{\uprho}} x$ if and only if $x \in Q_2$ or $y \notin Q_2$.
\end{enumerate}
\end{lemma}
\begin{proof}
The equivalences (a), (b), (c) and (d) follow immediately from the definition of $Q_1$ and $Q_2$.
Clearly (e) is equivalent to (f) and (g) is equivalent to (h). We will prove (e); the argument for (g) is similar. 
Let $x,y\in [-n,n]\setminus\{0\}$ such that $x<y$ and $x$ is incomparable to $y$ in $P$.
Consider the following four cases:
\[
\textrm{(i)}\,\, x,y \in Q_1,
\quad 
\textrm{(ii)}\,\, x \in Q_1 \text{ and } y \notin Q_1,
\quad 
\textrm{(iii)}\,\, x \notin Q_1 \text{ and } y \in Q_1,
\quad 
\textrm{(iv)}\,\, x,y \notin Q_1.
\]
 
\noindent
{\it Case {\rm (i)}.}
Suppose for a contradiction that $y \prec_{L_{\upsigma}} x$. Then we have $x = a_i$ and $y = a_j$ for $1 \le j < i \le n$.
Since $j$ is smaller than $i$, $x$ and $y$ are in $Q^{(j-1)}_1$ and $y = a_j$ is the smallest integer among the minimal elements of $Q^{(j-1)}_1$.
In addition, since $x < y$, there exists a minimal element $x'$ in $Q^{(j-1)}_1$ such that $x' \prec_P x$ and $y < x'$. (Note $x'\neq y$ since $x$ and $y$ are incomparable in $P$.)
Putting these together, we have $x < y < x'$ and $x' \prec_P x$.
So, by the regularity of $P$, either $x' \prec_P y$ or $y \prec_P x$.
But the former relation cannot occur because $y$ is minimal in $Q^{(j-1)}_1$, and the latter relation also cannot occur because $x$ and $y$ are incomparable in $P$.
Therefore, in this case we must have $x \prec_{L_{\upsigma}} y$.
\smallskip

\noindent
{\it Case {\rm (ii)}.} 
In this case, we have $y \prec_{L_{\upsigma}} 0 \prec_{L_{\upsigma}} x$.
\smallskip

\noindent
{\it Case {\rm (iii)}.}
In this case, we have $x \prec_{L_{\upsigma}} 0 \prec_{L_{\upsigma}} y$.
\smallskip

\noindent
{\it Case {\rm (iv)}.}
In this case, we have $-y < -x$  and $-x,-y \in Q_1$.
Using the argument in Case (i), we have $-y \prec_{L_{\upsigma}} -x$, and therefore $x \prec_{L_{\upsigma}} y$.

Thus $x\prec_{L_{\upsigma}} y$ if and only if we are in cases (i), (iii) or (iv), which establishes (e). 
\end{proof}

\begin{theorem}\label{thm: regular and interval}
We have the following.
\begin{enumerate}[label = {\rm (\arabic*)}]
\item 
For every regular $B_n$ poset $P$, $\Sigma^B_R(P) = [\upsigma_P, \uprho_P]_R$.
\item 
For every $\sigma,\rho \in \SG^B_n$ such that $\sigma \preceq_R \rho$, the $B_n$ poset $\poset(\{\sigma,\rho\})$ is regular and $\Sigma^B_R(\poset(\{\sigma,\rho\})) = [\sigma,\rho]_R$. 
\end{enumerate}
\end{theorem}
\begin{proof}
(1)
Let $P$ be a regular $B_n$ poset. We will show that $\upsigma, \uprho \in  \Sigma^B_R(P)$, $\upsigma \preceq_R \uprho$, and $P = \poset(\{\upsigma,\uprho\})$. Once these facts are established, by \cref{prop: convex and poset} and
\cref{lem: Conv and interval} we have
\[
\Sigma^B_R(P) 
= \Sigma^B_R(\poset(\{\upsigma, \uprho\})) 
= \Conv(\{\upsigma, \uprho\}) 
= [\upsigma, \uprho]_R.
\]

First we show that $\upsigma \in \Sigma^B_R(P)$. 
Since in \cref{alg: upsig and up rho} the elements $a_1,a_2,\ldots, a_n$ are chosen according to the partial order $\preceq_P$ and $\upsigma$ is defined by the signed permutation $[a_1,a_2,\ldots, a_n]$, it suffices to show that 
\begin{align}\label{eq: upsig membership condition claim}
x \not \prec_P y \quad \text{for all $x \in Q_1$ and $y \in [-n,n] \setminus Q_1$}. 
\end{align}

Take any $x \in Q_1$ and $y \in [-n,n] \setminus Q_1$.
For the sake of contradiction, suppose that $x \prec_P y$.
If $y \preceq_P 0$, then we have $x \prec_P 0$, which contradicts the choice of $x \in Q_1$.
So, $y$ is incomparable to $0$ in $P$.
Moreover, if $0 \prec_P x$, then the assumption $x \prec_P y$ implies $0 \prec_P y$, which contradicts the previous observation that $y$ is incomparable to $0$.
So, $x$ is also incomparable to $0$.
Then, by the definition of $Q_1$, we have $y < 0 < x$.
Since $P$ is regular and $x \prec_P y$, it follows that either $x \prec_P 0$ or $0 \prec_P y$, which is a contradiction.
Thus, \cref{eq: upsig membership condition claim} holds, and so $\upsigma \in \Sigma^B_R(P)$. A similar argument shows that $\uprho \in \Sigma^B_R(P)$.

Next, we show that $\upsigma \preceq_R \uprho$, equivalently, $I^B(\upsigma) \subseteq I^B(\uprho)$.
Take any $\sft_{(i,j)} \in I^B(\upsigma)$.
By the definition of $I^B(\upsigma)$, we have 
\[
1 \le i < j\le n \text{ and } \upsigma^{-1}(i) > \upsigma^{-1}(j)
\quad \text{or} \quad
1 \le i \le -j \le n \text{ and } \upsigma^{-1}(j) > \upsigma^{-1}(i).
\]
We will prove $\sft_{(i,j)} \in I^B(\uprho)$ in the first of these two cases; the argument for the second case is similar. Suppose that $1 \le i < j\le n$ and $\upsigma^{-1}(i) > \upsigma^{-1}(j)$. In this case, we have $j \prec_{L_{\upsigma}} i$. 
If $i$ and $j$ are comparable in $P$, then since $\upsigma, \uprho \in \Sigma_R^B(P)$, we must also have $j \prec_{L_{\uprho}} i$,  that is, $\uprho^{-1}(i) > \uprho^{-1}(j)$, and therefore $\sft_{(i,j)} \in I^B(\uprho)$. 
If $i$ and $j$ are incomparable in $P$, suppose for a contradiction that $\sft_{(i,j)} \notin I^B(\uprho)$. Then $\uprho^{-1}(i)<\uprho^{-1}(j)$, equivalently, $i\prec_{L_{\uprho}} j$. From \cref{lem: x incomparable to y} (f) and (g), we have $i \in Q_1$, $j \notin Q_1$, $i \notin Q_2$, and $j \in Q_2$, and then from \cref{lem: x incomparable to y} (b) and (c) we have $i>0$ and $j<0$, contradicting the assumption $i<j$. 
Thus $\upsigma \preceq_R \uprho$. It now follows from \cref{prop: convex and poset} and \cref{lem: Conv and interval} that 
\[\Sigma^B_R(\poset(\{\upsigma, \uprho\})) 
= \Conv(\{\upsigma, \uprho\}) 
= [\upsigma, \uprho]_R.
\]

Finally, we show that $P = \poset(\{\upsigma,\uprho\})$. Since $\Sigma^B_R(P)$ is convex (\cref{cor: P convex}) and $\upsigma, \uprho \in  \Sigma^B_R(P)$, we have $\Sigma^B_R(\poset(\{\upsigma, \uprho\})) = \Conv(\{\upsigma, \uprho\}) \subseteq \Sigma^B_R(P)$. 
Moreover, by \cref{lem: intersection of Bn linear order}(2) we have $\preceq_P \; = \; \preceq_{\poset(\Sigma^B_R(P))}$. 
Thus we have 
\begin{equation}\label{eqn: preceq containment}
\preceq_P \; \subseteq \; \preceq_{\poset(\{\upsigma,\uprho\})}.
\end{equation}

We will prove that $\preceq_{\poset(\{\upsigma,\uprho\})} \; \subseteq \; \preceq_P$.
Let $x,y \in [-n,n]$ such that $x\neq y$ and suppose $(x,y)\notin \; \preceq_P$. Then either $(y,x) \in \; \preceq_P$ or $x$ and $y$ are incomparable in $P$. If $(y,x) \in \; \preceq_P$, then by \cref{eqn: preceq containment} we have $(y,x) \in \;  \preceq_{\poset(\{\upsigma,\uprho\})}$, so $(x,y)\notin \; \preceq_{\poset(\{\upsigma,\uprho\})}$.

Hence, suppose $x$ and $y$ are incomparable in $P$. We now consider the case $x<y$. 
If $x=0$, then $y$ is incomparable to $0$ in $P$ and $y>0$. By \cref{lem: x incomparable to y} (b), $y\in Q_1$ and $y\notin Q_2$, which implies $x=0 \prec_{L_{\upsigma}} y$ and $y \prec_{L_{\uprho}} 0=x$. Thus, $x$ and $y$ are incomparable in $\poset(\{\upsigma,\uprho\})$. If $y=0$, then $x$ is incomparable to $0$ in $P$ and $x<0$. By \cref{lem: x incomparable to y} (c), $x\notin Q_1$ and $x\in Q_2$, which implies $x \prec_{L_{\upsigma}} 0=y$ and $y=0 \prec_{L_{\uprho}} x$. Thus, $x$ and $y$ are incomparable in $\poset(\{\upsigma,\uprho\})$. So $(x,y)\notin \; \preceq_{\poset(\{\upsigma,\uprho\})}$ if either $x$ or $y$ is $0$.

Now, suppose that neither $x$ nor $y$ are $0$, and suppose for a contradiction that $(x,y) \in \; \preceq_{\poset(\{\upsigma,\uprho\})}$; in particular, $x \prec_{L_{\upsigma}} y$ and $x \prec_{L_{\uprho}} y$. 
Then by \cref{lem: x incomparable to y} (g) we have $x\notin Q_2$ and $y\in Q_2$, and by \cref{lem: x incomparable to y} (e), we have one of the following: 
\[
x,y \in Q_1,
\quad
x,y \notin Q_1, 
\quad \text{or} \quad
x \notin Q_1 \text{ and } y \in Q_1.
\]
If $x,y \in Q_1$, then by \cref{lem: x incomparable to y} (b) and (a) we have 
\[
\text{$x$ is incomparable to $0$ in $P$}, 
\quad  x > 0,
\quad \text{and} \quad 0 \prec_P y.
\]
Since $0<x<y$ and $0$ is comparable to $y$ in $P$, the regularity of $P$ implies that $x$ is comparable to either $0$ or $y$ in $P$, which is a contradiction because $x$ is incomparable to both $0$ and $y$ in $P$.

If $x,y \notin Q_1$, then by \cref{lem: x incomparable to y} (d) and (c) we have 
\[
x \prec_P 0,
\quad
\text{$y$ is incomparable to $0$ in $P$}, 
\quad \text{and} \quad  y < 0.
\]
Since $x<y<0$ and $x$ is comparable to $0$ in $P$, the regularity of $P$ implies that $y$ is comparable to either $x$ or $0$ in $P$, which is a contradiction because $y$ is incomparable to both $0$ and $x$ in $P$.

If $x \notin Q_1$ and $y \in Q_1$, then by \cref{lem: x incomparable to y} (d) and (a) we have that $x \prec_P 0 \prec_P y$, which contradicts the assumption that $x$ and $y$ are incomparable in $P$.

Thus we cannot have both $x \prec_{L_{\upsigma}} y$ and $x \prec_{L_{\uprho}} y$. In particular, $(x,y)\notin \; \preceq_{L_{\upsigma}}\cap \preceq_{L_{\uprho}}$, so $(x,y)\notin \; \preceq_{\poset(\{\upsigma,\uprho\})}$. 
Therefore, $P = \poset(\{\upsigma,\uprho\})$. 
We omit the proof for the case $x > y$, since it is similar to the proof for the case $x < y$, using \cref{lem: x incomparable to y}(h), (f) instead of \cref{lem: x incomparable to y}(g), (e). This concludes the proof of (1). 
\medskip

(2) 
Let $\sigma,\rho \in \SG^B_n$ such that $\sigma \preceq_R \rho$, that is, $I^B(\sigma) \subseteq I^B(\rho)$. It follows from \cref{prop: convex and poset} and \cref{lem: Conv and interval} that $\Sigma^B_R(\poset(\{\sigma,\rho\}))=[\sigma,\rho]_R$.

For convenience, let $P$ denote $\poset(\{\sigma, \rho\})$. By \cref{lem: intersection of Bn linear order}(1), $P$ is a distinguished poset.
Suppose for a contradiction that $P$ is not regular. 
Then we can choose $x,y,z \in [-n,n]$ such that $x < y < z$ and $x$ and $z$ are comparable in $P$, but $x$ and $z$ are incomparable to $y$ in $P$.
Since $x$ is incomparable to $y$ in $P$, there must exist $\gamma \in [\sigma, \rho]_R$ such that the order of $x$ and $y$ in $L_\gamma$ is the reverse of that in $L_\sigma$, and then since $I^B(\sigma)\subseteq I^B(\gamma)\subseteq I^B(\rho)$,  the order of $x$ and $y$ in $L_\rho$ agrees with that in $L_\gamma$. Therefore we have
\[
{\rm (i)} \ \ x \prec_{L_\sigma} y \ \ \text{and} \ \ y \prec_{L_\rho} x
\quad \text{or} \quad
{\rm (ii)} \ \ y \prec_{L_\sigma} x \ \ \text{and} \ \ x \prec_{L_\rho} y
\]
and similarly, since $z$ is incomparable to $y$ in $P$, we have
\[
{\rm (iii)} \ \ y \prec_{L_\sigma} z \ \ \text{and} \ \  z \prec_{L_\rho} y
\quad \text{or} \quad
{\rm (iv)} \ \ z \prec_{L_\sigma} y \ \ \text{and} \ \  y \prec_{L_\rho} z.
\]
If (i) and (iii) hold, then we have $x \prec_{L_\sigma} z$ and $z \prec_{L_\rho} x$. 
This implies that $x$ and $z$ are incomparable in $P$, which is a contradiction. 
Similarly, if (ii) and (iv) hold, then $x$ and $z$ are incomparable in $P$, yielding a contradiction.
The remaining cases are that either (ii) and (iii) hold or (i) and (iv) hold.

In the case where (ii) and (iii) hold, we consider six cases and, in each case, we use (ii) to derive a contradiction to $I^B(\sigma) \subseteq I^B(\rho)$ as follows:
\begin{enumerate}[label = -, leftmargin = 4ex]
\item 
If $0 < x < y$, then $\sft_{(x,y)} \in I^B(\sigma)$, but $\sft_{(x,y)} \notin I^B(\rho)$.

\item 
If $x = 0 < y$, then $\sft_{(y,-y)} \in I^B(\sigma)$, but $\sft_{(y,-y)} \notin I^B(\rho)$.

\item 
If $x < 0 < -x < y$, then $\sft_{(-x,-y)} \in I^B(\sigma)$, but $\sft_{(-x,-y)} \notin I^B(\rho)$.

\item 
If $x < 0 < y \le -x$, then $\sft_{(y,x)} \in I^B(\sigma)$, but $\sft_{(y,x)} \notin I^B(\rho)$. 

\item 
If $x < y = 0$, then $\sft_{(-x,x)} \in I^B(\sigma)$, but $\sft_{(-x,x)} \notin I^B(\rho)$.

\item 
If $x<y<0$, then $\sft_{(-y,-x)} \in I^B(\sigma)$, but $\sft_{(-y,-x)} \notin I^B(\rho)$.
\end{enumerate}

Similarly, in the case where (i) and (iv) hold, we consider six cases and, in each case, we use (iv) to derive a contradiction to $I^B(\sigma) \subseteq I^B(\rho)$ as follows:
\begin{enumerate}[label = -, leftmargin = 4ex]
\item 
If $0 < y < z$, then $\sft_{(y,z)} \in I^B(\sigma)$, but $\sft_{(y,z)} \notin I^B(\rho)$.

\item 
If $y = 0 < z$, then $\sft_{(z,-z)} \in I^B(\sigma)$, but $\sft_{(z,-z)} \notin I^B(\rho)$.

\item 
If $y < 0 < -y < z$, then $\sft_{(-y,-z)} \in I^B(\sigma)$, but $\sft_{(-y,-z)} \notin I^B(\rho)$.

\item 
If $y < 0 < z \le -y$, then $\sft_{(z,y)} \in I^B(\sigma)$, but $\sft_{(z,y)} \notin I^B(\rho)$.

\item 
If $y < z = 0$, then $\sft_{(-y,y)} \in I^B(\sigma)$, but $\sft_{(-y,y)} \notin I^B(\rho)$.

\item 
If $y<z<0$, then $\sft_{(-z,-y)} \in I^B(\sigma)$, but $\sft_{(-z,-y)} \notin I^B(\rho)$.
\end{enumerate}
\smallskip

Consequently, for all $x,y,z \in [-n,n]$ with $x < y < z$,
if $x$ and $z$ are comparable in $P$, then $y$ is comparable to $x$ or $z$.
Therefore, $P$ is regular. 
\end{proof}

\begin{corollary}
If $P$ is a distinguished $B_n$ poset, then $\Sigma_R^B(P)$ is an interval in the right weak Bruhat order on $\SG^B_n$ if and only if $P$ is regular.     
\end{corollary}
\begin{proof}
One direction is given by \cref{thm: regular and interval}(1). For the other direction, 
let $P$ be a distinguished $B_n$ poset such that $\Sigma^B_R(P) = [\sigma, \rho]_R$ for some $\sigma, \rho \in \SG^B_n$.
By \cref{thm: regular and interval}(2), $[\sigma, \rho]_R$ is the set of linear extensions of a regular $B_n$ poset, and by \cref{thm: dist are representatives} every distinguished poset is determined by its set of linear extensions. 
It follows that $[\sigma, \rho]_R$ cannot be the set of linear extensions of a non-regular distinguished $B_n$ poset.
Therefore, $P$ is regular.
\end{proof}

\subsection{The category of weak Bruhat interval modules}
In this section, we consider the full subcategory $\scrB^B(n)$ of $\scrP^B(n)$ whose objects are direct sums of finitely many isomorphic copies of $B_n$ poset modules for regular $B_n$ posets. In particular, we have
\[
\scrB^B(n) \subseteq \scrP^B(n) \subseteq \modHBn.
\] 
Using \cref{thm: regular and interval}, we identify $\scrB^B(n)$ with the subcategory of $\modHBn$ corresponding to weak Bruhat interval modules of $\SG^B_n$. We show that the Grothendieck group $\bigoplus_{n \ge 0} \calG_0(\scrB^B(n))$ is isomorphic to $\QSym^B$ as $\QSym$-modules and comodules. 

\begin{lemma}\label{lem: disjoint union of intervals}
Every right weak Bruhat interval $I$ in $\SG^B_n$ with $|I|>1$ can be written as a disjoint union of two nonempty right weak Bruhat intervals.
\end{lemma}
\begin{proof}
Let $I=[\sigma, \rho]_R$ be a right weak Bruhat interval in $\SG^B_n$ with $|I|>1$, and consider the right weak Bruhat interval $[\id, \sigma^{-1}\rho]_R$. Since $\sigma^{-1} \rho \neq \id$, we can choose $k \in [0,n-1]$ such that $s_{k} \preceq_R \sigma^{-1} \rho$. Note that $\SG^B_n$ is the disjoint union of $\{\gamma \in \SG^B_n \mid s_k \notin \Des_L(\gamma)\}$ and $\{\gamma \in \SG^B_n \mid s_k \in \Des_L(\gamma)\}$. It follows from \cite[Theorem 6.2(2)]{BW88} that the former set is equal to the right weak Bruhat interval $[\id, s_k w_0]_R$, and the latter set is equal to the right weak Bruhat interval $[s_k, w_0]_R$.

Therefore, $I_1 := [\id, \sigma^{-1}\rho]_R \cap  [\id, s_k w_0]_R$ and $I_2:= [\id, \sigma^{-1}\rho]_R \cap [s_k, w_0]_R$ are disjoint right weak Bruhat intervals whose union is $[\id, \sigma^{-1}\rho]_R$. Moreover, both $I_1$ and $I_2$ are nonempty because $\id\in [\id, s_k w_0]_R$ and, by the choice of $s_k$, $\sigma^{-1}\rho \in [s_k, w_0]_R$. 

By \cite[Proposition 3.1.6]{BB05}, $[\id, \sigma^{-1}\rho]_R \cong I$ as posets. Therefore, the images of $I_1$ and $I_2$ under an isomorphism between $[\id, \sigma^{-1}\rho]_R$ and $I$  are disjoint nonempty right weak Bruhat intervals whose union is $I$. 
\end{proof}

\begin{lemma}\label{lem: regular poset module ses}
Let $P$ be a regular $B_n$ poset with $|\calL^B(P)| > 1$.
There exist two regular $B_n$ posets $P',P''$ such that $|\calL^B(P')|,~|\calL^B(P'')| < |\calL^B(P)|$ and
\[
\begin{tikzcd}
0 \arrow[r] & {M^B_{P'} } \arrow[r] & M^B_P \arrow[r] & {M^B_{P''}} \arrow[r] & 0
\end{tikzcd}
\]
is a short exact sequence of $H^B_n(0)$-modules.
\end{lemma}
\begin{proof}
Let $P$ be a regular $B_n$ poset with $|\calL^B(P)| > 1$.
By \cref{thm: regular and interval}(1), $\Sigma^B_R(P) = [\upsigma_P, \uprho_P]_R$.
By \cref{lem: disjoint union of intervals}, we have two disjoint nonempty right weak Bruhat intervals $I$ and $J$ such that $I \cup J = [\upsigma_P, \uprho_P]_R$. 
By \cref{thm: regular and interval}(2), $\poset(I)$ and $\poset(J)$ are regular. 
Supposing without loss of generality that $\uprho_P\in J$, we claim that
\[
\begin{tikzcd}
0 \arrow[r] & {M^B_{\poset(J)} } \arrow[r] & M^B_P \arrow[r] & {M^B_{\poset(I)}} \arrow[r] & 0
\end{tikzcd}
\]
is exact. 
It suffices to show that $\C J$ is a submodule of $M^B_P$. 
Suppose for the sake of contradiction that there exists $\gamma\in J$ and $i \in [0,n-1]$ such that $\gamma \cdot \opi_i \notin \C J$.
Then
\[
\gamma \cdot \opi_i = \gamma s_i \in [\gamma, \uprho_P]_R \setminus J.
\]
However, since $\gamma, \uprho_P\in J$, we have $[\gamma, \uprho_P]_R\subseteq J$, which contradicts $\gamma s_i \in [\gamma, \uprho_P]_R \setminus J$.
Therefore, $\gamma \cdot \opi_i \in \C J$, and consequently $\C J$ is a submodule of $M^B_P$. 
\end{proof}

Given a poset $P = ([n], \preceq_{P})$, we say that $P$ is \emph{regular} if for all $x,y,z \in [n]$ with $x < y < z$,
$x \prec_P y$ or $y \prec_P z$ if $x \prec_P z$ and $z \prec_P y$ or $y \prec_P x$ if $z \prec_P x$.

\begin{proposition}\label{prop: regular poset modules}
We have the following.
\begin{enumerate}[label = {\rm (\arabic*)}]
\item 
Let $P_1$ be a regular $B_m$ poset and $P_2 = ([n], \preceq_{P_2})$ a regular poset.
Then $P_1 \sqcup_B P_2$ is a regular $B_{m+n}$ poset.

\item 
Let $P$ be a regular $B_n$ poset and $1 \le m \le n$.
For $Q \in \LS^B(P;m)$ and $U \in \upper(Q)$, $\stB(Q)$ is a regular $B_m$ poset and $\rmst(U)$ is a regular poset.

\item 
For any regular $B_n$ poset $P$, $P^*$ is a regular $B_n$ poset.
\end{enumerate}
\end{proposition}

\begin{proof}
(1)
From \cref{eq: calIs} and \cref{eq: def of sqcupB}, recall the definitions of $\calI_{-1}, \calI_0, \calI_1$ and $P_1 \sqcup_B P_2$.
To prove that $P_1 \sqcup_B P_2$ is distinguished, take any $x \in [-m-n, m+n]$ such that $x$ and $-x$ are comparable.
If $x$ is in $\calI_{-1} \cup \calI_1$, then $x$ cannot be comparable to $-x$ in $P_1 \sqcup_B P_2$, so both $-x, x \in \calI_0$.
Since $P_1$ is distinguished, $x$ is comparable to $0$ in $P_1$, therefore, in $P_1 \sqcup_B P_2$.
Thus, $P_1 \sqcup_B P_2$ is distinguished.
Moreover, the property \cref{eq: regularity} for $P_1 \sqcup_B P_2$ is straightforwards from that of $P_1$ and $P_2$.
Therefore, $P_1 \sqcup_B P_2$ is regular.

We omit the proofs of (2) and (3), since they are routine.
\end{proof}

Let $(W,S)$ be a finite Coxeter system and $I$ a right weak Bruhat interval in $W$.
The \emph{weak Bruhat interval module associated with $I$} is the $H_W(0)$-module $\mathsf{B}(I)$ with $\C I$ as the underlying space and with the $H_W(0)$-action defined by
\begin{align*}
\gamma \cdot \opi_{s} :=
\begin{cases}
-\gamma & \text{if $s \in \Des_R(\gamma)$}, \\
0 & \text{if $s \notin \Des_R(\gamma)$ and $\gamma s \notin I$,} \\
\gamma s & \text{if $s \notin \Des_R(\gamma)$ and $\gamma s \in I$}
\end{cases}
\end{align*}
for any  $s \in S$ and $\gamma \in I$.
These modules were first introduced in type $A$ in \cite{JKLO22}, and subsequently studied for arbitrary finite Coxeter types in \cite{BS23, YY24}. 

By \cref{thm: regular and interval}, the set $\{M^B_P \mid \text{$P$ is a regular $B_n$ poset}\}$ coincides with the set of weak Bruhat interval modules of $H^B_n(0)$. Therefore, $\scrB^B(n)$ may be identified with the full subcategory of $\modHBn$ whose objects are direct sums of finitely many isomorphic copies of weak Bruhat interval modules of $\SG^B_n$. 

Let $\scrB(n)$ be the full subcategory of $\modHn$ whose objects are direct sums of finitely many isomorphic copies of weak Bruhat interval modules of $H_n(0)$. 
Through \cite{DKLT96, JKLO22, CKO24} it has been shown that, as Hopf algebras,
\begin{align}\label{eq: type A isomorphisms}
\bigoplus_{n \ge 0} \calG_0(\scrB(n)) 
\cong \bigoplus_{n \ge 0} \calG_0(\scrP(n)) 
\cong \bigoplus_{n \ge 0} \calG_0(\modHn)
\cong \QSym.
\end{align}

The following theorem shows that the type-$B$ analogue of \cref{eq: type A isomorphisms} also holds.

\begin{theorem}\label{thm: Grothendieck module of WBIM}
As right $\QSym$-modules and right $\QSym$-comodules,
\[
\bigoplus_{n \ge 0} \calG_0(\scrB^B(n)) 
\cong \bigoplus_{n \ge 0} \calG_0(\scrP^B(n)) 
\cong \bigoplus_{n \ge 0} \calG_0(\modHBn)
\cong \QSymB.
\]
\end{theorem}

\begin{proof}
In \cref{cor: module isom} and \cref{cor: comodule isom}, we showed
\[
\bigoplus_{n \ge 0} \calG_0(\scrP^B(n)) 
\cong \bigoplus_{n \ge 0} \calG_0(\modHBn)
\]
as right $\QSym$-modules and right $\QSym$-comodules.
Similarly, combining the isomorphisms 
\[
\bigoplus_{n \ge 0} \calG_0(\scrB(n)) 
\cong \bigoplus_{n \ge 0} \calG_0(\modHn) \cong \QSym
\]
with \cref{lem: regular poset module ses} and \cref{prop: regular poset modules}, we have
\[
\bigoplus_{n \ge 0} \calG_0(\scrB^B(n)) 
\cong \bigoplus_{n \ge 0} \calG_0(\modHBn)
\]
as right $\QSym$-modules and right $\QSym$-comodules.
Therefore, by the isomorphism \cref{eq: quasi characteristic of type B}, the assertion follows.
\end{proof}

\section*{Acknowledgements}
The first author was supported by the National Research Foundation of Korea(NRF) grant funded by the Korean Government (NRF-2020R1A5A1016126), Basic Science Research Program through NRF funded by the Ministry of Education (RS-2023-00240377), and a research grant from Seoul Women’s University (2025-0241).
The second author was supported by the Marsden Fund, administered by the Royal Society of New Zealand Te Ap{\= a}rangi.

\bibliographystyle{abbrv}
\bibliography{references}

@book {Chow01,
    AUTHOR = {Chow, Chak-On},
     TITLE = {Noncommutative symmetric functions of type {B}},
      NOTE = {Thesis (Ph.D.)--Massachusetts Institute of Technology},
 PUBLISHER = {ProQuest LLC, Ann Arbor, MI},
      YEAR = {2001},
     PAGES = {(no paging)},
   MRCLASS = {99-05},
  MRNUMBER = {2717011}
}

@book {BB05,
    AUTHOR = {Bj\"orner, Anders and Brenti, Francesco},
     TITLE = {Combinatorics of {C}oxeter groups},
    SERIES = {Graduate Texts in Mathematics},
    VOLUME = {231},
 PUBLISHER = {Springer, New York},
      YEAR = {2005},
     PAGES = {xiv+363},
      ISBN = {978-3540-442387; 3-540-44238-3},
   MRCLASS = {05-01 (05E15 20F55)},
  MRNUMBER = {2133266},
MRREVIEWER = {Jian-yi\ Shi},
}

@article {Huang16,
    AUTHOR = {Huang, Jia},
     TITLE = {A tableau approach to the representation theory of 0-{H}ecke
              algebras},
   JOURNAL = {Ann. Comb.},
  FJOURNAL = {Annals of Combinatorics},
    VOLUME = {20},
      YEAR = {2016},
    NUMBER = {4},
     PAGES = {831--868},
      ISSN = {0218-0006,0219-3094},
   MRCLASS = {20C08 (05E05 05E10 16T30)},
  MRNUMBER = {3572389},
MRREVIEWER = {Michael\ Chmutov},
       DOI = {10.1007/s00026-016-0338-5}
}

@article {Huang17,
    AUTHOR = {Huang, Jia},
     TITLE = {A uniform generalization of some combinatorial {H}opf
              algebras},
   JOURNAL = {Algebr. Represent. Theory},
  FJOURNAL = {Algebras and Representation Theory},
    VOLUME = {20},
      YEAR = {2017},
    NUMBER = {2},
     PAGES = {379--431},
      ISSN = {1386-923X,1572-9079},
   MRCLASS = {16T30 (05E05 20C08 20F55)},
  MRNUMBER = {3638354},
MRREVIEWER = {Ben\ Salisbury},
       DOI = {10.1007/s10468-016-9648-x}
}

@article {Norton79,
    AUTHOR = {Norton, P. N.},
     TITLE = {{$0$}-{H}ecke algebras},
   JOURNAL = {J. Austral. Math. Soc. Ser. A},
  FJOURNAL = {Australian Mathematical Society. Journal. Series A. Pure
              Mathematics and Statistics},
    VOLUME = {27},
      YEAR = {1979},
    NUMBER = {3},
     PAGES = {337--357},
      ISSN = {0263-6115},
   MRCLASS = {16A48 (17B10)},
  MRNUMBER = {532754},
MRREVIEWER = {S.\ B.\ Conlon},
}

@article {BL09,
    AUTHOR = {Bergeron, Nantel and Li, Huilan},
     TITLE = {Algebraic structures on {G}rothendieck groups of a tower of
              algebras},
   JOURNAL = {J. Algebra},
  FJOURNAL = {Journal of Algebra},
    VOLUME = {321},
      YEAR = {2009},
    NUMBER = {8},
     PAGES = {2068--2084},
      ISSN = {0021-8693,1090-266X},
   MRCLASS = {16T30 (05E05 05E15)},
  MRNUMBER = {2501510},
MRREVIEWER = {Mark\ J.\ Wildon},
       DOI = {10.1016/j.jalgebra.2008.12.005},
       URL = {https://doi.org/10.1016/j.jalgebra.2008.12.005},
}

@article {DKLT96,
    AUTHOR = {Duchamp, G\'erard and Krob, Daniel and Leclerc, Bernard and
              Thibon, Jean-Yves},
     TITLE = {Fonctions quasi-sym\'etriques, fonctions sym\'etriques non
              commutatives et alg\`ebres de {H}ecke \`a{} {$q=0$}},
   JOURNAL = {C. R. Acad. Sci. Paris S\'er. I Math.},
  FJOURNAL = {Comptes Rendus de l'Acad\'emie des Sciences. S\'erie I.
              Math\'ematique},
    VOLUME = {322},
      YEAR = {1996},
    NUMBER = {2},
     PAGES = {107--112},
      ISSN = {0764-4442},
   MRCLASS = {05E05 (20C05)},
  MRNUMBER = {1373744},
}

@article {KT97,
    AUTHOR = {Krob, Daniel and Thibon, Jean-Yves},
     TITLE = {Noncommutative symmetric functions. {IV}. {Q}uantum linear
              groups and {H}ecke algebras at {$q=0$}},
   JOURNAL = {J. Algebraic Combin.},
  FJOURNAL = {Journal of Algebraic Combinatorics. An International Journal},
    VOLUME = {6},
      YEAR = {1997},
    NUMBER = {4},
     PAGES = {339--376},
      ISSN = {0925-9899,1572-9192},
   MRCLASS = {05E05 (17B37 20C30)},
  MRNUMBER = {1471894},
MRREVIEWER = {Jian-yi\ Shi},
       DOI = {10.1023/A:1008673127310},
       URL = {https://doi.org/10.1023/A:1008673127310},
}

@article {DHT02,
    AUTHOR = {Duchamp, G\'erard and Hivert, Florent and Thibon, Jean-Yves},
     TITLE = {Noncommutative symmetric functions. {VI}. {F}ree
              quasi-symmetric functions and related algebras},
   JOURNAL = {Internat. J. Algebra Comput.},
  FJOURNAL = {International Journal of Algebra and Computation},
    VOLUME = {12},
      YEAR = {2002},
    NUMBER = {5},
     PAGES = {671--717},
      ISSN = {0218-1967,1793-6500},
   MRCLASS = {05E05 (20C08)},
  MRNUMBER = {1935570},
MRREVIEWER = {Jian-yi\ Shi},
       DOI = {10.1142/S0218196702001139},
       URL = {https://doi.org/10.1142/S0218196702001139},
}

@article {Fayers05,
    AUTHOR = {Fayers, Matthew},
     TITLE = {0-{H}ecke algebras of finite {C}oxeter groups},
   JOURNAL = {J. Pure Appl. Algebra},
  FJOURNAL = {Journal of Pure and Applied Algebra},
    VOLUME = {199},
      YEAR = {2005},
    NUMBER = {1-3},
     PAGES = {27--41},
      ISSN = {0022-4049,1873-1376},
   MRCLASS = {20C08},
  MRNUMBER = {2134290},
MRREVIEWER = {Robert\ B\'edard},
       DOI = {10.1016/j.jpaa.2004.12.001},
       URL = {https://doi.org/10.1016/j.jpaa.2004.12.001},
}

@article {JKLO22,
    AUTHOR = {Jung, Woo-Seok and Kim, Young-Hun and Lee, So-Yeon and Oh,
              Young-Tak},
     TITLE = {Weak {B}ruhat interval modules of the 0-{H}ecke algebra},
   JOURNAL = {Math. Z.},
  FJOURNAL = {Mathematische Zeitschrift},
    VOLUME = {301},
      YEAR = {2022},
    NUMBER = {4},
     PAGES = {3755--3786},
      ISSN = {0025-5874,1432-1823},
   MRCLASS = {20C08 (05E05 05E10)},
  MRNUMBER = {4449729},
MRREVIEWER = {Jianping\ Pan},
       DOI = {10.1007/s00209-022-03025-4}
}

@article {BW91,
    AUTHOR = {Bj\"orner, Anders and Wachs, Michelle L.},
     TITLE = {Permutation statistics and linear extensions of posets},
   JOURNAL = {J. Combin. Theory Ser. A},
  FJOURNAL = {Journal of Combinatorial Theory. Series A},
    VOLUME = {58},
      YEAR = {1991},
    NUMBER = {1},
     PAGES = {85--114},
      ISSN = {0097-3165,1096-0899},
   MRCLASS = {06A07 (05E99 20B35)},
  MRNUMBER = {1119703},
MRREVIEWER = {Sergey\ V.\ Fomin},
       DOI = {10.1016/0097-3165(91)90075-R}
}

@article {BW88,
    AUTHOR = {Bj\"orner, Anders and Wachs, Michelle L.},
     TITLE = {Generalized quotients in {C}oxeter groups},
   JOURNAL = {Trans. Amer. Math. Soc.},
  FJOURNAL = {Transactions of the American Mathematical Society},
    VOLUME = {308},
      YEAR = {1988},
    NUMBER = {1},
     PAGES = {1--37},
      ISSN = {0002-9947,1088-6850},
   MRCLASS = {05A99 (06F99 20B30 20F99)},
  MRNUMBER = {946427},
MRREVIEWER = {Joseph\ Kung},
       DOI = {10.2307/2000946},
       URL = {https://doi.org/10.2307/2000946},
}

@article {DS25,
    AUTHOR = {Defant, C. and Searles, D.},
     TITLE = {0-{H}ecke modules, domino tableaux, and type-${B}$ quasiysmmetric functions},
     JOURNAL = {Canad. J. Math.},
    VOLUME = {to appear, 32pp},
      YEAR = {2025},
    NUMBER = {},
     PAGES = {},
     note = {doi:10.4153/S0008414X24000762},
}

@article {Searles25,
    AUTHOR = {Dominic Searles},
     TITLE = {Diagram supermodules for $0$-{H}ecke--{C}lifford algebras},
   JOURNAL = {Math. Z.},
  FJOURNAL = {Mathematische Zeitschrift},
      YEAR = {2025},
    VOLUME = {310},
     PAGES = {Article No. 43},
       DOI = {https://doi.org/10.1007/s00209-025-03750-6},
       URL = {https://doi.org/10.1007/s00209-025-03750-6},
}

@article {CKO24,
    AUTHOR = {Choi, Seung-Il and Kim, Young-Hun and Oh, Young-Tak},
     TITLE = {Poset modules of the 0-{H}ecke algebras and related
              quasisymmetric power sum expansions},
   JOURNAL = {European J. Combin.},
  FJOURNAL = {European Journal of Combinatorics},
    VOLUME = {120},
      YEAR = {2024},
     PAGES = {Paper No. 103965, 34pp},
      ISSN = {0195-6698,1095-9971},
   MRCLASS = {20C08 (05E10)},
  MRNUMBER = {4732135},
MRREVIEWER = {Guiyu\ Yang},
       DOI = {10.1016/j.ejc.2024.103965},
       URL = {https://doi.org/10.1016/j.ejc.2024.103965},
}

@article {Stanley72,
    AUTHOR = {Stanley, Richard P.},
     TITLE = {Ordered structures and partitions},
   JOURNAL = {Mem. Amer. Math. Soc},
  FJOURNAL = {Memoirs of the American Mathematical Society},
    VOLUME = {1},
      YEAR = {1972},
    NUMBER = {119},
     PAGES = {},
}

@incollection {Gessel84,
    AUTHOR = {Gessel, Ira M.},
     TITLE = {Multipartite {$P$}-partitions and inner products of skew
              {S}chur functions},
 BOOKTITLE = {Combinatorics and algebra ({B}oulder, {C}olo., 1983)},
    SERIES = {Contemp. Math.},
    VOLUME = {34},
     PAGES = {289--317},
 PUBLISHER = {Amer. Math. Soc., Providence, RI},
      YEAR = {1984},
      ISBN = {0-8218-5029-6},
   MRCLASS = {05A17 (20C30)},
  MRNUMBER = {777705},
MRREVIEWER = {J.\ D\'esarm\'enien},
       DOI = {10.1090/conm/034/777705},
       URL = {https://doi.org/10.1090/conm/034/777705},
}

@article {GG22,
    AUTHOR = {Gaetz, Christian and Gao, Yibo},
     TITLE = {The hull metric on {C}oxeter groups},
   JOURNAL = {Comb. Theory},
  FJOURNAL = {Combinatorial Theory},
    VOLUME = {2},
      YEAR = {2022},
    NUMBER = {2},
     PAGES = {Paper No. 7, 15pp},
      ISSN = {2766-1334},
   MRCLASS = {05E16 (20F55)},
  MRNUMBER = {4449815},
}

@article {BS23,
    AUTHOR = {Bardwell, Joshua and Searles, Dominic},
     TITLE = {Weak {B}ruhat interval modules of finite-type {$0$}-{H}ecke
              algebras and projective covers},
   JOURNAL = {Electron. J. Combin.},
  FJOURNAL = {Electronic Journal of Combinatorics},
    VOLUME = {32},
      YEAR = {2025},
    NUMBER = {2},
     PAGES = {Article No. P2.54, 20pp},
      ISSN = {1077-8926},
   MRCLASS = {05E10 (05E05 20C08)},
  MRNUMBER = {4921234},
       DOI = {10.37236/13516}
}

@article{Gessel15,
    AUTHOR = {Gessel, Ira},
     TITLE = {A {H}istorical {S}urvey of ${P}$-{P}artitions},
   JOURNAL = {arXiv preprint},
      YEAR = {arXiv:1506.03508 [math.CO], 2015}
}

@article{YY24,
    AUTHOR = {Han Yang and Houyi Yu},
     TITLE = {Classification of weak {B}ruhat interval modules of $0$-{H}ecke algebras},
   JOURNAL = {arXiv preprint},
      YEAR = {arXiv:2410.07990 [math.RT], 2024}
}

@article{GR20,
    AUTHOR = {Darij Grinberg and Victor Reiner},
     TITLE = {{H}opf {A}lgebras in {C}ombinatorics},
   JOURNAL = {arXiv preprint},
      YEAR = {arXiv:1409.8356 [math.CO], 2020}
}

@article{Reiner93,
    AUTHOR = {Victor Reiner},
     TITLE = {Signed {P}osets},
   JOURNAL = {J. Combin. Theory Ser. A},
  FJOURNAL = {Journal of Combinatorial Theory. Series A},
    VOLUME = {62},
      YEAR = {1993},
    NUMBER = {2},
     PAGES = {324--360},
}

@article{Reiner92,
    AUTHOR = {Victor Reiner},
     TITLE = {Quotients of {C}oxeter complexes and ${P}$-partitions},
   JOURNAL = {Mem. Amer. Math. Soc},
  FJOURNAL = {Memoirs of the American Mathematical Society},
    VOLUME = {95},
      YEAR = {1992},
    NUMBER = {460},
     PAGES = {},
}

@article {TvW15,
    AUTHOR = {V. Tewari and S. van Willigenburg},
     TITLE = {Modules of the {$0$}-{H}ecke algebra and quasisymmetric {S}chur
              functions},
   JOURNAL = {Adv. Math.},
    VOLUME = {285},
      YEAR = {2015},
     PAGES = {1025--1065}
}

@article {TvW19,
    AUTHOR = {Tewari, V. and van Willigenburg, S.},
     TITLE = {Permuted composition tableaux, $0$-{H}ecke algebra and labeled binary trees},
   JOURNAL = {J. Combin. Theory Ser. A},
  FJOURNAL = {Journal of Combinatorial Theory Series A},
    VOLUME = {161},
      YEAR = {2019},
    NUMBER = {},
     PAGES = {420--452},
}

@article{Searles19,
    AUTHOR = {Searles, Dominic},
     TITLE = {Indecomposable {$0$}-{H}ecke modules for extended {S}chur
              functions},
   JOURNAL = {Proc. Amer. Math. Soc.},
  FJOURNAL = {Proceedings of the American Mathematical Society},
    VOLUME = {148},
      YEAR = {2020},
    NUMBER = {5},
     PAGES = {1933--1943}
}

@article {BBSSZ15,
    AUTHOR = {C. Berg and N. Bergeron and F. Saliola and
              L. Serrano and M. Zabrocki},
     TITLE = {Indecomposable modules for the dual immaculate basis of
              quasi-symmetric functions},
   JOURNAL = {Proc. Amer. Math. Soc.},
    VOLUME = {143},
      YEAR = {2015},
    NUMBER = {3},
     PAGES = {991--1000}
}

@article {KY24,
    AUTHOR = {Kim, Young-Hun and Yoo, Semin},
     TITLE = {Weak {B}ruhat interval modules of the $0$-{H}ecke algebra for genomic {S}chur functions},
   JOURNAL = {Electron. J. Combin.},
    VOLUME = {31},
      YEAR = {2024},
    NUMBER = {4},
     PAGES = {Article No. P4.73, 49pp}
}

@article{CKNO22,
    AUTHOR = {S.-I. Choi and Y.-H. Kim and S.-Y. Nam and Y.-T. Oh},
     TITLE = {The projective cover of tableau-cyclic indecomposable ${H}_n(0)$-modules},
   JOURNAL = {Trans. Amer. Math. Soc. },
    VOLUME = {375},
      YEAR = {2022},
     PAGES = {7747--7782}
}

@article{NSvWVW22,
    AUTHOR = {Niese, Elizabeth and Sundaram, Sheila and van Willigenburg, Stephanie and Vega, Julianne and Wang, Shiyun},
     TITLE = {0-{H}ecke modules for row-strict dual immaculate functions},
   JOURNAL = {Trans. Amer. Math. Soc. },
    VOLUME = {377},
      YEAR = {2024},
     PAGES = {2525--2582}
}

@article {BS22,
    AUTHOR = {Bardwell, Joshua and Searles, Dominic},
     TITLE = {0-{H}ecke modules for {Y}oung row-strict quasisymmetric
              {S}chur functions},
   JOURNAL = {European J. Combin.},
  VOLUME = {102},
      YEAR = {2022},
     PAGES = {103494, 18pp}
}

\end{document}